\documentclass[reqno,10pt]{amsart}
\usepackage{amsfonts}
\usepackage{a4wide}

\usepackage{amsmath,amssymb,amsthm,amsfonts}
\usepackage{mathrsfs}
\usepackage{bbm}
\usepackage{hyperref}

\usepackage{color}

\newtheorem{lemma}{Lemma}[section]
\newtheorem{theorem}{Theorem}[section]

\newtheorem{proposition}{Proposition}[section]

\numberwithin{equation}{section}

\newcommand{\dis}{\displaystyle}

\newcommand{\R}{\mathbb{R}}

\newcommand{\Z}{\mathbb{Z}}

\renewcommand{\S}{\mathbb{S}}

\newcommand{\FP}{\mathbf{P}}

\newcommand{\FY}{\mathbf{Y}}

\newcommand{\Fb}{\mathbf{b}}
\newcommand{\Fk}{\mathbf{k}}

\newcommand{\CA}{\mathcal{A}}

\newcommand{\CF}{\mathcal{F}}
\newcommand{\CG}{\mathcal{G}}
\newcommand{\CH}{\mathcal{H}}
\newcommand{\CI}{\mathcal{I}}
\newcommand{\CJ}{\mathcal{J}}
\newcommand{\CK}{\mathcal{K}}
\newcommand{\CL}{\mathcal{L}}
\newcommand{\CN}{\mathcal{N}}

\newcommand{\CQ}{\mathcal{Q}}

\newcommand{\na}{\nabla}

\newcommand{\al}{\alpha}

\newcommand{\ga}{\gamma}
\newcommand{\om}{\omega}

\newcommand{\la}{\lambda}
\newcommand{\de}{\delta}
\newcommand{\si}{\sigma}
\newcommand{\pa}{\partial}

\newcommand{\ta}{\theta}
\newcommand{\vth}{\vartheta}
\newcommand{\vho}{\varrho}
\newcommand{\vps}{\varepsilon}

\newcommand{\ze}{\zeta}

\newcommand{\Ga}{\Gamma}

\newcommand{\lag}{\langle}
\newcommand{\rag}{\rangle}

\usepackage{cite}

\makeatletter
\@namedef{subjclassname@2020}{%
  \textup{2020} Mathematics Subject Classification}
\makeatother

\begin{document}

\title[Uniform shear flow with hard potentials]{Uniform shear flow via the Boltzmann equation with hard potentials}

\author[R.-J. Duan]{Renjun Duan}
\address[RJD]{Department of Mathematics, The Chinese University of Hong Kong,
Shatin, Hong Kong, P.R.~China}
\email{rjduan@math.cuhk.edu.hk}

\author[S.-Q. Liu]{Shuangqian Liu}
\address[SQL]{School of Mathematics and Statistics and Hubei Key Laboratory of Mathematical Sciences, Central China Normal University, Wuhan 430079, P.R.~China}
\email{sqliu@ccnu.edu.cn}



\maketitle

\begin{abstract}
The motion of rarefied gases for uniform shear flow at the kinetic level is  governed by the spatially homogeneous Boltzmann equation with a deformation force. In the paper we study the corresponding Cauchy problem with initial data of finite mass and energy for the collision kernel in case of hard potentials $0<\gamma\leq 1$ under the cutoff assumption. We prove the global existence and large time behavior of solutions provided that the force strength $\alpha>0$ is small enough. In particular, when the initial perturbation is of order $\alpha^m$ for $m>2$, we make a rigorous justification of the uniform-in-time asymptotic expansion of solutions up to order $\alpha^2$ under a homoenergetic self-similar scaling that can capture the increase of temperature $\theta(t)\sim  (1+\ga \varrho_0\al^2 t)^{2/\gamma}$ when time tends to infinity, where $\varrho_0>0$ is a strictly positive constant depending only on the deformation force and the linearized collision operator. Specifically, we establish
$$
\theta^{3/2}(t)F(t,\theta^{1/2}(t)v)= \mu+\alpha \sqrt{\mu} G_1(t,v)+\alpha^2 \sqrt{\mu}G_2(t,v)+O(1)\alpha^m(1+\ga \varrho_0\al^2 t)^{-2}
$$
as $t\to\infty$,
where $\mu$ is a global Maxwellian and $G_1,G_2$ are microscopic bounded functions that can be explicitly determined and decay in time as $G_1\sim (1+\ga \varrho_0\al^2 t)^{-1}$ and $G_2\sim (1+\ga \varrho_0\al^2 t)^{-2}$.
\end{abstract}

%



\section{Intoduction}

\subsection{Problem}
For a rarefied gas, the uniform shear flow (USF in short) is characterized at a macroscopic level as a state where the density is constant, the velocity at $x\in \R^3$ takes the linear form of $\alpha Ax$ corresponding to a constant deformation matrix $A=(a_{ij})\in M_{3\times 3}(\R)$ with strength $\alpha>0$, and the temperature is uniform in space but may depend on time. For simplicity we always suppose ${\rm tr}A=0$. Due to the shearing motion and the associated viscous heating, the total energy and hence the temperature monotonically increase in time. It is more fundamental to understand the change of energy under the deformation force at the kinetic level (cf.~\cite{GaSa}) where the gas motion is governed by the Boltzmann equation for a finite Knudsen number
\begin{equation}
\pa_t F+v\cdot\nabla_x F=Q(F,F).\notag
\end{equation}
Here $F=F(t,x,v)\geq 0$ stands for the density distribution function of gas particles with velocity $v$ at time $t$ and position $x$ and $Q$ is the collision operator to be specified later. In this context the kinetic USF state is then defined as the one that is spatially homogeneous when the velocities of particles are referred to a Lagrangian frame moving with the velocity field $\alpha Ax$. Consequently, the density distribution function has the form
$F(t,x,v)=F(t,v-\alpha Ax)$.
With this ansatz the Boltzmann equation above becomes
\begin{align}\label{F-eq}
\pa_t F-\al\na_v\cdot(AvF)=Q(F,F),
\end{align}
for a spatially homogeneous unknown function $F=F(t,v)\geq 0$. Here, the bilinear collision operator $Q(\cdot,\cdot)$ is given as
\begin{equation}\label{Q-op}
Q(F_1,F_2)=\int_{\R^3}\int_{\S^2}B(\om,v-v_\ast)[F_1(v_\ast')F_2(v')-F_1(v_\ast)F_2(v)]\,d\omega dv_\ast.
\end{equation}
In the integral we have denoted $v'=v+[(v_\ast-v)\cdot\omega]\om$ and $v_\ast'=v_\ast-[(v_\ast-v)\cdot\omega]\om$ with $\omega\in \S^2$ in terms of the conservation laws $v_\ast+v=v_\ast'+v'$ and $|v_\ast|^2+|v|^2=|v_\ast'|^2+|v'|^2$. Throughout this paper, we define 
\begin{eqnarray}
&\dis B(\om,v-v_\ast)=|v-v_\ast|^\ga B_0(\cos\ta),\ \cos\ta=\om\cdot\frac{v-v_\ast}{|v-v_\ast|},\ \om\in\S^2,\label{hard-sp}\\
&\dis 0\leq\ga\leq1,\ 0\leq B_0(\ta)\leq C|\cos\ta|.\label{hd-po}
\end{eqnarray}
This includes general hard potentials under the Grad's angular cutoff assumption (cf.~\cite{CerBook}).

In the paper, we are interested in studying the global existence and long time behavior of solutions to the spatially homogeneous Boltzmann equation \eqref{F-eq} supplemented with initial data
\begin{align}\label{F-id}
F(0,v)=F_0(v),
\end{align}
which has finite mass and energy. Since the case of $\ga=0$ has been considered in our previous work \cite{DL-arma-2021}, we restrict our attentions in this paper to the only case of $0<\gamma\leq 1$. The problem in case of $0<\gamma\leq 1$ was addressed in  \cite{JNV-ARMA, JNV-JNS} by James, Nota and Vel\'azquez; see also a recent survey \cite{NV}. More related results will be reviewed later on.

\subsection{Normal solution}
To solve the Cauchy problem \eqref{F-eq} and \eqref{F-id}, we consider the normal solution under a certain scaling such that the profile has conservative mass, momentum and energy for all nonnegative time. For the purpose, let us now define the mass $\rho$, momentum $u$ and temperature $\ta$ associated with $F(t,v)$ as follows
\begin{align*}
   \rho=\int_{\R^3}F(t,v)\,dv,\\
   \rho u=\int_{\R^3} vF(t,v)\,dv,\\
   \rho\ta =\int_{\R^3} \frac{1}{3} |v-u|^2F(t,v)\,dv.
\end{align*}
From \eqref{F-eq}, it then follows
\begin{align}
\frac{d \rho}{dt}&=0,\label{con0}\\
\frac{d(\rho u)}{dt}+\al A\rho u&=0, \label{con1}\\
\frac{d \ta }{dt}+\frac{2\al}{3\rho}\int_{\R^3}(v-u)\cdot(A(v-u))F(t,v)\,dv&=0.\label{con2}
\end{align}
Without loss of generality we assume
\begin{equation}
\label{idF0mm}
\int_{\R^3} F_0(v)\,dv=1, \quad \int_{\R^3} vF_0(v)\,dv=0,
\end{equation}
so that \eqref{con0} and \eqref{con1} give that $\rho(t)\equiv 1$ and $u(t)\equiv 0$ for all $t\geq 0$. To further get the conservative temperature, we introduce a scaled variable $\xi=\frac{v}{\beta}$ with $\beta=\beta(t)=\sqrt{\ta(t)}$ and set
\begin{align}
F(t,v)=\beta^{-3}G(t,\xi),\quad \xi=\frac{v}{\beta},\label{F-G}
\end{align}
where $\beta(t)$ is the so-called thermal speed measuring the temperature of particles. For simplicity, we assume $\beta(0)=1$ or equivalently
\begin{equation}
\label{idF0e}
\int_{\R^3}|v|^2F_0(v)\,dv=3,
\end{equation}
so that $G_0\equiv F_0$ at initial time.
As a consequence, it holds
\begin{align}
\int_{\R^3}G(t,\xi)\,d\xi=1,\quad \int_{\R^3}\xi G(t,\xi)\,d\xi=0,\quad \int_{\R^3}|\xi|^2G(t,\xi)\,d\xi=3,\label{G-con}
\end{align}
for all $t\geq 0$.
Substituting \eqref{F-G} into \eqref{F-eq}, we obtain
\begin{align}\label{G-eq}
\pa_t G-\frac{\beta'}{\beta}\na_\xi\cdot(\xi G)-\al\na_\xi\cdot (A\xi G)=\beta^\ga Q(G,G),
\end{align}
with the initial data
\begin{align}\label{G-id}
G(0,\xi)=G_0(\xi)=F_0(\xi),
\end{align}
where we have denoted $\beta'=\frac{d\beta}{dt}$.

Note that under the above setting the time-dependent thermal speed $\beta$ is given by the solution $F(t,v)$ itself in the way that
\begin{equation}
\beta(t)=\sqrt{\int_{\R^3} \frac{1}{3} |v|^2F(t,v)\,dv}.\notag
\end{equation}
To determine $\beta(t)$ in terms of $G(t,\xi)$, it follows from the temperature equation \eqref{con2} and the scaling \eqref{F-G} with $\beta(t)=\sqrt{\theta(t)}$ that
\begin{align}
\frac{\beta'}{\beta}=-\frac{\al}{3}\int_{\R^3}\xi\cdot A\xi G(t,\xi)\,d\xi,\quad \beta(0)=1.
\label{bet-G}
\end{align}

Therefore we conclude that under conditions \eqref{idF0mm} and \eqref{idF0e} for the same initial data $F_0=G_0$, to solve the Cauchy problem \eqref{F-eq} and \eqref{F-id} is equivalent to solve the Cauchy problem \eqref{G-eq} and \eqref{G-id} coupled with the first oder ODE problem \eqref{bet-G}.  Note that for the latter we have all the physical conservation laws as in \eqref{G-con}.

\subsection{Expansion}
We solve the reformulated problem \eqref{G-eq}, \eqref{G-id} and \eqref{bet-G} in the perturbation framework for any small enough deformation strength $\alpha>0$. Note that for $\alpha=0$ meaning that there is no deformation force, the solution to the problem exists globally in time and tends asymptotically toward the global Maxwellian determined by the conservation laws. Hence, in terms of initial conditions \eqref{idF0mm} and \eqref{idF0e}, we set the reference global  Maxwellian to be
$$
\mu(\xi)=(2\pi)^{-3/2}e^{-|\xi|^2/2}.
$$
For any $\alpha>0$, we therefore define
\begin{align}\label{G-exp}
G=\mu+\sqrt{\mu}\{\al G_1+\al^2 G_2+\al^m G_R\},
\end{align}
where $G_1, G_2,G_R$ and the integer $m>2$ are to be determined later. In order for \eqref{G-con} to be satisfied, we require
\begin{align}\label{G1-G2-mi}
(f,[1,\xi,\frac{1}{2}|\xi|^2]\mu^{\frac{1}{2}})=0\quad \text{for }f=G_1,G_2,G_R.
\end{align}
Here and in the sequel, $(\cdot,\cdot)$ is used to denote the inner product of two functions on $L^2(\R^3_\xi)$ for brevity. Further plugging \eqref{G-exp} into \eqref{bet-G}  gives
\begin{align}
\frac{\beta'}{\beta}=\beta_0\al^2+\beta_1\al^3,\label{be-exp}
\end{align}
with
\begin{align}
\beta_0=\beta_0(G)=-\frac{1}{3}\int_{\R^3}\xi\cdot A\xi \sqrt{\mu}G_1\,d\xi,\label{be0-def}
\end{align}
and
\begin{align}\label{be1-def}
\beta_1=\beta_1(G)=-\frac{1}{3}\int_{\R^3}\xi\cdot A\xi \sqrt{\mu}G_2\,d\xi-\frac{\al^{m-2}}{3}\int_{\R^3}\xi\cdot A\xi \sqrt{\mu}G_R\,d\xi.
\end{align}
It should be noted that $\beta$, $\beta_0$ and $\beta_1$ all are functions of time $t$ and depend on the solution itself. For simplicity, we will omit such dependence unless the explicit expressions are important for discussions at some places.


Plugging \eqref{G-exp} into \eqref{G-eq} and comparing the coefficients of terms with different powers of $\al$, one has the equations for $G_1$ and $G_2$
\begin{align}\label{G1-eq}
-\na_\xi\cdot(A\xi\mu)\mu^{-\frac{1}{2}}+\beta^\ga L G_1=0,
\end{align}
\begin{align}\label{G2-eq}
-\beta_0\na_\xi\cdot(\xi\mu)\mu^{-\frac{1}{2}}
-\na_\xi\cdot(A\xi\sqrt{\mu}G_1)\mu^{-\frac{1}{2}}+\beta^\ga L G_2=\beta^\ga \Ga(G_1,G_1),
\end{align}
and the equation for the remainder $G_R$
\begin{align}\label{GR-eq}
&\pa_t G_R-\frac{\beta'}{\beta}\na_\xi\cdot(\xi\sqrt{\mu}G_R)\mu^{-\frac{1}{2}}
-\na_\xi\cdot(A\xi\sqrt{\mu}G_R)\mu^{-\frac{1}{2}}+\beta^\ga L G_R\notag\\
&=-\al^{1-m}\pa_tG_1-\al^{2-m}\pa_t G_2+\al^{3-m}\beta_1 \na_\xi\cdot(\xi\mu)\mu^{-\frac{1}{2}}
+\al^{1-m}\frac{\beta'}{\beta}\na_\xi\cdot(\xi\sqrt{\mu}G_1)\mu^{-\frac{1}{2}}
\notag\\&+\al^{2-m}\frac{\beta'}{\beta}\na_\xi\cdot(\xi\sqrt{\mu}G_2)\mu^{-\frac{1}{2}}
+\al^{3-m}\na_\xi\cdot(A\xi\sqrt{\mu}G_2)\mu^{-\frac{1}{2}}
+\al^{3-m}\beta^\ga \{\Ga(G_1,G_2)+\Ga(G_2,G_1)\}\notag\\&+\al^{4-m}\beta^\ga\Ga(G_2,G_2)
+\al\beta^\ga\{\Ga(G_1+\al G_2,G_R)+\Ga(G_R,G_1+\al G_2)\}+\al^m\beta^\ga\Ga(G_R,G_R),
\end{align}
with
\begin{align}
\sqrt{\mu}G_R(0,\xi)=G_{R,0}=\al^{-m}\{G_0-\mu-\al\sqrt{\mu} G_1(0,\xi)-\al^2\sqrt{\mu} G_2(0,\xi)\}.\label{GR-id}
\end{align}
Here, $L$ is the linearized Boltzmann operator and $\Ga$ is the nonlinear collision Boltzmann operator,  respectively given by
\begin{align}
Lg=-\mu^{-1/2}\left\{Q(\mu,\sqrt{\mu}g)+Q(\sqrt{\mu}g,\mu)\right\},\notag
\end{align}
and
\begin{align}
\Gamma (f,g)=\mu^{-1/2} Q(\sqrt{\mu}f,\sqrt{\mu}g).\notag
\end{align}

For later use we first introduce some notations. Note that
$$
Lf=\nu f-Kf
$$
with
\begin{align}\label{sp.L}
\left\{\begin{array}{rll}
&\nu=\dis{\int_{\R^3}\int_{\S^2}}B(\om,\xi-\xi_\ast)\mu(v_\ast)\, d\om d\xi_\ast\sim (1+|\xi|)^{\ga},\\[2mm]
&Kf=\mu^{-\frac{1}{2}}\left\{Q(\mu^{\frac{1}{2}}f,\mu)+Q_{\textrm{gain}}(\mu,\mu^{\frac{1}{2}}f)\right\},
\end{array}\right.
\end{align}
where $Q_{\textrm{gain}}$ denotes the positive part of $Q$ in \eqref{Q-op}. Moreover, it holds that
\begin{align}
Kf=\int_{\R^3}\Fk(\xi,\xi_\ast)f(\xi_\ast)\,d\xi_\ast=\int_{\R^3}(\Fk_2-\Fk_1)(\xi,\xi_\ast)f(\xi_\ast)\,d\xi_\ast,\notag
\end{align}
with
\begin{equation}
0\leq \Fk_1(\xi,\xi_\ast)\leq \tilde{c}_1|\xi-\xi_\ast|^\ga e^{-\frac{1}{4}(|\xi|^2+|\xi_\ast|^2)},\ 0\leq \Fk_2(\xi,\xi_\ast)\leq \tilde{c}_2|\xi-\xi_\ast|^{-2+\ga}e^{-
\frac{1}{8}|\xi-\xi_\ast|^{2}-\frac{1}{8}\frac{\left||\xi|^{2}-|\xi_\ast|^{2}\right|^{2}}{|\xi-\xi_\ast|^{2}}},\notag
\end{equation}
where both $\tilde{c}_1$ and $\tilde{c}_2$ are positive constants. The kernel of $L$, denoted as $\ker L$, is a five-dimensional subspace of $L^2(\R^3_\xi)$, spanned by
$$
\{1,\xi,|\xi|^2-3\}\sqrt{\mu}:= \{\phi_i\}_{i=1}^5.
$$
We further define the projection from $L^2$ to $\ker L$ by
\begin{align*}
\FP_0 g=\left\{a_g+\Fb_g\cdot \xi+(|\xi|^2-3)c_g\right\}\sqrt{\mu}
\end{align*}
for $g\in L^2$, and correspondingly denote the operator $\FP_1$ by $\FP_1g=g-\FP_0 g$, which is orthogonal to $\FP_0$ in $L^2$.
Traditionally, $\FP_0 g$ is also called the macroscopic component, while $\FP_1 g$ stands for the microscopic component. As in \cite{DL-arma-2021}, to treat the polynomial tail part of $G_R$, it is also convenient to define
\begin{align}\notag
\CL f=-\left\{Q(f,\mu)+Q(\mu,f)\right\}=\nu f-\CK f,
\end{align}
with
\begin{equation}\label{sp.cL}
\CK f=Q(f,\mu)+Q_{\textrm{gain}}(\mu,f)=\sqrt{\mu}K(\frac{f}{\sqrt{\mu}}).
\end{equation}

We now determine $G_1$ and $G_2$. Notice that one has $\xi\cdot A\xi\mu^{\frac{1}{2}}\in (\ker L)^\perp$ due to ${\rm tr}A=0$. We then get from \eqref{G1-eq} and \eqref{G1-G2-mi} that
\begin{align}
G_1=-\beta^{-\ga}L^{-1}\{\xi\cdot A\xi\mu^{\frac{1}{2}}\}.\label{G1-exp}
\end{align}
As a consequence, \eqref{be0-def} gives
\begin{align}\label{bet0-eq}
\beta_0=-\frac{1}{3}( \xi\cdot A\xi\mu^{\frac{1}{2}},G_1)=\frac{1}{3}\beta^{-\ga} \left( \xi\cdot A\xi\mu^{\frac{1}{2}},L^{-1}\{\xi\cdot A\xi\mu^{\frac{1}{2}}\}\right):= \beta^{-\ga}\varrho_0,
\end{align}
with
\begin{equation}
\label{def.rho0}
\varrho_0=\frac{1}{3}\left( \xi\cdot A\xi\mu^{\frac{1}{2}},L^{-1}\{\xi\cdot A\xi\mu^{\frac{1}{2}}\}\right).
\end{equation}
Note that $\varrho_0>0$ is a constant independent of $\al$ and $t$. Moreover, \eqref{bet0-eq} also implies
\begin{align}
\left( -\beta_0\na_\xi\cdot(\xi\mu)\mu^{-\frac{1}{2}}
-\na_\xi\cdot(A\xi\sqrt{\mu}G_1)\mu^{-\frac{1}{2}},[1,\xi,\frac{1}{2}|\xi|^2]\mu^{\frac{1}{2}}\right)=0,\notag
\end{align}
and note that $G_2$ can be only microscopic as required in the condition \eqref{G1-G2-mi}. Then, it is valid to derive from \eqref{G2-eq} that
\begin{align}
G_2=&\beta^{-\ga}L^{-1}\left\{\beta^\ga\Ga(G_1,G_1)+\beta_0\na_\xi\cdot(\xi\mu)\mu^{-\frac{1}{2}}
+\na_\xi\cdot(A\xi\sqrt{\mu}G_1)\mu^{-\frac{1}{2}}\right\}.\label{G2-exp}
\end{align}
Using the expression of $G_2$ above, we define
\begin{align}
\vho_1=&-\frac{\beta^{2\ga}}{3}\int_{\R^3}\xi\cdot A\xi \sqrt{\mu}G_2d\xi\notag\\
=&-\frac{1}{3}\int_{\R^3}\xi\cdot A\xi \sqrt{\mu}\bigg\{L^{-1}\Big\{\Ga(L^{-1}\{\xi\cdot A\xi\mu^{\frac{1}{2}}\},L^{-1}\{\xi\cdot A\xi\mu^{\frac{1}{2}}\})+\vho_0\na_\xi\cdot(\xi\mu)\mu^{-\frac{1}{2}}
\notag\\&\qquad\qquad-\na_\xi\cdot(A\xi\sqrt{\mu}L^{-1}\{\xi\cdot A\xi\mu^{\frac{1}{2}}\})\mu^{-\frac{1}{2}}\Big\}\bigg\}d\xi. \label{vh01}
\end{align}
Note again that $\vho_1$ is a constant independent of $\al$ and $t$.
Furthermore, by \eqref{G2-exp} and \eqref{G1-exp}  as well as the definition \eqref{be1-def}, one gets from \eqref{GR-eq} that
\begin{align}\label{GR-con}
\left( G_R,[1,\xi,\frac{1}{2}|\xi|^2]\mu^{\frac{1}{2}}\right)=\left( G_{R,0},[1,\xi,\frac{1}{2}|\xi|^2]\right),
\end{align}
which coincides with our assumption \eqref{G-con} when further assuming
\begin{align}\label{GR-con-id}
\left( G_{R,0},[1,\xi,\frac{1}{2}|\xi|^2]\right)=0.
\end{align}

Let us now briefly illustrate how to solve $\beta$ from \eqref{be-exp}.  By \eqref{bet0-eq}, \eqref{vh01} and \eqref{be-exp}, we rewrite
\begin{align}
\frac{\beta'}{\beta}=\vho_0\al^2\beta^{-\ga}+\vho_1\al^3\beta^{-2\ga}-\frac{\al^{m+1}}{3}\int_{\R^3}\xi\cdot A\xi \sqrt{\mu}G_R\,d\xi.
\label{eq.betaR}
\end{align}
It turns out that both constants $\varrho_0>0$ and $\varrho_1$ can be proved to be finite. Furthermore $\beta^{2\ga}\sqrt{\mu}G_R$ can be verified to be bounded in terms of $G_1$ and $G_2$. Therefore, if $\al>0$ is suitably small and one takes $2<m<3$, we formally have
\begin{align}
\frac{\beta'}{\beta}\sim \varrho_0\beta^{-\ga}\al^2,\notag
\end{align}
which may give
\begin{align}
\beta\sim(1+\ga \varrho_0\al^2 t)^{\frac{1}{\ga}}.\notag
\end{align}

\subsection{Main result}

Define a polynomial velocity weighted function $w^\ell(\xi)=(1+|\xi|^2)^{\ell}$ for  $\ell\geq0$.
We now state the main result below for the Cauchy problem \eqref{F-eq} and \eqref{F-id}.


\begin{theorem}\label{mth}
Assume \eqref{hard-sp} and \eqref{hd-po}. Let $0<\ga\leq 1$, $2<m<3$ and an integer $N\geq 1$. Let $\ell_\infty\gg 4$ be a constant that can be arbitrarily large. Suppose ${\rm tr}A=0$. There are constants $\al_0>0$, $M_0>0$ and $C>0$ such that for any $\al\in(0,\al_0)$, if initial data $F_0=F_0(v)\geq 0$ satisfies
\begin{align}
\int_{\R^3} F_0(v)\,dv=1, \quad \int_{\R^3} vF_0(v)\,dv=0,\quad\int_{\R^3}|v|^2F_0(v)\,dv=3,
 \notag
\end{align}
and
\begin{align}\notag
\sum\limits_{|\vth|\leq N}\|w^{\ell_\infty}\pa_v^\vth [F_0-\mu-\al \sqrt{\mu}G_1(0,v)-\al^2 \sqrt{\mu}G_2(0,v)]\|_{L^\infty}\leq M_0\al^{m},
\end{align}
then the spatially homogeneous Cauchy problem \eqref{F-eq} and \eqref{F-id} admits a unique global  solution $F=F(t,v)\geq0$ satisfying
\begin{align}
\int_{\R^3} F(t,v)\,dv=1, \quad \int_{\R^3} vF(t, v)\,dv=0,\quad \int_{\R^3}|v|^2F(t,v)\,dv:=3\beta^2(t),\quad\forall\,t\geq 0,
 \notag
\end{align}
with estimates
\begin{equation}
\beta(t)>0,\quad \beta(0)=1, \quad \lim\limits_{t\to\infty}\frac{\beta(t)}{(1+\ga \varrho_0\al^2 t)^{\frac{1}{\ga}}}=1,
\label{sol.beta}
\end{equation}
and
\begin{multline}\label{sol-decay}
\sum\limits_{|\vth|\leq N}\left\|w^{\ell_\infty}\pa_v^\vth \left[\beta^3(t)F(t,\beta(t)v)-\mu(v)-\al \sqrt{\mu}G_1(t,v)-\al^2 \sqrt{\mu}G_2(t,v)\right]\right\|_{L^\infty}\\
\leq C(1+\ga \vho_0\al^2 t)^{-2}
(M_0\al^m+\al^3), \quad \forall\,t\geq 0.
\end{multline}
Here $\vho_0>0$, $G_1(t,\cdot)$ and $G_2(t,\cdot)$ are respectively given by \eqref{def.rho0}, \eqref{G1-exp} and \eqref{G2-exp}.
\end{theorem}


Theorem \ref{mth} above shows the following uniform asymptotic expansion of the obtained global solution up to $\alpha^2$ in the homoenergetic self-similar scaling
\begin{equation}
\label{expcon}
\theta^{3/2}(t)F(t,\theta^{1/2}(t)v)=\mu+\alpha \sqrt{\mu} G_1(t,v)+\alpha^2 \sqrt{\mu}G_2(t,v)+O(1)\alpha^m (1+\ga \varrho_0\al^2 t)^{-2},
\end{equation}
where the thermal speed $\theta(t)=\beta^2(t)$ satisfies
\begin{equation}
\label{tincr}
\theta(t)\sim(1+\ga \varrho_0\al^2 t)^{\frac{2}{\ga}}.
\end{equation}
By \eqref{G1-exp} and \eqref{G2-exp} one has 
$$
G_1(t,v)\sim (1+\ga \varrho_0\al^2 t)^{-1}\text{ and } G_2(t,v)\sim (1+\ga \varrho_0\al^2 t)^{-2}
$$ 
in large time. We remark that it is also interesting to carry out the higher order expansion of $G$, for instance, up to the $n$th-order for an integer $n>0$, namely,
\begin{align}
G=\mu+\sqrt{\mu}\{\al G_1+\al^2 G_2+\cdots+\al^n G_n+\al^m G_R\}, \ m>n.\notag
\end{align}
Then the corresponding decay rate in \eqref{sol-decay} could  be improved to be $(1+\ga \vho_0\al^2 t)^{-n}$. Furthermore, if one can obtain uniform estimates for any $n>\frac{1}{\ga}$ as $\ga\rightarrow 0+$, then the decay rate of the remainder $G_R$ as $\ga\rightarrow 0+$ should be recovered by taking the limit
$$
\lim\limits_{\ga\rightarrow 0+}(1+\ga \vho_0\al^2 t)^{-\frac{1}{\ga}}=e^{-\vho_{00}\al^2t},
$$
which is exponential in time with size of $\al^2$ order, where $\vho_{00}>0$ is the limit of $\vho_0$ as $\ga\rightarrow 0+$ in terms of \eqref{def.rho0}. This will coincide with the result proved in our previous work \cite{DL-arma-2021} for the case of $\ga=0$. The rigorous study of such issue is left for the future.


\subsection{Literature}
In what follows we mention some related literature. Readers may refer to our previous work \cite{DL-arma-2021, DL-2022, DLY-2021} for a detailed review. For instance, Galkin \cite{G1} and Truesdell \cite{T}  first independently introduced the concept of homoenergetic solutions to the Boltzmann equation; see also an introduction to the topic in Truesdell-Muncaster \cite{TM}. The numerical investigation has been extensively made in the monograph Garz\'o-Santos \cite{GaSa}. Related to the USF state without any boundary, a physically more realistic topic is the planar Couette flow governed by the Boltzmann equation for a rarefied gas between two parallel infinite plates moving with opposite velocities and such topic was also discussed in many books on kinetic theory such as Kogan \cite[Chapter 4]{Ko} and Sone \cite[Chapter 4]{Sone07}; see also \cite[Chapter 5]{GaSa}.

When there is no deformation force, the topic on the self-similar solution and its asymptotic stability for the spatially homogeneous Boltzmann equation was investigated in early 2000s by Bobylev-Cercignani  \cite{BC02b, BC02a, BC03} and later by Cannone-Karch \cite{CK10, CK13} and Morimoto-Yang-Zhao \cite{MYZ} among many others.

When there is a deformation force, we mention many early results on the shear flow topic by Cercignani \cite{Cer89, Cer00, Cer02} and Bobylev-Caraffini-Spiga \cite{BCS}. Recently, the significant progress was made by  James-Nota-Vel\'azquez \cite{JNV-ARMA, JNV-JNS, JNV19} and later by Bobylev-Nota-Vel\'azquez \cite{BNV-2019}. In particular, for the USF governed by the Boltzmann equation in case of the Maxwell molecule model $\gamma=0$ under the cutoff assumption, \cite{BNV-2019} proved the existence (obtained also in \cite{JNV-ARMA}) and the uniqueness, non-negativity and stability (as well as the analysis of the moments and the exponential rate of convergence) of self-similar profiles in the class of measures for small enough deformation strength; see also Bobylev \cite{Bo21} for a further study to provide explicit estimates on smallness of the deformation matrix. Here, the approach used in \cite{JNV-ARMA} is based on the fixed point argument on the integral form of the problem over a set of non-negative Radon measures, while \cite{BNV-2019} gave a different proof by means of the Fourier transform method (cf.~\cite{Bo75, Bo88}) taking the full advantage of the Bobylev formula. An interesting result on self-similar profiles for the non-cutoff Maxwell molecule model was also obtained by Kepka \cite{Ke1}. Readers may refer to \cite{NV} for a thorough review to those and other related works. Moreover, following \cite{JNV-ARMA} and \cite{BNV-2019}, in the case of Maxwell molecule with cutoff, we also constructed in  \cite{DL-arma-2021}  smooth self-similar profiles for the shear flow problem on the Boltzmann equation and proved the dynamical stability of the stationary solution via a perturbation approach.

In the current work, we are interested in the uniform shear flow governed by the Boltzmann equation in the case of hard potentials. The problem was addressed in  \cite{JNV-ARMA, JNV-JNS}; see also \cite{NV} as mentioned before. In fact, the formal Hilbert expansion similar to  \eqref{G-exp} as given in \cite{JNV-JNS} implies that the temperature of gas particles increases in time with an algebraic rate in \eqref{tincr} and the self-similar asymptotics of the form \eqref{expcon} was conjectured. In particular, since $G_1$ and $G_2$ decay in time, the solution converges self-similarly in large time to the global Maxwellian in contrast to a non-equilibrium state with a polynomial large velocity tail obtained in case of the Maxwell molecule model (cf.~\cite{MT}). To treat the issue we recently considered a closely related problem in \cite{DL-2022} where an extra thermostated term is added to compensate the viscous heating energy such that the system of gas particles can be driven in large time to the non-equilibrium steady state under the interplay of both thermostated and sheared forces. Using the developed techniques in \cite{DL-arma-2021,DL-2022}, we aim in this paper at making a rigorous justification of the expansion \eqref{expcon} with the temperature behavior \eqref{tincr}. We remark that during the preparation of the current work, we have been aware of a preprint \cite{Ke2} for treating a similar issue which includes both cutoff and non-cutoff cases. The approach used in \cite{Ke2} is based on the construction of polynomial tail solutions with the help of the robust semigroup property, cf.~\cite{GMM}.

In the end, we remark that there have been extensive studies for stability of shear flow in the context of fluid dynamic equations (cf.~\cite{ScHe}). In particular, we mention major contributions \cite{BM-15,BMV-16,BGM-17} recently made by Bedrossian together with his collaborators; see the survey \cite{BGM-BAMS} for the subject. Regarding the shear flow with physical boundaries, we refer to recent progress by Ionescu-Jia \cite{IJ} and Masmoudi-Zhao \cite{MZ} which inspired us to study in \cite{DLY-2021} the kinetic planar Couette flow with boundaries as mentioned before. We point that it would be interesting to understand the relation of those fluid solutions and Boltzmann solutions through the rigorous justification of the hydrodynamic limit in case of small shear strength, cf.~\cite{ELM-94, ELM-95}.

\subsection{Strategies and ideas of the proof}
We now outline some key ideas and methods used in the paper. One of typical features for the shear flow governed by the Boltzmann equation is the rapid increasing of the total energy of gas particles.
A serious consequence of such a scenario is that the macroscopic component is out of control in $L^2$ setting.
Therefore the self-similar structure of solutions need to be explored to look for the normal form. To do so, a suitable scaling should be introduced so that the energy of the self-similar profile can be conserved. In this paper, the scaling  parameter $\beta$ of self-similar solutions is chosen to satisfy the ODE
\begin{align}
\frac{\beta'}{\beta}=-\frac{\al}{3}\int_{\R^3}\xi\cdot A\xi G\,d\xi. \notag
\end{align}
In particular, $\beta=\beta(t)$ is an unknown function of time which depends on the solution itself. This leads to a nonlinear convection term $\frac{\beta'}{\beta}\na_\xi\cdot(\xi G)$ in the scaled equation. However, if the drift terms in the equation are of higher order when $t\rightarrow\infty$ compared with the Boltzmann collision operator, the solution should converge to Maxwellian equilibrium. Hence the following Hilbert type expansion is introduced
\begin{align}
G=\mu+\sqrt{\mu}\{\al G_1+\al^2 G_2+\al^m G_R\},\notag
\end{align}
with $G_1,G_2$ and $G_R$ all belonging to $(\ker L)^\bot $, provided that the shear strength $\al>0$ is small enough. Unfortunately, unlike the case of Maxwell molecular, not only the remainder $G_R$ but also both the correction terms $G_1$ and $G_2$ depend on $\beta$. To overcome this difficulty, a continua argument is employed. More precisely, we first construct the local existence by the contraction mapping method, and then prove the {\it a priori} estimates in $L^\infty$ setting. To show the local existence, there are two difficulties:  one is to determine the time-dependent scaling function $\beta$ and the other is to justify that the solution operator is contractive in a short time.
To overcome the first difficulty, an expansion in the form of
\begin{align}
\frac{\beta'}{\beta}=\beta_0\al^2+\beta_1\al^3,\notag
\end{align}
is considered,
and then the problem is reduced to solve the ordinary differential inequalities
\begin{align}
\frac{1}{2}\vho_0\beta^{-\ga}\al^2\leq \frac{\beta'}{\beta}\leq \frac{3}{2}\vho_0\beta^{-\ga}\al^2.\notag
\end{align}
To treat the second difficulty, the stability of $\beta$ with respect to the solution variable $h$ is proved, namely we verify
\begin{align}
|(\beta^\ga(h)-\beta^\ga(\bar{h}))(t)\leq C(T_0)\al \|w^\ell[h_1-\bar{h}_1,h_2-\bar{h}_2](t)\|_{L^\infty}.\notag
\end{align}
Due to this, $V(h)$, the velocity determined by the  characteristic line as in \eqref{CL}, and $\CA(h)$, the generator of the semigroup as given by \eqref{A-def}, both can be shown to be stable.

Another typical feature for the shear flow problem on the Boltzmann equation is the velocity growth caused by the shear force (or deformation force).
To overcome this difficulty, we introduce the following Caflisch's decomposition
\begin{align}
\pa_t G_{R,1}&-\frac{\beta'}{\beta}\na_\xi\cdot(\xi G_{R,1})
-\al\na_\xi\cdot(A\xi G_{R,1})+\beta^\ga \nu G_{R,1}\notag\\
=&\beta^\ga\chi_M \CK G_{R,1}-\frac{\beta'}{2\beta} |\xi|^2\sqrt{\mu}G_{R,2}-\frac{\al}{2} \xi\cdot(A\xi)\sqrt{\mu}G_{R,2}
+\cdots,\notag
\end{align}
and
\begin{align}
\pa_t G_{R,2}&-\frac{\beta'}{\beta}\na_\xi\cdot(\xi G_{R,2})
-\al\na_\xi\cdot(A\xi G_{R,2})+\beta^\ga L G_{R,2}
=\beta^\ga(1-\chi_M)\mu^{-\frac{1}{2}} \CK G_{R,1}.\notag
\end{align}
The disadvantage of such a decomposition is that the large velocity behavior of the integration operator $\CK$ is hard to be obtained. We have already settled this problem in $L^\infty$ setting in our previous work \cite{DL-arma-2021} and \cite{DL-2022} in the case of $\ga=0$ and $\ga\in(0,1]$, respectively. Here the new difficulty stems from the time growth $\beta^\ga$ in $\beta^\ga(1-\chi_M)\mu^{-\frac{1}{2}} \CK G_{R,1}$, which is caused by the hard potential kernel under the self-similar scaling \eqref{F-G}.
To overcome this difficulty, an $L^2$ estimate on the first component is developed, in particular, a new $L^2$ estimates on the collision operator $\CK$ is proved with the aid of Riesz-Thorin's interpolation inequality, cf. Lemma \ref{RT-lem}.

\subsection{Notations}
We list some notations and norms used in the paper. Throughout this paper,  $C$ denotes some generic positive (generally large) constant and $\la$ denote some generic positive (generally small) constant. $D\lesssim E$ means that  there is a generic constant $C>0$
such that $D\leq CE$. $D\sim E$
means $D\lesssim E$ and $E\lesssim D$.
For multi-indices
$\vth=[\vth_1, \vth_2, \vth_3]$, we denote
$
\partial^{\vth}_{\xi}=\partial_{\xi_{1}}^{\vth_{1}}
\partial_{\xi_{2}}^{\vth_{2}}\partial_{\xi_{3}}^{\vth_{3}}
$
and likewise for $
\partial^{\vth}_{v}$, and the length of $\vth$ is denoted by $|\vth|=\vth_1+\vth_2+\vth_3$.
$\vth'\leq\vth$ means that no component of $\vth'$
is greater than the component of $\vth$, and $\vth'<\vth$ means that
$\vth'\leq\vth$ and $|\vth'|<|\vth|$. $(\cdot,\cdot)$
denotes the $L^{2}$ inner product in ${\R}^{3}_{\xi}$, with the
$L^{2}$ norm $\|\cdot\|$.

\subsection{Organization of the paper}
The rest of this paper is arranged as follows. Section \ref{sec-k} is devoted to obtaining a crucial $L^2$ estimate for the
integration operator $\CK$ given by \eqref{sp.cL}. The local existence of Cauchy problem \eqref{GR-eq} and \eqref{GR-id} is constructed in Section \ref{sec.le}. The proof of Theorem \ref{mth} is given in Section \ref{sec.rm}. Finally, some basic estimates are collected in Appendix \ref{pre-sec}.

\section{$L^2$ estimate for $\CK$ with large velocity}\label{sec-k}
In this section we consider an $L^2$ estimate for $\CK$ defined by \eqref{sp.cL}. The proof is based on the application of Riesz-Thorin's interpolation inequality. For completeness, we first quote the following lemma, 
cf.~\cite[Theorem 1.3.4, pp.37]{LG-fourier}.

\begin{lemma}\label{RT-lem}
Let $(X_1,\CF_1,\mathfrak{M}_1)$ and $(X_2,\CF_2,\mathfrak{M}_2)$ be $\si$ finite measure spaces. Fix $1\leq p_1,p_2,q_1,q_2\leq\infty$ and
$0 <\la<1$. Define $p_\la$ and $q_\la$ by
$$
p_\la=\frac{\la}{p_1}+\frac{1-\la}{p_2},\ q_\la=\frac{\la}{q_1}+\frac{1-\la}{q_2}.
$$
Let both $T_1$ and $T_2$ be continuous linear mappings such that
$$
T_1: L^{p_1}(\mathfrak{M}_1)\rightarrow L^{q_1}(\mathfrak{M}_2),\ \textrm{with the operator norm}\ \|T_1\| = M_1,
$$
and
$$
T_2: L^{p_2}(\mathfrak{M}_1)\rightarrow L^{q_2}(\mathfrak{M}_2),\ \textrm{with the operator norm}\ \|T_2\| = M_2,$$
and
$T_1f = T_2f$ for any $f\in L^{p_1}(\mathfrak{M}_1)\cap L^{p_2}(\mathfrak{M}_1).$ Here and below,
we shorten $L^r (X,\CF,\mathfrak{M})$
to $L^r (\mathfrak{M})$.

Thus we can define a mapping
$$
T :L^{p_1}(\mathfrak{M}_1)\cap L^{p_2}(\mathfrak{M}_1)\rightarrow L^{q_1}(\mathfrak{M}_2)\cap L^{q_2}(\mathfrak{M}_2)
$$
by
$$Tf = T_1f = T_2f,\ \textrm{for}\ f\in L^{p_1}(\mathfrak{M}_1)\cap L^{p_2}(\mathfrak{M}_1).$$
For each $f\in L^{p_1}(\mathfrak{M}_1)\cap L^{p_2}(\mathfrak{M}_1),$ it holds
\begin{align}\label{RT-it}
\|Tf\|_{L^{q_\la}(\mathfrak{M}_2)}\leq M_1^{\la}M_2^{1-\la}\|f\|_{L^{p_\la}(\mathfrak{M}_1)}.
\end{align}
Furthermore, $T$ has a unique continuous linear extension
$$
T_\la: L^{p_\la}(\mathfrak{M}_1)\rightarrow L^{q_\la}(\mathfrak{M}_2)
$$
with
$$
\|T_\la\|\leq M_1^{\la}M_2^{1-\la}.
$$
\end{lemma}
With this lemma in hand, we now intend to prove the following $L^2$ estimate which plays an crucial role in the proof of global existence in Section \ref{sec.rm}. The estimate gives the smallness of $\CK$ at large velocities. We also consider the corresponding velocity derivative estimates.

\begin{proposition}\label{CK-l2-pro}
Let $0<\ga\leq1$, then there is a constant $C>0$ such that for suitably large $\ell_2>0$, there are sufficiently large $M=M(\ell_2)>0$ and suitably small $\varsigma=\varsigma(\ell_2)>0$ such that for any $\vth\geq0$ it holds that
\begin{align}\label{CK-l2}
\|\nu^{-1/2} w^{\ell_2}\pa_\xi^\vth(\chi_M\CK f)\|\leq C\{(1+M)^{-\ga}+\varsigma\}^{1/2}(1+M)^{-\ga/2}\sum\limits_{\vth'\leq\vth}\|\nu^{1/2}w^{\ell_2}\pa_\xi^{\vth'} f\|,
\end{align}
where $\chi_{M}(\xi)$ is a non-negative smooth cutoff function such that
\begin{align}
\chi_{M}(\xi)=\left\{\begin{array}{rll}
1,&\ |\xi|\geq M+1,\\[2mm]
0,&\ |\xi|\leq M.
\end{array}\right.\notag
\end{align}

\end{proposition}

\begin{proof}
We first consider the case that $\vth=0$.
From Lemma \ref{g-ck-lem}, it follows
\begin{align}\label{CK-ifif}
\|\chi_M\nu^{-1} w^{2\ell_2}\CK f\|_{L^\infty}\leq C_1\{(1+M)^{-\ga}+\varsigma\}\|w^{2\ell_2}f\|_{L^\infty},
\end{align}
where $C_1>0$ and independent of $M.$ Recalling Lemma \ref{RT-lem}, to prove \eqref{CK-l2} with $\vth=0$, it suffices to prove the following
\begin{align}\label{CK-l1}
\|\nu^{-1} w^{2\ell_2}\chi_M\CK f\|_{L^1}\leq C_2(1+M)^{-\ga}\|\nu w^{2\ell_2}f\|_{L^1},
\end{align}
for some $C_2>0$ and independent of $M.$ In view of \eqref{sp.cL}, we have
\begin{align}
\int_{\R^3}\chi_M\nu^{-1} w^{2\ell_2}|\CK f|d\xi\leq&
\int_{\R^3}\chi_M\nu^{-1} w^{2\ell_2}\int_{\R^3\times\S^2}B_0|\xi_\ast-\xi|^\ga \mu(\xi'_\ast)|f(\xi')|d\xi_\ast d\om d\xi
\notag\\&+\int_{\R^3}\chi_M\nu^{-1} w^{2\ell_2}\int_{\R^3\times\S^2}B_0|\xi_\ast-\xi|^\ga \mu(\xi')|f(\xi'_\ast)|d\xi_\ast d\om d\xi
\notag\\&+\int_{\R^3}\chi_M\nu^{-1} w^{2\ell_2}\int_{\R^3\times\S^2}B_0|\xi_\ast-\xi|^\ga \mu(\xi)|f(\xi_\ast)|d\xi_\ast d\om d\xi
\notag\\&=:\CI_1+\CI_2+\CI_3,\notag
\end{align}
where $\CI_i$ $(1\leq i \leq3)$ denote those terms on the right respectively. We now compute them individually.
First of all, by choosing $M>0$ sufficiently large, one sees that there exists a constant $C_0$ which is independent of $\ell_2$ such that
\begin{align}\label{ew-re}
\chi_M\nu w^{2\ell_2}\mu^{\frac{1}{2\ell_2}}\leq C_0.
\end{align}
For $\CI_3$, by \eqref{ew-re}, we then have
\begin{align}
|\CI_3|\leq CC_0\int_{\R^3}\chi_M\mu^{1/2}(\xi)d\xi\int_{\R^3}\lag\xi_\ast\rag^\ga |(w^{-2\ell_2}w^{2\ell_2}f)(\xi_\ast)|d\xi_\ast
\leq \frac{C}{(1+M)^\ga}\|\nu w^{2\ell_2}f\|_{L^1}.\notag
\end{align}
For $\CI_1$, if $|\xi_\ast'|<\frac{1}{\sqrt{\ell_2}}|\xi|$, on the one hand, it follows $|\xi'|^2\geq|\xi|^2-|\xi_\ast'|^2\geq (1-\ell_2^{-1})|\xi|^2,$
which gives
$$
\frac{w^{2\ell_2}(\xi)}{w^{2\ell_2}(\xi')}\leq \left(\frac{1}{1-\ell_2^{-1}}\right)^{2\ell_2}\leq e^2.
$$
On the other hand, it follows
$$
|\xi_\ast-\xi|^\ga\leq \lag\xi_\ast'\rag^\ga\lag\xi'\rag^\ga.
$$
Thus, one has
\begin{align}
|\CI_1|\leq&C\int_{\R^3}\int_{\R^3\times\S^2}B_0\chi_M\nu^{-1}(\xi)\frac{w^{2\ell_2}(\xi)}{w^{2\ell_2}(\xi')}\lag\xi'\rag^\ga
\lag\xi_\ast'\rag^\ga \mu(\xi'_\ast)w^{2\ell_2}(\xi')|f(\xi')|d\om d\xi_\ast d\xi\notag\\
\leq&\frac{C}{(1+M)^\ga}\int_{\R^3}\int_{\R^3\times\S^2}B_0\mu(\xi'_\ast)\lag\xi'\rag^\ga
\lag\xi_\ast'\rag^\ga w^{2\ell_2}(\xi')|f(\xi')|d\xi_\ast d\om d\xi
\leq \frac{C}{(1+M)^\ga}\|\nu w^{2\ell_2}f\|_{L^1},\notag
\end{align}
where a change of variables $(\xi,\xi_\ast)\rightarrow(\xi',\xi'_\ast)$ has been used.

If $|\xi_\ast'|\geq \frac{1}{\sqrt{\ell_2}}|\xi|$, one has by applying \eqref{ew-re} and a change of variables $(\xi,\xi_\ast)\rightarrow(\xi',\xi'_\ast)$ that
\begin{align}\label{I1-es}
|\CI_1|\leq&C\int_{\R^3}\chi_M\nu^{-1}(w^{2\ell_2}\mu^{\frac{1}{2\ell_2}})(\xi)d\xi\int_{\R^3\times\S^2}\lag\xi'\rag^\ga \lag\xi'_\ast\rag^\ga \mu^{\frac{1}{2}}(\xi'_\ast)|f(\xi')|d\xi_\ast d\om\notag\\
\leq& \frac{C}{(1+M)^\ga}\int_{\R^3\times\R^3}\lag\xi\rag^\ga \lag\xi_\ast\rag^\ga \mu^{\frac{1}{2}}(\xi_\ast)|f(\xi)|d\xi_\ast d\xi
\leq \frac{C}{(1+M)^\ga}\|\nu w^{2\ell_2}f\|_{L^1}.
\end{align}

Similarly, for $\CI_2$, if $|\xi'|<\frac{1}{\sqrt{\ell_2}}|\xi|$, one gets
$|\xi_\ast'|^2\geq|\xi|^2-|\xi'|^2\geq (1-\ell_2^{-1})|\xi|^2,$
which further implies
$$
\frac{w^{2\ell_2}(\xi)}{w^{2\ell_2}(\xi_\ast')}\leq \left(\frac{1}{1-\ell_2^{-1}}\right)^{2\ell_2}\leq e^2.
$$
Thus, it follows
\begin{align}
|\CI_2|\leq&C\int_{\R^3}\int_{\R^3\times\S^2}B_0
\chi_M\nu^{-1}(\xi)\frac{w^{2\ell_2}(\xi)}{w^{2\ell_2}(\xi_\ast')}\lag\xi'\rag^\ga \lag\xi'_\ast\rag^\ga\mu(\xi')w^{2\ell_2}(\xi_\ast')|f(\xi_\ast')|d\xi_\ast d\om d\xi\notag\\
\leq&\frac{C}{(1+M)^\ga}\int_{\R^3}\int_{\R^3\times\S^2}B_0
\lag\xi'\rag^\ga \lag\xi'_\ast\rag^\ga \mu(\xi')|f(\xi'_\ast)|d\xi_\ast d\om d\xi
\leq
\frac{C}{(1+M)^\ga}\|\nu w^{\ell_2}f\|_{L^1}.\notag
\end{align}
If $|\xi'|\geq\frac{1}{\sqrt{\ell_2}}|\xi|$, similarly for obtaining \eqref{I1-es}, one has
\begin{align}
|\CI_2|\leq&C\int_{\R^3}\chi_M\nu^{-1}(w^{2\ell_2}\mu^{\frac{1}{2\ell_2}})(\xi)d\xi\int_{\R^3\times\S^2}\lag\xi'\rag^\ga \lag\xi'_\ast\rag^\ga \mu^{\frac{1}{2}}(\xi')|f(\xi_\ast')|d\xi_\ast d\om\notag\\
\leq&\frac{C}{(1+M)^\ga}\int_{\R^3\times\R^3}\lag\xi\rag^\ga \lag\xi_\ast\rag^\ga \mu^{\frac{1}{2}}(\xi_\ast)|f(\xi)|d\xi_\ast d\xi
\leq \frac{C}{(1+M)^\ga}\|\nu w^{2\ell_2}f\|_{L^1}.\notag
\end{align}
Combing the above estimates together, we then see that \eqref{CK-l1} is true.
Next, we define linear operators
\begin{align}
T=T_1=T_2=\chi_{M} \CK:&\ L^{1}(\nu w^{2\ell_2})\cap L^{\infty}(\nu w^{2\ell_2})\rightarrow L^{1}(\nu^{-1}w^{2\ell_2})\cap L^{\infty}(\nu^{-1}w^{2\ell_2})\notag\\
&f\mapsto \chi_{M} \CK f, \notag
\end{align}
where $L^p(W(\xi))$ $(p\in[1,+\infty])$ is a Lebesgue space with  weighted measure $d\mathfrak{M}=Wd\xi.$
Therefore \eqref{CK-l2} with $\vth=0$ follows from \eqref{CK-l1} and \eqref{CK-ifif} as well as \eqref{RT-it}.

We now turn to show that \eqref{CK-l2} is also true for $\vth>0.$ In this case, we first have by a change of variables $\xi_\ast-\xi\rightarrow u$ that
\begin{align}\label{CK-de-exp-p1}
\pa^\vth_\xi(\chi_M\CK f)=&\sum\limits_{\vth'+\vth''\leq\vth}C_\vth^{\vth',\vth''}
(\pa^{\vth-\vth'-\vth''}_\xi\chi_M)\int_{\R^3\times\S^2}B_0|u|^\ga (\pa^{\vth''}_\xi\mu)(\xi+u_{\perp})
(\pa^{\vth'}_\xi f)(\xi+u_{\parallel}) d\om du
\notag\\&+\sum\limits_{\vth'+\vth''\leq\vth}C_\vth^{\vth',\vth''}
(\pa^{\vth-\vth'-\vth''}_\xi\chi_M)\int_{\R^3\times\S^2}B_0|u|^\ga
(\pa^{\vth''}_\xi\mu)(\xi+u_{\parallel})(\pa^{\vth'}_\xi f)(\xi+u_{\perp})d\om du
\notag\\&+\sum\limits_{\vth'+\vth''\leq\vth}C_\vth^{\vth',\vth''}(\pa^{\vth-\vth'-\vth''}_\xi\chi_M)
\int_{\R^3\times\S^2}B_0|u|^\ga (\pa^{\vth''}_\xi\mu)(\xi)(\pa^{\vth'}_\xi f)(u+\xi)d\om du,
\end{align}
where we have also used the notations $u_{\parallel}=(u\cdot\om)\om$ and $u_{\perp}=u-u_{\parallel}$. Then, changing back to the original variables, one has
\begin{align}\label{CK-de-exp-p2}
\pa^\vth_\xi(\chi_M\CK f)=&\sum\limits_{\vth'+\vth''\leq\vth}C_\vth^{\vth',\vth''}
(\pa^{\vth-\vth'-\vth''}_\xi\chi_M)\int_{\R^3\times\S^2}B_0|\xi_\ast-\xi|^\ga (\pa^{\vth''}_\xi\mu)(\xi_\ast')
(\pa^{\vth'}_\xi f)(\xi') d\om d\xi_\ast
\notag\\&-\sum\limits_{\vth'+\vth''\leq\vth}C_\vth^{\vth',\vth''}
(\pa^{\vth-\vth'-\vth''}_\xi\chi_M)\int_{\R^3\times\S^2}B_0|\xi_\ast-\xi|^\ga
(\pa^{\vth''}_\xi\mu)(\xi')(\pa^{\vth'}_\xi f)(\xi_\ast')d\om d\xi_\ast
\notag\\&+\sum\limits_{\vth'+\vth''\leq\vth}C_\vth^{\vth',\vth''}(\pa^{\vth-\vth'-\vth''}_\xi\chi_M)
\int_{\R^3\times\S^2}B_0|\xi_\ast-\xi|^\ga (\pa^{\vth''}_\xi\mu)(\xi)(\pa^{\vth'}_\xi f)(\xi_\ast)d\om d\xi_\ast.
\end{align}
We now define the following linear continuous operators
\begin{align}
T_\vth=T_{\vth,1}=T_{\vth,2}: &\ L^{1}(\nu w^{2\ell_2})\cap L^{\infty}(\nu w^{2\ell_2})
\rightarrow L^{1}(\nu^{-1}w^{2\ell_2})\cap L^{\infty}(\nu^{-1}w^{2\ell_2})\notag,
\end{align}
with
\begin{align}
T_\vth g=T_{\vth,1}g=T_{\vth,2}g=&\sum\limits_{\vth'+\vth''\leq\vth}C_\vth^{\vth',\vth''}
(\pa^{\vth-\vth'-\vth''}_\xi\chi_M)\int_{\R^3\times\S^2}B_0|\xi_\ast-\xi|^\ga (\pa^{\vth''}_\xi\mu)(\xi_\ast')
g(\xi') d\om d\xi_\ast
\notag\\&+\sum\limits_{\vth'+\vth''\leq\vth}C_\vth^{\vth',\vth''}
(\pa^{\vth-\vth'-\vth''}_\xi\chi_M)\int_{\R^3\times\S^2}B_0|\xi_\ast-\xi|^\ga
(\pa^{\vth''}_\xi\mu)(\xi')g(\xi_\ast')d\om d\xi_\ast
\notag\\&+\sum\limits_{\vth'+\vth''\leq\vth}C_\vth^{\vth',\vth''}(\pa^{\vth-\vth'-\vth''}_\xi\chi_M)
\int_{\R^3\times\S^2}B_0|\xi_\ast-\xi|^\ga (\pa^{\vth''}_\xi\mu)(\xi)g(\xi_\ast)d\om d\xi_\ast,\notag
\end{align}
for any given $\vth\in\Z_+^3$.

Because $|\pa^{\vth''}_\xi\mu|\leq C_{\vth''}\mu^{\frac{1}{2}}$, one can also deduce by performing the similar calculations as for obtaining
\eqref{CK-ifif} and \eqref{CK-l1}  that
\begin{align}\label{CK-ifif-h}
\|\nu^{-1} w^{2\ell_2}T_{\vth,2}(\pa^{\vth'}_\xi f)\|_{L^\infty}\leq \tilde{C}_1\{(1+M)^{-\ga}+\varsigma\}\|\nu w^{2\ell_2}\pa^{\vth'}_\xi f\|_{L^\infty},
\end{align}
and
\begin{align}\label{CK-l1-h}
\|\nu^{-1} w^{2\ell_2}T_{\vth,1}(\pa^{\vth'}_\xi f)\|_{L^1}\leq \frac{\tilde{C}_2}{(1+M)^\ga}\|\nu w^{2\ell_2}\pa^{\vth'}_\xi f\|_{L^1}.
\end{align}
Then, \eqref{CK-ifif-h} together with \eqref{CK-l1-h} further gives
\begin{align}\label{CK-l2-h}
\|\nu^{-1/2} w^{2\ell_2}T_{\vth}(\pa^{\vth'}_\xi f)\|\leq \sqrt{\tilde{C}_1\tilde{C}_2}
\left\{(1+M)^{-\ga}+\varsigma\right\}^{\frac{1}{2}}(1+M)^{-\ga/2}\|\nu^{1/2} w^{2\ell_2}\pa^{\vth'}_\xi f\|,
\end{align}
according to Lemma \ref{RT-lem}. From \eqref{CK-l2-h} and \eqref{CK-de-exp-p2}, we see that \eqref{CK-l2} is also valid in the case $\vth>0$ .
This ends the proof of Proposition \ref{CK-l2-pro}.
\end{proof}

\section{Local existence}\label{sec.le}
The goal of this section is to construct the local existence of the remainder problem \eqref{GR-eq} and \eqref{GR-id} in the Sobolev space $W^{N,\infty}$ for an arbitrary positive integer $N$. 

Since $G_1$ and $G_2$ are already given by  \eqref{G1-exp} and  \eqref{G2-exp},
to solve \eqref{F-eq} and \eqref{F-id} it suffices
to determine $G_R$ by  the Cauchy problem \eqref{GR-eq} and \eqref{GR-id}. Here, we also recall that $\beta=\beta(t)$ is determined by solving \eqref{bet-G} or equivalently \eqref{be-exp} in terms of $G_R$.
To do this, one crucial idea behind the proof is to split $G_R$ as $\sqrt{\mu}G_R=G_{R,1}+\sqrt{\mu}G_{R,2}$, where $G_{R,1}$ and $G_{R,2}$ satisfy
\begin{align}\label{GR1-eq}
\pa_t G_{R,1}&-\frac{\beta'}{\beta}\na_\xi\cdot(\xi G_{R,1})
-\al\na_\xi\cdot(A\xi G_{R,1})+\beta^\ga \nu G_{R,1}\notag\\
=&\beta^\ga\chi_M \CK G_{R,1}-\frac{\beta'}{2\beta} |\xi|^2\sqrt{\mu}G_{R,2}-\frac{\al}{2} \xi\cdot(A\xi)\sqrt{\mu}G_{R,2}
-\al^{1-m}\sqrt{\mu}\pa_t G_1
-\al^{2-m}\sqrt{\mu}\pa_t G_2\notag\\&+\al^{3-m}\beta_1\na_\xi\cdot(\xi\mu)+\al^{1-m}\frac{\beta'}{\beta}\na_\xi\cdot(\xi\sqrt{\mu}G_1)
+\al^{2-m}\frac{\beta'}{\beta}\na_\xi\cdot(\xi\sqrt{\mu}G_2)
+\al^{3-m}\na_\xi\cdot(A\xi\sqrt{\mu}G_2)\notag\\&+\al^{3-m}\beta^\ga \{Q(\sqrt{\mu}G_1,\sqrt{\mu}G_2)+Q(\sqrt{\mu}G_2,\sqrt{\mu}G_1)\}+\al^{4-m}\beta^\ga Q(\sqrt{\mu}G_2,\sqrt{\mu}G_2)
\notag\\&+\al\beta^\ga\{Q(\sqrt{\mu}G_1+\al \sqrt{\mu}G_2,\sqrt{\mu}G_{R})+Q(\sqrt{\mu}G_{R},\sqrt{\mu}G_1+\al \sqrt{\mu}G_2)\}\notag\\
&+\al^m\beta^\ga Q(\sqrt{\mu}G_{R},\sqrt{\mu}G_{R}),
\end{align}
\begin{align}
G_{R,1}(0,\xi)=G_{R,0},\notag
\end{align}
\begin{align}\label{GR2-eq}
\pa_t G_{R,2}&-\frac{\beta'}{\beta}\na_\xi\cdot(\xi G_{R,2})
-\al\na_\xi\cdot(A\xi G_{R,2})+\beta^\ga L G_{R,2}
=\beta^\ga(1-\chi_M)\mu^{-\frac{1}{2}} \CK G_{R,1},
\end{align}
and
\begin{align}
G_{R,2}(0,\xi)=0,\notag
\end{align}
respectively. 

\begin{theorem}[Local existence]\label{loc.ex}
Let $\ga\in(0,1]$, $m\in(2,3)$, $\ell_\infty\gg 5$ and assume ${\rm tr}A=0$.
There exists a constant $M_0>0$ and a suitably small constant $\al_0>0$  such that if $\al\in(0,\al_0)$
and
\begin{align}\label{id-tt}
\sum\limits_{|\vth|\leq N}\|w^{\ell_\infty}\pa_\xi^\vth G_{R,0}\|_{L^\infty}\leq M_0,
\end{align}
for an integer $N\geq 1$,
then there exists $T_0>0$ which may depend on $\al$ and $M_0$ such that \eqref{GR-eq} and \eqref{GR-id} admits a unique local  solution $G_R(t,\xi)$ satisfying $\sqrt{\mu}G_R=G_{R,1}+\sqrt{\mu}G_{R,2}$ with the estimate
\begin{align}
\sup\limits_{0\leq t\leq T_0}\sum\limits_{|\vth|\leq N}&
\left\{\left\|w^{\ell_\infty}\beta^{2\ga}\pa_\xi^\vth G_{R,1}\right\|_{L^\infty}
+\left\|w^{\ell_\infty}\beta^{2\ga}\pa_\xi^\vth  G_{R,2}\right\|_{L^\infty}\right\}
\leq 2M_0,\notag
\end{align}
where
\begin{align}
\beta^\ga(t)\sim 1+\ga\vho_0\al^2t.\notag
\end{align}
\end{theorem}

\begin{proof}
Our proof is based on the Duhamel's principle and contraction mapping method. It is convenient to look for the weighted form
$
[g_{1},g_{2}](t,\xi)={\beta}^{2\ga}(t)[G_{R,1}(t,\xi),G_{R,2}(t,\xi)].
$ 
For the purpose, we consider the following linear inhomogeneous equations for the unknown $[g_{1},g_{2}](t,\xi)$:
\begin{align}\label{ap-g1-eq}
\pa_t g_{1}
&-\frac{\tilde{\beta}'}{\tilde{\beta}}\xi \cdot\na_\xi  g_{1}
-\frac{(2\ga+3)\tilde{\beta}'}{\tilde{\beta}}g_{1}
-\al A\xi\cdot\na_\xi g_{1}
+\nu\tilde{\beta}^{\ga}g_{1}\notag\\
=&\tilde{\beta}^\ga\chi_M \CK h_1-\frac{\tilde{\beta}'}{2\tilde{\beta}}|\xi|^2\sqrt{\mu}h_2-\frac{\al}{2} \xi\cdot(A\xi)\sqrt{\mu}h_2
-\al^{1-m}\tilde{\beta}^{2\ga}\sqrt{\mu}\pa_t H_1-\al^{2-m}\tilde{\beta}^{2\ga}\sqrt{\mu}\pa_t H_2\notag\\&+\al^{3-m}\tilde{\beta}^{2\ga}\tilde{\beta}_1\na_\xi\cdot(\xi\mu)
+\al^{1-m}\tilde{\beta}^{2\ga-1}\tilde{\beta}'\na_\xi\cdot(\xi\sqrt{\mu}H_1)
+\al^{2-m}\tilde{\beta}^{2\ga-1}\tilde{\beta}'\na_\xi\cdot(\xi\sqrt{\mu}H_2)
\notag\\&+\al^{3-m}\tilde{\beta}^{2\ga}\na_\xi\cdot(A\xi\sqrt{\mu}H_2)
+\al^{3-m}\tilde{\beta}^{3\ga}\{Q(\sqrt{\mu}H_1,\sqrt{\mu}H_2)+Q(\sqrt{\mu}H_2,\sqrt{\mu}H_1)\}
\notag\\&+\al^{4-m}\tilde{\beta}^{3\ga} Q(\sqrt{\mu}H_2,\sqrt{\mu}H_2)
+\al\tilde{\beta}^\ga\{Q(\sqrt{\mu}G_1+\al \sqrt{\mu}G_2,\sqrt{\mu}h)+Q(\sqrt{\mu}h,\sqrt{\mu}G_1+\al \sqrt{\mu}G_2)\}
\notag\\&+\al^m\tilde{\beta}^{-\ga} Q(\sqrt{\mu}h,\sqrt{\mu}h),
\end{align}
\begin{align}
g_{1}(0,\xi)=G_{R,0},\label{ap-g1-id}
\end{align}
\begin{align}\label{ap-g2-eq}
\pa_t g_{2}
&-\frac{\tilde{\beta}'}{\tilde{\beta}}\xi \cdot\na_\xi  g_{2}
-\frac{(2\ga+3)\tilde{\beta}'}{\tilde{\beta}}g_{2}
-\al A\xi\cdot\na_\xi g_{2}
+\nu\tilde{\beta}^{\ga}g_{2}\notag\\
=&\tilde{\beta}^\ga\chi_M  K  h_{2}
+\tilde{\beta}^\ga (1-\chi_M)\mu^{-\frac{1}{2}} \CK h_{1},
\end{align}
and
\begin{align}
g_{2}(0,\xi)=0.
\label{ap-g2-id}
\end{align}
Here, we have given
\begin{align}
\sqrt{\mu}h=h_1+\sqrt{\mu}h_2=\tilde{\beta}^{2\ga}(H_{R,1}+\sqrt{\mu}H_{R,2}),\ \sqrt{\mu}H_{R}=H_{R,1}+\sqrt{\mu}H_{R,2},\label{h-def}
\end{align}
\begin{align}
H=\mu+\sqrt{\mu}\{\al H_1+\al^2 H_2+\al^m H_R\},\notag
\end{align}
\begin{align}
H_1=-\tilde{\beta}^{-\ga}L^{-1}\{\xi\cdot A\xi\mu^{\frac{1}{2}}\},\label{H1-exp}
\end{align}
\begin{align}
H_2=&\tilde{\beta}^{-\ga}L^{-1}\left\{\tilde{\beta}^\ga\Ga(H_1,H_1)+\tilde{\beta}_0\na_\xi\cdot(\xi\mu)\mu^{-\frac{1}{2}}
+\na_\xi\cdot(A\xi\sqrt{\mu}H_1)\mu^{-\frac{1}{2}}\right\},\label{H2-exp}
\end{align}
\begin{align}
\frac{\tilde{\beta}'}{\tilde{\beta}}=-\frac{\al}{3}\int_{\R^3}\xi\cdot A\xi Hd\xi,\ \tilde{\beta}'=\frac{d\tilde{\beta}}{dt},\ \tilde{\beta}(0)=1,\label{bet-H}
\end{align}
\begin{align}
\frac{\tilde{\beta}'}{\tilde{\beta}}=\tilde{\beta}_0\al^2+\tilde{\beta}_1\al^3,\label{beh-exp}
\end{align}
\begin{align}
\tilde{\beta}_0=-\frac{1}{3}\int_{\R^3}\xi\cdot A\xi \sqrt{\mu} H_1d\xi=\vho_0\tilde{\beta}^{-\ga},\label{be0h-def}
\end{align}
\begin{align}\label{be1h-def}
\tilde{\beta}_1=-\frac{1}{3}\int_{\R^3}\xi\cdot A\xi \sqrt{\mu}H_2d\xi-\frac{\al^{m-2}}{3}\int_{\R^3}\xi\cdot A\xi \sqrt{\mu}H_Rd\xi
=\vho_1\tilde{\beta}^{-2\ga}+\vho_R\tilde{\beta}^{-2\ga},
\end{align}
with
\begin{align}
\vho_1=&-\frac{\tilde{\beta}^{2\ga}}{3}\int_{\R^3}\xi\cdot A\xi \sqrt{\mu}H_2d\xi\notag\\
=&-\frac{1}{3}\int_{\R^3}\xi\cdot A\xi \sqrt{\mu}\bigg\{L^{-1}\Big\{\Ga(L^{-1}\{\xi\cdot A\xi\mu^{\frac{1}{2}}\},L^{-1}\{\xi\cdot A\xi\mu^{\frac{1}{2}}\})+\vho_0\na_\xi\cdot(\xi\mu)\mu^{-\frac{1}{2}}
\notag\\&\qquad\qquad-\na_\xi\cdot(A\xi\sqrt{\mu}L^{-1}\{\xi\cdot A\xi\mu^{\frac{1}{2}}\})\mu^{-\frac{1}{2}}\Big\}\bigg\}d\xi,\notag
\end{align}
and
\begin{align}
\vho_R=-\frac{\tilde{\beta}^{2\ga}\al^{m-2}}{3}\int_{\R^3}\xi\cdot A\xi \sqrt{\mu}H_Rd\xi.\notag
\end{align}
Consequently, \eqref{bet-H} and \eqref{be1h-def} give
\begin{align}
\frac{d\tilde{\beta}^{\ga}}{dt}=\ga\left\{\vho_0\al^2+\vho_1\al^3\tilde{\beta}^{-\ga}+\vho_R\al^3\tilde{\beta}^{-\ga}\right\}.\label{tbeta-ga}
\end{align}
It should be pointed out that both $\vho_0$ and $\vho_1$ are independent of $\tilde{\beta}.$

Let $[g_1,g_2]$ be a solution of the coupled problems \eqref{ap-g1-eq}, \eqref{ap-g1-id} and \eqref{ap-g2-eq}, \eqref{ap-g2-id} with $[h_1,h_2]$ being given.
Then the solution operator $\CN$ is formally defined as
$$
[g_1,g_2]=\CN([h_1,h_2]).
$$
We aim to prove that there exists a sufficiently small $T_0>0$ such that the solution mapping $\CN([\cdot,\cdot])$ has a unique fixed point in
some Banach space by adopting the contraction mapping method. In fact, from \eqref{id-tt} and \eqref{ap-g2-id}, one has
\begin{align*}
\sum\limits_{|\vth|\leq N}\|w^{\ell_\infty}\pa^{\vth}_\xi g_1(0,\xi)\|_{L^\infty}\leq M_0.
\end{align*}
Thus we can define the following Banach space
\begin{equation*}
\begin{split}
\FY_{\al,T}=&\Big\{(\CG_1,\CG_2)\bigg| \sup\limits_{0\leq t\leq T}\sum\limits_{|\vth|\leq N}\left\{\|w^{\ell_\infty}\pa_\xi^{\vth}\CG_1(t)\|_{L^\infty}+\|w^{\ell_\infty}\pa_\xi^{\vth}\CG_2(t)\|_{L^\infty}\right\}
\leq2M_0,\ \\
&\qquad\qquad\qquad\CG_1(0)=G_{R,0},\ \CG_2(0)=0,\ \lag\CG_1,[1,\xi,\frac{1}{2}|\xi|^2]\rag
+\lag\CG_2,[1,\xi,\frac{1}{2}|\xi|^2]\mu^{\frac{1}{2}}\rag=0
\Big\},
\end{split}
\end{equation*}
associated with the norm
$$
\|[\CG_1,\CG_2]\|_{\FY_{\al,T}}=\sup\limits_{0\leq t\leq T}\sum\limits_{|\vth|\leq N}\left\{\|w^{\ell_\infty}\pa_\xi^{\vth}\CG_1(t)\|_{L^\infty}+\|w^{\ell_\infty}\pa_\xi^{\vth}\CG_2(t)\|_{L^\infty}\right\}.
$$
We now show that
\begin{align*}
\CN: \FY_{\al,T}\rightarrow \FY_{\al,T},
\end{align*}
is well-defined and $\CN$ is a contraction mapping for some $T>0$. To do this, we start from the following approximation equations
\begin{align}\label{h1-eq}
\pa_t &(w^{\ell_\infty}\pa_\xi^\vth g_{1})
-\frac{\tilde{\beta}'}{\tilde{\beta}}\xi \cdot\na_\xi(w^{\ell_\infty} \pa_\xi^\vth  g_{1})
-\frac{(2\ga+3)\tilde{\beta}'}{\tilde{\beta}}w^{\ell_\infty} \pa_\xi^\vth  g_{1}
+\frac{2\ell_{\infty}\tilde{\beta}'}{\tilde{\beta}}\frac{|\xi|^2}{1+|\xi|^2}w^{\ell_\infty}\pa_\xi^{\vth}g_{1}
\notag\\&-\al A\xi\cdot\na_\xi( w^{\ell_\infty} \pa_\xi^\vth g_{1})
+2\ell\al\frac{\xi\cdot A\xi}{1+|\xi|^2}w^{\ell_\infty}\pa_\xi^{\vth}g_{1}+\nu\tilde{\beta}^{\ga} w^{\ell_\infty} \pa_\xi^\vth g_{1}\notag\\
=&\sum\limits_{\vth'\leq\vth}C_\vth^{\vth'} w^{\ell_\infty}\tilde{\beta}^{\ga}\pa_\xi^{\vth'}(\chi_M \CK)\pa_\xi^{\vth-\vth'}h_1
+\frac{\tilde{\beta}'}{\tilde{\beta}} {\bf 1}_{|\vth|>0}\sum\limits_{|\vth'|=1}
C_\vth^{\vth'}w^\ell\pa_\xi^{\vth'}\xi\cdot\na_\xi\pa_\xi^{\vth-\vth'}h_1
\notag\\&+\al{\bf 1}_{|\vth|>0}\sum\limits_{|\vth'|=1}C_\vth^{\vth'}
w^{\ell_\infty}\pa_\xi^{\vth'}(A\xi)\cdot\na_\xi\pa_\xi^{\vth-\vth'}h_1
-{\bf 1}_{|\vth|>0}C_\vth^{\vth'}\tilde{\beta}^{\ga} \pa_\xi^{\vth'}\nu w^{\ell_\infty} \pa_\xi^{\vth-\vth'} h_1
-\frac{\tilde{\beta}'}{2\tilde{\beta}} w^{\ell_\infty}\pa_\xi^{\vth}\left\{|\xi|^2\sqrt{\mu}h_2\right\}
\notag\\&-\frac{\al}{2} w^{\ell_\infty}\pa_\xi^{\vth}\left\{\xi\cdot(A\xi)\sqrt{\mu}h_{2}\right\}
-\al^{1-m}\tilde{\beta}^{2\ga}w^{\ell_\infty}\pa_\xi^{\vth}\left\{\sqrt{\mu}\pa_t H_1\right\}
-\al^{2-m}\tilde{\beta}^{2\ga}w^{\ell_\infty}\pa_\xi^{\vth}\left\{\sqrt{\mu}\pa_t H_2\right\}
\notag\\&+\al^{3-m}\tilde{\beta}^{2\ga}\tilde{\beta}_1w^{\ell_\infty}\pa_\xi^{\vth}\left\{\na_\xi\cdot(\xi\mu)\right\}
+\al^{1-m}\tilde{\beta}'\tilde{\beta}^{2\ga-1} w^{\ell_\infty}\pa_\xi^{\vth}\left\{\na_\xi\cdot(\xi\sqrt{\mu}H_1)\right\}
+\al^{2-m}\tilde{\beta}'\tilde{\beta}^{2\ga-1}w^{\ell_\infty}\pa_\xi^{\vth}\left\{\na_\xi\cdot(\xi\sqrt{\mu}H_2)\right\}
\notag\\&+\al^{3-m}\tilde{\beta}^{2\ga} w^{\ell_\infty}\na_\xi\cdot(A\xi\sqrt{\mu}H_2)
+\al^{3-m}\tilde{\beta}^{3\ga}w^{\ell_\infty}\pa_\xi^{\vth}\{Q(\sqrt{\mu}G_1,\sqrt{\mu}G_2)+Q(\sqrt{\mu}G_2,\sqrt{\mu}G_1)\}
\notag\\&+\al^{4-m}\tilde{\beta}^{3\ga}w^{\ell_\infty} \pa_\xi^{\vth}Q(\sqrt{\mu}G_2,\sqrt{\mu}G_2)
+\al^m \tilde{\beta}^{-\ga}w^{\ell_\infty} \pa_\xi^{\vth}Q(\sqrt{\mu} h,\sqrt{\mu} h)\notag\\&+\al \tilde{\beta}^{\ga}w^{\ell_\infty}\pa_\xi^{\vth}\{Q(\sqrt{\mu}G_1+\al \sqrt{\mu}G_2,\sqrt{\mu}h)
+Q(\sqrt{\mu}h,\sqrt{\mu}G_1+\al \sqrt{\mu}G_2)\},
\end{align}
\begin{align}
g_{1}(0,\xi)=G_{R,0},\label{h1-id}
\end{align}
\begin{align}\label{h2-eq}
\pa_t (w^{\ell_\infty}\pa_\xi^\vth g_{2})&
-\frac{\tilde{\beta}'}{\tilde{\beta}}\xi \cdot\na_\xi(w^{\ell_\infty} \pa_\xi^\vth  g_{2})
-\frac{(2\ga+3)\tilde{\beta}'}{\tilde{\beta}}w^{\ell_\infty} \pa_\xi^\vth  g_{2}
+\frac{2{\ell_\infty}\tilde{\beta}'}{\tilde{\beta}}\frac{|\xi|^2}{1+|\xi|^2}w^{\ell_\infty}\pa_\xi^{\vth}g_{2}
\notag\\&-\al A\xi\cdot\na_\xi( w^{\ell_\infty}\pa_\xi^\vth g_{2})+2{\ell}\al\frac{\xi\cdot A\xi}{1+|\xi|^2}w^{\ell_\infty}\pa_\xi^{\vth}g_{2}
+\tilde{\beta}^\ga \nu w^{\ell_\infty} \pa_\xi^\vth g_{2}\notag\\
=&\sum\limits_{\vth'\leq\vth}C_\vth^{\vth'}\tilde{\beta}^\ga w^{\ell_\infty}(\pa_\xi^{\vth'} K)\pa_\xi^{\vth-\vth'}h_2
+\frac{\tilde{\beta}'}{\tilde{\beta}} {\bf 1}_{|\vth|>0}\sum\limits_{|\vth'|=1}
C_\vth^{\vth'}w^{\ell_\infty}\pa_\xi^{\vth'}\xi\cdot\na_\xi\pa_\xi^{\vth-\vth'}h_{2}
\notag\\&
-{\bf 1}_{|\vth|>0}C_\vth^{\vth'}\tilde{\beta}^\ga w^{\ell_\infty}\pa_\xi^{\vth'}\nu  \pa_\xi^{\vth-\vth'}h_{2}
+\tilde{\beta}^\ga w^{\ell_\infty}\pa_\xi^\vth\left\{(1-\chi_M)\mu^{-\frac{1}{2}} \CK h_{2}\right\},
\end{align}
and
\begin{align}
g_{2}(0,\xi)=0.\label{h2-id}
\end{align}
Next, we define the characteristic line $[s,V(s;t,\xi)]$ for equations \eqref{h1-eq} and \eqref{h2-eq} going through $(t,\xi)$ such that
\begin{align}\label{CL}
\left\{\begin{array}{rll}
&\frac{d V(s;t,\xi)}{ds}=-\frac{\tilde{\beta}'(s)}{\tilde{\beta}(s)} V(s;t,\xi)-\al A V(s;t,\xi),\\[2mm]
&V(t;t,\xi)=\xi,
\end{array}\right.
\end{align}
which can be solved as
\begin{equation}\label{V}
V(s):=V(s;t,\xi)=\frac{\tilde{\beta}(t)}{\tilde{\beta}(s)}e^{-(s-t)\al A}\xi.
\end{equation}
By this, we can write the solution of \eqref{h1-eq}, \eqref{h1-id}, \eqref{h2-eq} and \eqref{h2-id} as
\begin{align}
[w^{\ell_\infty}\pa_\xi^\vth g_{1},w^{\ell_\infty}\pa_\xi^\vth g_{2}]=\CQ_{\vth}(h_1,h_2)=[\CQ_{1,\vth}(h_1,h_2),\CQ_{2,\vth}(h_1,h_2)],\notag
\end{align}
where
\begin{align}
\CQ_{1,\vth}(h_1,h_2)=\sum\limits_{i=1}^{10}\CJ_{i},\ \CQ_{2,\vth}(h_1,h_2)=\sum\limits_{i=11}^{15}\CJ_{i},\
\label{Q12}
\end{align}
with
\begin{align}
\CJ_1=e^{-\int_0^t\tilde{\CA}(s)ds}w^{\ell_\infty}\pa_\xi^\vth G_{R,0}(V(0)),\notag
\end{align}
\begin{align}
\CJ_2= {\bf 1}_{|\vth|>0}\sum\limits_{|\vth'|=1}C_\vth^{\vth'}\int_0^te^{-\int_s^t\tilde{\CA}(\tau)d\tau}
\left\{\frac{\tilde{\beta}'}{\tilde{\beta}} w^\ell\pa_\xi^{\vth'}\xi\cdot\na_\xi\pa_\xi^{\vth-\vth'}h_1\right\}(s,V(s))ds,\notag
\end{align}
\begin{align}
\CJ_3=\al{\bf 1}_{|\vth|>0}\sum\limits_{|\vth'|=1}C_\vth^{\vth'}\int_0^te^{-\int_s^t\tilde{\CA}(\tau)d\tau}
\{w^{\ell_\infty}\pa_\xi^{\vth'}(A\xi)\cdot\na_\xi\pa^{\vth-\vth'}_\xi h_{1}\}(s,V(s))ds,\notag
\end{align}
\begin{align}
\CJ_4=-{\bf 1}_{|\vth|>0}\sum\limits_{0<\vth'\leq \vth}C_\vth^{\vth'}
\int_0^te^{-\int_s^t\tilde{\CA}(\tau)d\tau}\{\tilde{\beta}^{\ga} w^{\ell_\infty}\pa^{\vth'}_\xi\nu\pa^{\vth-\vth'}_{\xi}h_{1}\}(s,V(s))ds
,\notag
\end{align}
\begin{align}
\CJ_5=\sum\limits_{\vth'\leq\vth}C_\vth^{\vth'}\int_0^te^{-\int_s^t\tilde{\CA}(\tau)d\tau}\{\tilde{\beta}^{\ga} w^{\ell_\infty}\pa_\xi^{\vth'}(\chi_M \CK)\pa_\xi^{\vth-\vth'}h_1\}(s,V(s))ds
,\notag
\end{align}
\begin{align}
\CJ_6=-\int_0^t&e^{\int_s^t\tilde{\CA}(\tau)d\tau}\left\{\frac{\tilde{\beta}'}{2\tilde{\beta}}w^{\ell_\infty}\pa_{\xi}^\vth(|\xi|^2\sqrt{\mu}h_{2})
+\frac{\al}{2} w^{\ell_\infty}\pa_{\xi}^\vth(\xi\cdot A\xi\sqrt{\mu}h_{2})\right\}(s,V(s))ds,\notag
\end{align}
\begin{multline}
\CJ_7=-\int_0^te^{\int_s^t\tilde{\CA}(\tau)d\tau}\Big\{-\al^{1-m}\tilde{\beta}^{2\ga}w^{\ell_\infty}\pa_\xi^{\vth}\left\{\sqrt{\mu}\pa_t H_1\right\}
-\al^{2-m}\tilde{\beta}^{2\ga} w^{\ell_\infty}\pa_\xi^{\vth}\left\{\sqrt{\mu}\pa_t H_2\right\}
\notag\\
+\al^{3-m}\tilde{\beta}^{2\ga}\tilde{\beta}_1w^{\ell_\infty}\pa_\xi^{\vth}\left\{\na_\xi\cdot(\xi\mu)\right\}
+\al^{1-m}\tilde{\beta}^{2\ga-1}\tilde{\beta}' w^{\ell_\infty}\pa_\xi^{\vth}\left\{\na_\xi\cdot(\xi\sqrt{\mu}H_1)\right\}\Big\}(s,V(s))ds,\notag
\end{multline}
\begin{multline}
\CJ_8=-\int_0^te^{\int_s^t\tilde{\CA}(\tau)d\tau}\left\{\al^{2-m}\tilde{\beta}^{2\ga-1}\tilde{\beta}' w^{\ell_\infty}\pa_\xi^{\vth}
\left\{\na_\xi\cdot(\xi\sqrt{\mu}H_2)\right\}\right.\notag\\
\left.+\al^{3-m}\tilde{\beta}^{2\ga}w^{\ell_\infty}\pa_\xi^{\vth}\left\{\na_\xi\cdot(A\xi\sqrt{\mu}H_2)\right\}\right\}(s,V(s))ds,\notag
\end{multline}
\begin{align}
\CJ_{9}=\int_0^te^{-\int_s^t\tilde{\CA}(\tau)d\tau}\Big\{\al^{3-m}&\tilde{\beta}^{3\ga}  w^{\ell_\infty} \pa_\xi^{\vth}\{Q(\sqrt{\mu}H_1,\sqrt{\mu}H_2)+Q(\sqrt{\mu}H_2,\sqrt{\mu}H_1)\}
\notag\\&+\al^{4-m}\tilde{\beta}^{3\ga}  w^{\ell_\infty}\pa_\xi^{\vth}Q(\sqrt{\mu}H_2,\sqrt{\mu}H_2)\Big\}(s,V(s))ds,\notag
\end{align}
\begin{multline}
\CJ_{10}=\int_0^te^{-\int_s^t\tilde{\CA}(\tau)d\tau}\Big\{\al^m\tilde{\beta}^{-\ga} w^{\ell_\infty} \pa_\xi^{\vth}Q(\sqrt{\mu}h,\sqrt{\mu}h)
\notag\\
+\al \tilde{\beta}^{\ga}w^{\ell_\infty}\pa_\xi^{\vth}\{Q(\sqrt{\mu}H_1+\al \sqrt{\mu}H_2,\sqrt{\mu}h)
+Q(\sqrt{\mu}h,\sqrt{\mu}H_1+\al \sqrt{\mu}H_2)\}\Big\}(s,V(s))ds,\notag
\end{multline}
\begin{align}
\CJ_{11}= {\bf 1}_{|\vth|>0}\sum\limits_{|\vth'|=1}C_\vth^{\vth'}\int_0^te^{-\int_s^t\tilde{\CA}(\tau)d\tau}
\left\{\frac{\tilde{\beta}'}{\tilde{\beta}}w^{\ell_\infty}\pa^{\vth'}_\xi \xi\cdot\na_\xi\pa_\xi^{\vth-\vth'}h_{2}\right\}(s,V(s))ds,\notag
\end{align}
\begin{align}
\CJ_{12}=\al{\bf 1}_{|\vth|>0}\sum\limits_{|\vth'|=1}C_\vth^{\vth'}\int_0^te^{-\int_s^t\tilde{\CA}(\tau)d\tau}
\{w^{\ell_\infty}\pa_\xi^{\vth'}(A\xi)\cdot\na_\xi\pa^{\vth-\vth'}_\xi h_{2}\}(s,V(s))ds,\notag
\end{align}
\begin{align}
\CJ_{13}=-{\bf 1}_{|\vth|>0}\sum\limits_{0<\vth'\leq \vth}C_\vth^{\vth'}
\int_0^te^{-\int_s^t\tilde{\CA}(\tau)d\tau}\{\tilde{\beta}^\ga w^{\ell_\infty}\pa^{\vth'}_\xi\nu\pa^{\vth-\vth'}_{\xi}h_{2}\}(s,V(s))ds
,\notag
\end{align}
\begin{align}
\CJ_{14}=\int_0^te^{-\int_s^t\tilde{\CA}(\tau)d\tau}
\{\tilde{\beta}^\ga w^{\ell_\infty}\pa_\xi^{\vth}( K h_2)\}(s,V(s))ds
,\notag
\end{align}
and
\begin{align}
\CJ_{15}=\int_0^te^{-\int_s^t\tilde{\CA}(\tau)d\tau}
\{\tilde{\beta}^{\ga}w^{\ell_\infty}\pa_{\xi}^\vth((1-\chi_M)\mu^{-\frac{1}{2}}\CK h_1)\}(s,V(s))ds,\notag
\end{align}
where we have denoted
\begin{align}\label{A-def}
\tilde{\CA}(\tau,V(\tau))=&\tilde{\beta}^{\ga}\nu(V(\tau))-\frac{(2\ga+3)\tilde{\beta}'}{\tilde{\beta}}+2{\ell_\infty} \frac{\tilde{\beta}'}{\tilde{\beta}} \frac{|V(\tau)|^2}{1+|V(\tau)|^2} +2{\ell_\infty} \al \frac{V(\tau)\cdot (AV(\tau))}{{1+|V(\tau)|^2}}.
\end{align}
Before computing $\CJ_{i}$$(1\leq i\leq 15)$, we first prove the following estimates.

\begin{lemma}\label{H12-pro}
Let $[h_1,h_2]\in \FY_{\al,T}$ with $0\leq T\leq+\infty$. For any $\ell\geq0$, $k\in\Z^+$, and $N\in\Z^+$, it holds that
\begin{align}
\sum\limits_{|\vth|\leq N+1}\|w^{\ell}\pa_t^k\pa_\xi^\vartheta H_1(t,\xi)\|_{L^\infty}\leq C_1\al^{2k}\tilde{\beta}^{-\ga-k\ga},\label{H1-es-pr0}
\end{align}
and
\begin{align}
\sum\limits_{|\vth|\leq N}\|w^\ell\pa_t^k\pa_\xi^\vartheta H_2(t,\xi)\|_{L^\infty}\leq C_2\al^{2k}\tilde{\beta}^{-2\ga-k\ga},\label{H2-es-pr0}
\end{align}
where both $C_1$ and $C_2$ depend on $\ell, k$ and $N.$ Moreover, if $\tilde{\beta}(0)=1$, it holds that
\begin{align}\label{tbet-eq}
\tilde{\beta}^\ga(t)\sim 1+\ga\vho_0\al^2t,
\end{align}
for any $0\leq t\leq T.$

\end{lemma}
\begin{proof}
In light of \eqref{H1-exp} and ${\rm tr} A=0$, by using the same argument as for obtaining \eqref{H2-es} below, one
can show that for any $\vartheta\in \Z_+^3$ and $\ell\geq0$
\begin{align}
\|w^\ell\pa_\xi^\vartheta H_1(t,\xi)\|_{L^\infty}\leq C\tilde{\beta}^{-\ga},\label{H1-es}
\end{align}
where $C>0$ depends on $\vartheta$ and $\ell.$

For $H_2$, we intend to prove
\begin{align}
\sum\limits_{|\vth|\leq N}\|w^\ell\pa_\xi^\vartheta H_2(t,\xi)\|_{L^\infty}\leq
C\tilde{\beta}^{-2\ga}\label{H2-es}
\end{align}
for any $N\geq0.$ To do this, we first get from \eqref{H2-exp} that
\begin{align}
w^\ell\pa_\xi^\vartheta H_2=&\nu^{-1}w^\ell K\pa_\xi^\vartheta H_2+{\bf 1}_{\vth>0} \sum\limits_{0<\vth'+\vth''\leq \vth}C_\vth^{\vth',\vth''}w^\ell \pa_\xi^{\vth'}(\nu^{-1})(\pa_\xi^{\vth''}K)\pa_\xi^{\vartheta-\vth'-\vth''} H_2
\notag\\&+\sum\limits_{\vth'\leq \vth}C_\vth^{\vth'}w^\ell\pa_\xi^{\vth'}(\nu^{-1})\pa_\xi^{\vartheta-\vth'}\Ga(H_1,H_1)
\notag\\&+\sum\limits_{\vth'\leq \vth}C_\vth^{\vth'}
w^\ell\tilde{\beta}^{-\ga}\pa_\xi^{\vth'}(\nu^{-1})\pa_\xi^{\vartheta-\vth'}
\left\{\tilde{\beta}_0\na_\xi\cdot(\xi\mu)\mu^{-\frac{1}{2}}\right\}
\notag\\&+\sum\limits_{\vth'\leq \vth}C_\vth^{\vth'}
w^\ell\tilde{\beta}^{-\ga}\pa_\xi^{\vth'}(\nu^{-1})\pa_\xi^{\vartheta-\vth'}
\left\{\na_\xi\cdot(A\xi\sqrt{\mu}H_1)\mu^{-\frac{1}{2}}\right\}.\notag
\end{align}
By Lemma \ref{Ga}, \eqref{H1-es} and \eqref{be0h-def}, one further has
\begin{align}\label{H2-ho}
|w^\ell\pa_\xi^\vartheta H_2|\leq &C|w^\ell K\pa_\xi^\vartheta H_2|
+C{\bf 1}_{\vth>0}\sum\limits_{\vth'<\vth}\|w^\ell\pa_\xi^{\vth'} H_2\|_{L^\infty}+C\tilde{\beta}^{-2\ga}.
\end{align}
To handle the first term on the right hand side of \eqref{H2-ho},
we rewrite
\begin{align}
w^\ell K \pa_\xi^\vartheta H_2= \int_{\R^3}\Fk_w(\xi,\xi_\ast)w^\ell \pa_\xi^\vartheta H_2(t,\xi_\ast)d\xi_\ast:=\CI_0,\notag
\end{align}
and then divide our computations in the following three cases.

\medskip
\noindent\underline{{\it Case 1. $|\xi|\geq M$ with $M>0$ suitably large.}}
From Lemma \ref{Kop}, it follows that
$$
\int_{\R^3}\mathbf{k}_w(\xi,\xi_\ast)\,d\xi_\ast\leq \frac{C}{1+|\xi|}\leq \frac{C}{M}.
$$
Using this, one has
\begin{equation}\label{I01}
|\CI_{0}|\leq \int_{\R^3}\mathbf{k}_w(\xi,\xi_\ast)\,d\xi_\ast\|w^\ell \pa_\xi^\vartheta H_2\|_{L^\infty}\leq \frac{C}{M}\|\pa_\xi^\vartheta H_2\|_{L^\infty}.
\end{equation}

\medskip
\noindent\underline{{\it Case 2. $|\xi|\leq M$ and $|\xi_\ast|\geq 2M$.}} In this situation, we have
$|\xi-\xi_\ast|\geq M$, then
\begin{equation*}
\mathbf{k}_w(\xi,\xi_\ast)
\leq Ce^{-\frac{\vps M^2}{8}}\mathbf{k}_w(\xi,\xi_\ast)e^{\frac{\vps |\xi-\xi_\ast|^2}{8}}.
\end{equation*}
In light of Lemma \ref{Kop}, one sees that
$\int\mathbf{k}_w(\xi,\xi_\ast)e^{\frac{\vps |\xi-\xi_\ast|^2}{8}}\,d\xi_\ast$ is still bounded. Thus,
a similar computation as for obtaining \eqref{I01} yields
\begin{equation*}
\begin{split}
|\CI_{0}|\leq Ce^{-\frac{\vps M^2}{8}}\|w^\ell \pa_\xi^\vartheta H_2\|_{L^\infty}.
\end{split}
\end{equation*}

\noindent\underline{{\it Case 3. $|\xi|\leq M$ and $|\xi_\ast|\leq 2M$}}. In this case, we convert the bound in $L^\infty$-norm to the one in $L^2$-norm which will be established later on. To do so, for any large $M>0$,
we choose a number $p=p(M)$ to define
\begin{equation}\label{km}
\mathbf{k}_{w,p}(\xi,\xi_\ast)\equiv \mathbf{1}_{|\xi-\xi_\ast|\geq\frac{1}{p},|\xi_\ast|\leq p}\mathbf{k}_{w}(\xi,\xi_\ast),
\end{equation}
such that $\sup\limits_{\xi}\int_{\mathbf{R}^{3}}|\mathbf{k}_{w,p}(\xi,\xi_\ast)
-\mathbf{k}_{w}(\xi,\xi_\ast)|\,d\xi_\ast\leq
\frac{1}{M}.$ Then it follows
\begin{align*}
|\CI_{0}|&\leq C\int_{|\xi_\ast|\leq 2M}\mathbf{k}_{w,p}(\xi,\xi_\ast)|\pa_\xi^\vartheta H_2|d\xi_\ast+\frac{1}{M}\|w^\ell \pa_\xi^\vartheta H_2\|_{L^\infty}\\
&\leq C_{p,M}\|\pa_\xi^\vartheta H_2\|+\frac{1}{M}\|w^\ell \pa_\xi^\vartheta H_2\|_{L^\infty},
\end{align*}
according to H\"older's inequality and the fact that $\int_{\R^3}\mathbf{k}^2_{w,p}(\xi,\xi_\ast)d\xi_\ast<\infty.$
Putting the calculations above together, we arrive at
\begin{align}
|\CI_{0}|&\leq C\|\pa_\xi^\vartheta H_2\|
+\left\{\frac{C}{M}+Ce^{-\frac{\vps M^2}{8}}\right\}\|w^\ell \pa_\xi^\vartheta H_2\|_{L^\infty}.\label{I0-es}
\end{align}
We now turn to deduce the $L^2$ estimate for $H_2$. In view of \eqref{H2-exp}, one has
\begin{align}\label{d-H2}
-\tilde{\beta}^{-\ga}\tilde{\beta}_0\pa_\xi^\vartheta&\left\{\na_\xi\cdot(\xi\mu)\mu^{-\frac{1}{2}}\right\}
-\pa_\xi^\vartheta\left\{\tilde{\beta}^{-\ga}\na_\xi\cdot(A\xi\sqrt{\mu}H_1)\mu^{-\frac{1}{2}}\right\}+\pa_\xi^\vth L H_2
=\pa_\xi^\vartheta\Ga(H_1,H_1).
\end{align}
Taking the inner product of \eqref{d-H2} with $\pa_\xi^\vth H_2$  and applying \eqref{H1-es} and Lemma \ref{Ga} as well as Lemma \ref{es-L}, one has, if $\vth=0$
\begin{align}
\|H_2\|^2_\nu\leq C\tilde{\beta}^{-\ga}\tilde{\beta}_0\leq C\tilde{\beta}^{-2\ga},\label{H2-l2-p1}
\end{align}
and if $\vth>0$
\begin{align}
\|\pa_\xi^\vth H_2\|^2_\nu\leq C\|H_2\|^2+\tilde{\beta}^{-2\ga}.\label{H2-l2-p2}
\end{align}
We now have by putting \eqref{H2-l2-p1} and \eqref{H2-l2-p2} into \eqref{I0-es} that
\begin{align}\notag
|\CI_{0}|&\leq C\tilde{\beta}^{-2\ga}
+\left\{\frac{C}{M}+Ce^{-\frac{\vps M^2}{8}}\right\}\|w^\ell\pa_\xi^\vartheta H_2\|_{L^\infty},
\end{align}
which together with \eqref{H2-ho} gives
\begin{align}\label{H2-ho-p2}
|w^\ell\pa_\xi^\vartheta H_2|\leq
C{\bf 1}_{\vth>0}\sum\limits_{\vth'<\vth}\|w^\ell \pa_\xi^{\vth'} H_2\|_{L^\infty}+C\tilde{\beta}^{-2\ga}.
\end{align}
Consequently, \eqref{H2-es} follows from a linear combination of \eqref{H2-ho-p2} over $|\vth|=0,1,2,\cdots,N.$
Once \eqref{H1-es} and \eqref{H2-es} are obtained, we now turn to determine $\tilde{\beta}$.
Since $[h_1,h_2]\in\FY_{\al,T}$, from \eqref{h-def}, it follows
\begin{align}\label{H-decay}
\left\|w^{\ell}\pa_\xi^\vth H_{R,1}(t)\right\|_{L^\infty}
+\left\|w^{\ell}\pa_\xi^\vth  H_{R,2}(t)\right\|_{L^\infty}\leq 2M_0\tilde{\beta}^{-2\ga}(t),
\end{align}
for any $0\leq t\leq T.$
Therefore,
using\eqref{be1h-def}, \eqref{H2-es} and \eqref{H-decay} and taking $\al$ to be suitably small, we get
\begin{align}
\al|\tilde{\beta}_1|\leq\frac{1}{2}\vho_0\tilde{\beta}^{-\ga},\label{bet1-es}
\end{align}
where we also have used the fact that $\tilde{\beta}^{-2\ga}\leq C(T)\tilde{\beta}^{-\ga}.$

Next plugging \eqref{bet1-es} and \eqref{be0h-def} into \eqref{beh-exp}, one has
\begin{align}
\frac{1}{2}\vho_0\tilde{\beta}^{-\ga}\al^2\leq \frac{\tilde{\beta}'}{\tilde{\beta}}\leq \frac{3}{2}\vho_0\tilde{\beta}^{-\ga}\al^2,\notag
\end{align}
which further gives \eqref{tbet-eq}.

We now turn to prove \eqref{H1-es-pr0} and \eqref{H2-es-pr0} involving $t-$derivatives. As a matter of fact, since $\tilde{\beta}(t)$ is given as \eqref{tbet-eq}, we see that
\begin{align}
\left|\frac{d^k}{dt^k}(\tilde{\beta}^{-\ga})\right|\leq C\al^{2k}\tilde{\beta}^{-\ga-k\ga}.\notag
\end{align}
On the other hand, it follows from \eqref{H1-exp} and \eqref{H2-exp} that
\begin{align}
w^\ell\pa_t^k\pa_\xi^\vartheta H_1
=\frac{d^k}{dt^k}(\tilde{\beta}^{-\ga})w^\ell\pa_\xi^\vartheta L^{-1}\{\xi\cdot A\xi\mu^{\frac{1}{2}}\},\notag
\end{align}
and
\begin{align}
w^\ell\pa_t^k\pa_\xi^\vartheta H_2=&\nu^{-1}w^\ell K\pa_t^k\pa_\xi^\vartheta H_2+{\bf 1}_{\vth>0} \sum\limits_{0<\vth'+\vth''\leq \vth}C_\vth^{\vth',\vth''}w^\ell \pa_\xi^{\vth'}(\nu^{-1})(\pa_\xi^{\vth''}K)\pa_t^k\pa_\xi^{\vartheta-\vth'-\vth''} H_2
\notag\\&+\sum\limits_{\vth'\leq \vth}C_\vth^{\vth'}\sum\limits_{k'\leq k}C_k^{k'}w^\ell\pa_\xi^{\vth'}(\nu^{-1})\pa_\xi^{\vartheta-\vth'}
\Ga(\pa_t^{k'}H_1,\pa_t^{k-k'}H_1)
\notag\\&+\sum\limits_{\vth'\leq \vth}C_\vth^{\vth'}
w^\ell\frac{d^k}{dt^k}\left(\tilde{\beta}^{-\ga}\tilde{\beta}_0\right)\pa_\xi^{\vth'}(\nu^{-1})\pa_\xi^{\vartheta-\vth'}
\left\{\na_\xi\cdot(\xi\mu)\mu^{-\frac{1}{2}}\right\}
\notag\\&+\sum\limits_{\vth'\leq \vth}C_\vth^{\vth'}\sum\limits_{k'\leq k}C_k^{k'}
w^\ell\frac{d^{k'}}{dt^{k'}}\left(\tilde{\beta}^{-\ga}\right)\pa_\xi^{\vth'}(\nu^{-1})\pa_\xi^{\vartheta-\vth'}
\left\{\na_\xi\cdot(A\xi\sqrt{\mu}\pa_t^{k-k'}H_1)\mu^{-\frac{1}{2}}\right\}.\notag
\end{align}
Thus we can perform similar calculations as for obtaining \eqref{H1-es} and \eqref{H2-es} and then see that
both \eqref{H1-es-pr0} and \eqref{H2-es-pr0} with $k>0$ also hold true. This ends the proof of Lemma \ref{H12-pro}.
\end{proof}

With Lemma \ref{H12-pro} in our hands, we now turn to estimate  $\CJ_{i}$$(1\leq i\leq 16)$ term by term. First of all, by taking ${\ell_\infty}\al\ll1$, we get that
\begin{align}
\tilde{\CA}(\tau,V(\tau))\geq \frac{1}{2}\nu\tilde{\beta}^\ga(\tau)\geq c_0\tilde{\beta}^\ga(\tau),\notag
\end{align}
for some $c_0>0$. Thus it follows
\begin{align}
\int_0^te^{-\int_s^t\tilde{\CA}(\tau)d\tau}(\nu\tilde{\beta}^{\ga})(s)ds\leq 2\left(1-e^{-\frac{1}{2}\int_0^t(\nu\tilde{\beta}^{\ga})(\tau)d\tau}\right)
\leq 2.\label{CA-bd}
\end{align}
Moreover, one also has
\begin{align}
\int_0^te^{-\int_s^t\tilde{\CA}(\tau)d\tau}(\tilde{\beta}^{\ga})(s)ds\leq 2\left(1-e^{-c_0\int_0^t\tilde{\beta}^{\ga}(\tau)d\tau}\right)
\leq \int_0^t\tilde{\beta}^{\ga}(\tau)d\tau\leq t\tilde{\beta}^{\ga}(t)\leq 2T_0,\label{CA-bd-p2}
\end{align}
provided that $t\in[0,T_0]$ and $T_0$ is suitably small.

It is straightforward to see
\begin{align}
|\CJ_1|\leq \|w^{\ell_\infty}\pa_\xi^\vth G_{R,0}(V(0))\|_{L^\infty}\leq M_0.\notag
\end{align}
Since
\begin{align}
|\frac{\tilde{\beta}'}{\tilde{\beta}}|\leq C\al^2\label{bet-de-sec}
\end{align}
according to \eqref{tbet-eq}, one has
\begin{align}
|\CJ_2|\leq C\al^2\tilde{\beta}^{-\ga}\sum\limits_{\vth'\leq\vth}\|w^{\ell_\infty}\pa_\xi^{\vth'} h_1\|_{L^\infty}\leq C\al^2M_0,
\ \ |\CJ_{11}|\leq C\al^2\tilde{\beta}^{-\ga}\sum\limits_{\vth'\leq\vth}\|w^{\ell_\infty}\pa_\xi^{\vth'} h_2\|_{L^\infty}\leq C\al^2M_0.\notag
\end{align}
For $\CJ_3$ and $\CJ_{12}$, it follows from \eqref{CA-bd} that
\begin{align}
|\CJ_3|\leq C\al\tilde{\beta}^{-\ga}\sum\limits_{\vth'\leq\vth}\|w^{\ell_\infty}\pa_\xi^{\vth'} h_1\|_{L^\infty}\leq C\al M_0
,\ \ |\CJ_{12}|\leq C\al\tilde{\beta}^{-\ga}\sum\limits_{\vth'\leq\vth}\|w^{\ell_\infty}\pa_\xi^{\vth'} h_2\|_{L^\infty}\leq C\al M_0.\notag
\end{align}
Similarly, for $\CJ_4$ and $\CJ_{13}$, one has
\begin{align}
|\CJ_4|\leq CT_0\sum\limits_{\vth'<\vth}\|w^{\ell_\infty}\pa_\xi^{\vth'} h_1\|_{L^\infty}\leq CT_0 M_0
,\ \ |\CJ_{13}|\leq CT_0\sum\limits_{\vth'<\vth}\|w^{\ell_\infty}\pa_\xi^{\vth'} h_2\|_{L^\infty}\leq CT_0 M_0.\notag
\end{align}
For $\CJ_5$, Lemma \ref{g-ck-lem} and \eqref{CA-bd} give
\begin{align}
|\CJ_5|\leq C(M^{-\ga}+\varsigma)\sum\limits_{\vth'\leq\vth}\|w^{\ell_\infty}\pa_\xi^{\vth'} h_1\|_{L^\infty}\leq CM_0(M^{-\ga}+\varsigma).\notag
\end{align}
For $\CJ_6$, by virtue of \eqref{bet-de-sec}, \eqref{CA-bd}, we get
\begin{align}
|\CJ_6|\leq C\al\sum\limits_{\vth'\leq\vth}\|w^{\ell_\infty}\pa_\xi^{\vth'} h_2\|_{L^\infty}\leq C\al M_0,\notag
\end{align}
and
\begin{align}
|\CJ_{15}|\leq C\al\sum\limits_{\vth'\leq\vth}\|w^{\ell_\infty}\pa_\xi^{\vth'} h_2\|_{L^\infty}\leq C\al M_0.\notag
\end{align}
For $\CJ_7$ and $\CJ_8$, Lemma \ref{H12-pro}, \eqref{bet-de-sec} and \eqref{CA-bd} imply
\begin{align}
|\CJ_7|+|\CJ_8|\leq C\al^{3-m}+ C\al\sum\limits_{\vth'\leq\vth}\|w^{\ell_\infty}\pa_\xi^{\vth'}[h_1, h_2]\|_{L^\infty}
\leq C\al^{3-m}+ C\al M_0.\notag
\end{align}
For the rest non-local terms, from Lemmas \ref{op.es.lem}, \ref{Kop} and \ref{H12-pro} as well as \eqref{CA-bd} and \eqref{CA-bd-p2}, one has
\begin{align}
|\CJ_9|\leq C\al^{3-m}\sum\limits_{\vth'\leq\vth}\|\tilde{\beta}^\ga w^{\ell_\infty}\pa_\xi^{\vth'} [H_1,H_2]\|_{L^\infty}\leq C\al^{3-m},\notag
\end{align}
\begin{align}
|\CJ_{10}|\leq C\al^m\sum\limits_{\vth'\leq\vth}\|w^{\ell_\infty}\pa_\xi^{\vth'} [h_1,h_2]\|^2_{L^\infty}
+ C\al\sum\limits_{\vth'\leq\vth}\|w^{\ell_\infty}\pa_\xi^{\vth'} [h_1,h_2]\|_{L^\infty}\leq C\al M_0,\notag
\end{align}
\begin{align}
|\CJ_{14}|\leq CT_0\sum\limits_{\vth'\leq\vth}\|w^{\ell_\infty}\pa_\xi^{\vth'}h_2\|_{L^\infty}\leq CT_0 M_0,\ \
|\CJ_{15}|\leq CT_0\sum\limits_{\vth'\leq\vth}\|w^{\ell_\infty}\pa_\xi^{\vth'} h_1\|_{L^\infty}\leq CT_0 M_0.\notag
\end{align}
Putting the above estimates together, we arrive at
\begin{align}\label{g1-g2-if-sum}
\sum\limits_{|\vth|\leq N}&\left\{\|w^{\ell_\infty}\pa_\xi^{\vth} g_1\|_{L^\infty}+\|w^{\ell_\infty}\pa_\xi^{\vth} g_2\|_{L^\infty}\right\}
\notag\\&\leq M_0+CN\al M_0+CNT_0 M_0+CN\al^{3-m}+CN\al^2+CN M_0(M^{-\ga}+\varsigma).
\end{align}
Choosing now both $\al$, $\varsigma$ and $T_0$ to be suitably small and letting $M$ be sufficiently large such that
\begin{align}
\si_0=\max\left\{CN\al,CNT_0,CN(M^{-\ga}+\varsigma),\frac{CN\al^{3-m}}{ M_0},\frac{CN\al^2}{ M_0}\right\}\leq\frac{1}{16},\notag
\end{align}
we get from  \eqref{g1-g2-if-sum} that
\begin{align}
\sum\limits_{|\vth|\leq N}&\sup\limits_{0\leq t\leq T_0}\left\{\|w^{\ell_\infty}\pa_\xi^{\vth} g_1(t)\|_{L^\infty}+\al\|w^{\ell_\infty}\pa_\xi^{\vth} g_2(t)\|_{L^\infty}\right\}
\leq 2 M_0.\notag
\end{align}
Next, for $[h_1,h_2]\in\FY_{\al,T}$ and
$[\bar{h}_1,\bar{h}_2]\in\FY_{\al,T}$, we aim to prove that
\begin{multline}
\sum\limits_{|\vth|\leq N-1}\sup\limits_{0\leq t\leq T_0}\left\{\|[\CQ_{1}(h_1,h_2)-\CQ_{1}(\bar{h}_1,\bar{h}_2)]\|_{L^\infty}
+\|[\CQ_{2}(h_1,h_2)-\CQ_{1}(\bar{h}_1,\bar{h}_2)]\|_{L^\infty}\right\}\\
\leq 
C\{T_0+\al+M^{-\ga}+\varsigma\} \sum\limits_{|\vth|\leq N-1}\sup\limits_{0\leq t\leq T_0}\left\{\|w^{\ell_\infty}\pa_\xi^{\vth} [h_1-\bar{h}_1]\|_{L^\infty}+\|w^{\ell_\infty}\pa_\xi^{\vth} [h_2-\bar{h}_2]\|_{L^\infty}\right\}.\label{CQ-c}
\end{multline}
To do this, as in the proof of Lemma \ref{H12-pro}, we first denote
\begin{align}\notag
\sqrt{\mu}\bar{h}=\bar{h}_1+\sqrt{\mu}\bar{h}_2=\bar{\beta}^{2\ga}(\bar{H}_{R,1}+\sqrt{\mu}\bar{H}_{R,2}),\ \sqrt{\mu}\bar{H}_{R}=\bar{H}_{R,1}+\sqrt{\mu}\bar{H}_{R,2},
\end{align}
\begin{align}
\bar{H}=\mu+\sqrt{\mu}\{\al \bar{H}_1+\al^2 \bar{H}_2+\al^m \bar{H}_R\},\notag
\end{align}
\begin{align}\notag
\bar{H}_1=-\bar{\beta}^{-\ga}L^{-1}\{\xi\cdot A\xi\mu^{\frac{1}{2}}\},
\end{align}
\begin{align}\notag
\bar{H}_2=&\bar{\beta}^{-\ga}L^{-1}\left\{\bar{\beta}^\ga\Ga(\bar{H}_1,\bar{H}_1)+\bar{\beta}_0\na_\xi\cdot(\xi\mu)\mu^{-\frac{1}{2}}
+\na_\xi\cdot(A\xi\sqrt{\mu}\bar{H}_1)\mu^{-\frac{1}{2}}\right\},
\end{align}
\begin{align}
\frac{\bar{\beta}'}{\bar{\beta}}=-\frac{\al}{3}\int_{\R^3}\xi\cdot A\xi \bar{H}d\xi,\ \bar{\beta}'=\frac{d\bar{\beta}}{dt},\ \bar{\beta}(0)=1,\label{bet-bH}
\end{align}
\begin{align}\notag
\frac{\bar{\beta}'}{\bar{\beta}}=\bar{\beta}_0\al^2+\bar{\beta}_1\al^3,
\end{align}
\begin{align}\notag
\bar{\beta}_0=-\frac{1}{3}\int_{\R^3}\xi\cdot A\xi \sqrt{\mu} \bar{H}_1d\xi=\vho_0\bar{\beta}^{-\ga},
\end{align}
\begin{align}\label{be1bh-def}
\bar{\beta}_1=-\frac{1}{3}\int_{\R^3}\xi\cdot A\xi \sqrt{\mu}\bar{H}_2d\xi-\frac{\al^{m-2}}{3}\int_{\R^3}\xi\cdot A\xi \sqrt{\mu}\bar{H}_Rd\xi
=\vho_1\bar{\beta}^{-2\ga}+\bar{\vho}_R\bar{\beta}^{-2\ga},
\end{align}
with
\begin{align}
\bar{\vho}_R=-\frac{\al^{m-2}\bar{\beta}^{2\ga}}{3}\int_{\R^3}\xi\cdot A\xi \sqrt{\mu}\bar{H}_Rd\xi,\notag
\end{align}
and similar to \eqref{tbeta-ga}
\begin{align}
\frac{d\bar{\beta}^{\ga}}{dt}=\ga\left\{\vho_0\al^2+\vho_1\al^3\bar{\beta}^{-\ga}+\bar{\vho}_R\al^3\bar{\beta}^{-\ga}\right\}.\label{bbeta-ga}
\end{align}
Moreover, $[\bar{H}_1,\bar{H}_2,\bar{\beta}]$ also satisfies Lemma \ref{H12-pro}. Namely, we have the following lemma.

\begin{lemma}\label{bH12-pro}
Let $[\bar{h}_1,\bar{h}_2]\in \FY_{\al,T}$ with $0\leq T\leq+\infty$. For any $\ell\geq0$, $k\in\Z^+$, and $N\in\Z^+$, it holds that
\begin{align}\notag
\sum\limits_{|\vth|\leq N+1}\|w^\ell\pa_t^k\pa_\xi^\vartheta \bar{H}_1(t,\xi)\|_{L^\infty}\leq C_1\al^{2k}\bar{\beta}^{-\ga-k\ga},
\end{align}
and
\begin{align}\notag
\sum\limits_{|\vth|\leq N}\|w^\ell\pa_t^k\pa_\xi^\vartheta \bar{H}_2(t,\xi)\|_{L^\infty}\leq C_2\al^{2k}\bar{\beta}^{-2\ga-k\ga},
\end{align}
where both $C_1$ and $C_2$ depend on $\ell, k$ and $N.$ Moreover, if $\bar{\beta}(0)=1$, it holds that
\begin{align}\label{bbet-eq}
\bar{\beta}^\ga(t)\sim 1+\ga\vho_0\al^2t,
\end{align}
for any $0\leq t\leq T.$

\end{lemma}
Using Lemmas \ref{H12-pro} and \ref{bH12-pro}, given $[\bar{h}_1,\bar{h}_2]\in \FY_{\al,T_0}$ one gets from ODE equations
\eqref{tbeta-ga} and \eqref{bbeta-ga} that
\begin{align}
|(\tilde{\beta}^\ga-\bar{\beta}^\ga)(t)|\leq C(T_0)\al |\vho_R-\bar{\vho}_R|\leq C(T_0)\al \|w^{\ell_\infty}[h_1-\bar{h}_1,h_2-\bar{h}_2](t)\|_{L^\infty},\label{diff-bet}
\end{align}
for $\ell_\infty>5/2$. This together with \eqref{bet-H} and \eqref{bet-bH} further gives
\begin{align}
|(\frac{\tilde{\beta}'}{\tilde{\beta}}-\frac{\bar{\beta}'}{\bar{\beta}})(t)|\leq C(T_0)\al \|w^{\ell_\infty}[h_1-\bar{h}_1,h_2-\bar{h}_2](t)\|_{L^\infty}.
\label{diff-bet-d}
\end{align}
In addition, \eqref{be1h-def} and \eqref{be1bh-def} directly imply
\begin{align}
|(\tilde{\beta}^{2\ga}\tilde{\beta}_1-\bar{\beta}^{2\ga}\bar{\beta}_1)(t)|\leq C(T_0)\al |\vho_R-\bar{\vho}_R|\leq C(T_0)\al \|w^{\ell_\infty}[h_1-\bar{h}_1,h_2-\bar{h}_2](t)\|_{L^\infty}.\notag
\end{align}
Next, as in \eqref{A-def}, we define
\begin{align}
\bar{\CA}(\tau,\bar{V}(\tau))=&\bar{\beta}^{\ga}\nu(\bar{V}(\tau))-\frac{(2\ga+3)\bar{\beta}'}{\bar{\beta}}+2{\ell_\infty} \frac{\bar{\beta}'}{\bar{\beta}} \frac{|\bar{V}(\tau)|^2}{1+|\bar{V}(\tau)|^2} +2{\ell_\infty} \al \frac{\bar{V}(\tau)\cdot (A\bar{V}(\tau))}{{1+|\bar{V}(\tau)|^2}},\notag
\end{align}
where $\bar{V}$ is given by
\begin{equation}\label{bV}
\bar{V}(s)=\bar{V}(s;t,\xi)=\frac{\bar{\beta}(t)}{\bar{\beta}(s)}e^{-(s-t)\al A}\xi,
\end{equation}
according to \eqref{CL}. Note that $\bar{V}(0)=V(0)$, since $\tilde{\beta}(0)=\bar{\beta}(0).$ Furthermore, for $0\leq s\leq t\leq T_0$, it follows from \eqref{bV}, \eqref{V} and \eqref{diff-bet} that
\begin{align}\label{VV}
|\bar{V}(s)-V(s)|\leq C(T_0)\al\|w^{\ell_\infty}[h_1-\bar{h}_1,h_2-\bar{h}_2](t)\|_{L^\infty}|\bar{V}(s)|.
\end{align}
Next, in view of \eqref{tbet-eq}, \eqref{bbet-eq} and \eqref{diff-bet}, one has
\begin{align}\label{AA-lb}
|\tilde{\CA}(\tau,V(\tau))-\bar{\CA}(\tau,\bar{V}(\tau))|
\leq& C(T_0)\al \|w^{\ell_\infty}[h_1-\bar{h}_1,h_2-\bar{h}_2](t)\|_{L^\infty}\nu(\bar{V}(\tau)),
\end{align}
where the following fact has been used
\begin{align}
|\nu(\bar{V}(\tau))-\nu(V(\tau))|
\leq  C(T_0)\al \|w^{\ell_\infty}[h_1-\bar{h}_1,h_2-\bar{h}_2](t)\|_{L^\infty}\nu(\bar{V}(\tau))\notag
\end{align}
by mean value theorem. Consequently, by choosing both $\al$ and $\vps_0$ are suitably small, one gets
\begin{align}\label{VV-eq}
\frac{1}{C}|\bar{V}(s)|\leq |V(s)|\leq C|\bar{V}(s)|,\ \frac{1}{C}\nu(\bar{V}(s))\leq \nu(V(s))\leq C\nu(\bar{V}(s)),
\end{align}
and
\begin{align}\label{AA-eq}
\tilde{\CA}(s,V(s))\geq c_0\nu(\bar{V}(\tau)),\ \ \bar{\CA}(s,\bar{V}(s))\geq c_0\nu(\bar{V}(\tau)),
\end{align}
for $0\leq s\leq T_0$ and some $C>0$ as well as $c_0>0$.

Recalling \eqref{Q12}, we have
\begin{align*}
|\CQ_{1,\vth}(h_1,h_2)-\CQ_{1,\vth}(\bar{h}_1,\bar{h}_2)|\leq \sum\limits_{1\leq i\leq 10}|\CJ_i-\bar{\CJ}_i|,\ \
|\CQ_{2,\vth}(h_1,h_2)-\CQ_{2,\vth}(\bar{h}_1,\bar{h}_2)|\leq \sum\limits_{11\leq i\leq 15}|\CJ_i-\bar{\CJ}_i|,
\end{align*}
where $\bar{\CJ}_i$ $(1\leq i\leq15)$ is in the form of $\CJ_i$ with $[h_1,h_2,\tilde{\beta}]$ replaced by $[\bar{h}_1,\bar{h}_2,\bar{\beta}]$.

To prove \eqref{CQ-c}, we now turn to compute $\CJ_i-\bar{\CJ}_i$ $(1\leq i\leq15)$ term by term.
First of all, by mean value theorem, we have
\begin{align}
\left|e^{-\int_0^t\tilde{\CA}(s)ds}-e^{-\int_0^t\bar{\CA}(s)ds}\right|\leq e^{-\Theta}\int_0^t|\tilde{\CA}(s)-\bar{\CA}(s)|ds,\notag
\end{align}
where we have taken $\Theta$ such that $\min\{\int_0^t\tilde{\CA}(s)ds,\int_0^t\bar{\CA}(s)ds\}\leq \Theta\leq \max\{\int_0^t\tilde{\CA}(s)ds,\int_0^t\bar{\CA}(s)ds\}$. Then, 
 \eqref{AA-lb} and \eqref{AA-eq} further give
\begin{align}
&\left|e^{-\int_0^t\tilde{\CA}(s)ds}-e^{-\int_0^t\bar{\CA}(s)ds}\right|\notag\\
&\leq
C(T_0) \al e^{-c_0\int_0^t\nu(\bar{V}(\tau))d\tau}\int_0^t\nu(\bar{V}(\tau))d\tau
\sup\limits_{0\leq t\leq T_0}\|w^{\ell_\infty}[h_1-\bar{h}_1,h_2-\bar{h}_2](t)\|_{L^\infty}\notag\\
&\leq C(T_0)\al \sup\limits_{0\leq t\leq T_0}\|w^{\ell_\infty}[h_1-\bar{h}_1,h_2-\bar{h}_2](t)\|_{L^\infty}.\label{d-AA}
\end{align}
Since $\bar{V}(0)=V(0)$, one has by \eqref{d-AA}
\begin{align}
|\CJ_1-\bar{\CJ}_1|=&\left|e^{-\int_0^t\tilde{\CA}(s)ds}-e^{-\int_0^t\bar{\CA}(s)ds}\right||w^{\ell_\infty}\pa_\xi^\vth G_{R,0}(V(0))|\notag\\
\leq&C(T_0) M_0\al\sup\limits_{0\leq t\leq T_0}\|w^{\ell_\infty}[h_1-\bar{h}_1,h_2-\bar{h}_2](t)\|_{L^\infty}.\notag
\end{align}
For $\CJ_2-\bar{\CJ}_2$, we first write
\begin{align}
\CJ_2-\bar{\CJ}_2=& {\bf 1}_{|\vth|>0}\sum\limits_{|\vth'|=1}C_\vth^{\vth'}\int_0^te^{-\int_s^t\tilde{\CA}(\tau)d\tau}
\left\{\frac{\tilde{\beta}'}{\tilde{\beta}}
w^{\ell_\infty}\pa_\xi^{\vth'}\xi\cdot\na_\xi\pa_\xi^{\vth-\vth'}h_1\right\}(s,V(s))ds\notag\\
&-{\bf 1}_{|\vth|>0}C_\vth^{\vth'}\sum\limits_{|\vth'|=1}\int_0^te^{-\int_s^t\bar{\CA}(\tau)d\tau}
\left\{\frac{\bar{\beta}'}{\bar{\beta}}
w^{\ell_\infty}\pa_\xi^{\vth'}\xi\cdot\na_\xi\pa_\xi^{\vth-\vth'}\bar{h}_1\right\}(s,\bar{V}(s))ds\notag\\
=& {\bf 1}_{|\vth|>0}\sum\limits_{|\vth'|=1}C_\vth^{\vth'}\left\{\int_0^te^{-\int_s^t\tilde{\CA}(\tau)d\tau}
-\int_0^te^{-\int_s^t\bar{\CA}(\tau)d\tau}\right\}
\left\{\frac{\tilde{\beta}'}{\tilde{\beta}} w^{\ell_\infty}\pa_\xi^{\vth'}\xi\cdot\na_\xi\pa_\xi^{\vth-\vth'}h_1\right\}(s,V(s))ds\notag\\
&+{\bf 1}_{|\vth|>0}\sum\limits_{|\vth'|=1}C_\vth^{\vth'}\int_0^te^{-\int_s^t\bar{\CA}(\tau)d\tau}
\left\{\frac{\tilde{\beta}'}{\tilde{\beta}} -\frac{\bar{\beta}'}{\bar{\beta}}\right\}\left\{w^{\ell_\infty}\pa_\xi^{\vth'}\xi\cdot\na_\xi\pa_\xi^{\vth-\vth'}h_1\right\}(s,V(s))ds\notag\\
&+{\bf 1}_{|\vth|>0}\sum\limits_{|\vth'|=1}C_\vth^{\vth'}\int_0^te^{-\int_s^t\bar{\CA}(\tau)d\tau}
\frac{\bar{\beta}'}{\bar{\beta}} \bigg\{\left\{w^{\ell_\infty}\pa_\xi^{\vth'}\xi\cdot\na_\xi\pa_\xi^{\vth-\vth'}h_1\right\}(s,V(s))\notag\\&\qquad\qquad\qquad\qquad-
\left\{w^{\ell_\infty}\pa_\xi^{\vth'}\xi\cdot\na_\xi\pa_\xi^{\vth-\vth'}\bar{h}_1\right\}(s,\bar{V}(s))\bigg\}ds.\label{CJ-22}
\end{align}
On the other hand, it follows
\begin{align}
\{w^{\ell_\infty}&\pa_\xi^{\vth'}\xi\cdot\na_\xi\pa_\xi^{\vth-\vth'}h_1\}(s,V(s))-
\{w^{\ell_\infty}\pa_\xi^{\vth'}\xi\cdot\na_\xi\pa_\xi^{\vth-\vth'}\bar{h}_1\}(s,\bar{V}(s))\notag\\
=&\{w^{\ell_\infty}\pa_\xi^{\vth'}\xi\cdot\na_\xi\pa_\xi^{\vth-\vth'}h_1\}(s,V(s))-
\{w^{\ell_\infty}\pa_\xi^{\vth'}\xi\cdot\na_\xi\pa_\xi^{\vth-\vth'}h_1\}(s,\bar{V}(s))
\notag\\&+\{w^{\ell_\infty}\pa_\xi^{\vth'}\xi\cdot\na_\xi\pa_\xi^{\vth-\vth'}h_1\}(s,\bar{V}(s))-
\{w^{\ell_\infty}\pa_\xi^{\vth'}\xi\cdot\na_\xi\pa_\xi^{\vth-\vth'}\bar{h}_1\}(s,\bar{V}(s))\notag\\
=&\left\{\na_\xi\{w^{\ell_\infty}\pa_\xi^{\vth'}\xi\cdot\na_\xi\pa_\xi^{\vth-\vth'}h_1\}\right\}(s,\Theta_1)\cdot(V(s)-\bar{V}(s))
\notag\\&+\{w^{\ell_\infty}\pa_\xi^{\vth'}\xi\cdot\na_\xi\pa_\xi^{\vth-\vth'}h_1\}(s,\bar{V}(s))-
\{w^{\ell_\infty}\pa_\xi^{\vth'}\xi\cdot\na_\xi\pa_\xi^{\vth-\vth'}\bar{h}_1\}(s,\bar{V}(s)),\notag
\end{align}
where $\Theta_1=\bar{V}(s)+\sigma_1(V(s)-\bar{V}(s))$ with $\si_1\in[0,1]$. Since $[h_1,h_2]\in\FY_{\al,T_0}$ and $|\vth|\leq N-1$ here, one has by \eqref{VV} that
\begin{align}
\left|\left\{\na_\xi\{w^{\ell_\infty}\pa_\xi^{\vth'}\xi\cdot\na_\xi\pa_\xi^{\vth-\vth'}h_1\}\right\}(s,\Theta_1)\cdot(V(s)-\bar{V}(s))\right|
\leq C(T_0) M_0\al\|w^{\ell_\infty}[h_1-\bar{h}_1,h_2-\bar{h}_2](t)\|_{L^\infty}.\notag
\end{align}
Consequently, we get
\begin{align}
\Big|\{w^{\ell_\infty}&\pa_\xi^{\vth'}\xi\cdot\na_\xi\pa_\xi^{\vth-\vth'}h_1\}(s,V(s))-
\{w^{\ell_\infty}\pa_\xi^{\vth'}\xi\cdot\na_\xi\pa_\xi^{\vth-\vth'}\bar{h}_1\}(s,\bar{V}(s))\Big|\notag\\
\leq& C\sum\limits_{\vth'\leq\vth}\|w^{\ell_\infty}\pa_\xi^{\vth'}[h_1-\bar{h}_1,h_2-\bar{h}_2](t)\|_{L^\infty}.\label{h-var}
\end{align}
Therefore, \eqref{diff-bet-d}, \eqref{d-AA}, \eqref{h-var} and \eqref{CJ-22} imply 
\begin{align}
|\CJ_2-\bar{\CJ}_2|\leq C\al\sum\limits_{\vth'\leq\vth}\|w^{\ell_\infty}\pa_\xi^{\vth'}[h_1-\bar{h}_1,h_2-\bar{h}_2](t)\|_{L^\infty}.\notag
\end{align}
In what follows, we only compute the nonlocal terms $\CJ_5-\bar{\CJ}_5$ and $\CJ_{10}-\bar{\CJ}_{10}$, since the other terms can be treated similarly. For $\CJ_5-\bar{\CJ}_5$, as \eqref{CJ-22}, we first write
\begin{align}
\CJ_5-\bar{\CJ}_5=&\sum\limits_{\vth'\leq\vth}C_\vth^{\vth'}
\left\{\int_0^te^{-\int_s^t\tilde{\CA}(\tau)d\tau}-\int_0^te^{-\int_s^t\bar{\CA}(\tau)d\tau}\right\}\left\{\tilde{\beta}^{\ga} w^{\ell_\infty}\pa_\xi^{\vth'}(\chi_M \CK)\pa_\xi^{\vth-\vth'}h_1\right\}(s,V(s))ds\notag\\
&+\sum\limits_{\vth'\leq\vth}C_\vth^{\vth'}
\int_0^te^{-\int_s^t\bar{\CA}(\tau)d\tau}\left\{\tilde{\beta}^{\ga}-\bar{\beta}^{\ga}\right\}\left\{ w^{\ell_\infty}\pa_\xi^{\vth'}(\chi_M \CK)\pa_\xi^{\vth-\vth'}h_1\right\}(s,V(s))ds
\notag\\
&+\sum\limits_{\vth'\leq\vth}C_\vth^{\vth'}
\int_0^te^{-\int_s^t\bar{\CA}(\tau)d\tau}\bar{\beta}^{\ga}\Bigg\{\left\{ w^{\ell_\infty}\pa_\xi^{\vth'}(\chi_M \CK)\pa_\xi^{\vth-\vth'}h_1\right\}(s,V(s))\notag\\&\qquad\qquad\qquad-\left\{ w^{\ell_\infty}\pa_\xi^{\vth'}(\chi_M \CK)\pa_\xi^{\vth-\vth'}\bar{h}_1\right\}(s,\bar{V}(s))\Bigg\}ds,\label{CJ-55}
\end{align}
and furthermore
\begin{align}
&\left\{ w^{\ell_\infty}\pa_\xi^{\vth'}(\chi_M \CK)\pa_\xi^{\vth-\vth'}h_1\right\}(s,V(s))-\left\{ w^{\ell_\infty}\pa_\xi^{\vth'}(\chi_M \CK)\pa_\xi^{\vth-\vth'}\bar{h}_1\right\}(s,\bar{V}(s))\notag
\\&\quad=\left\{ w^{\ell_\infty}\pa_\xi^{\vth'}(\chi_M \CK)\pa_\xi^{\vth-\vth'}h_1\right\}(s,V(s))-\left\{ w^{\ell_\infty}\pa_\xi^{\vth'}(\chi_M \CK)\pa_\xi^{\vth-\vth'}h_1\right\}(s,\bar{V}(s))
\notag
\\&\quad\quad+\left\{ w^{\ell_\infty}\pa_\xi^{\vth'}(\chi_M \CK)\pa_\xi^{\vth-\vth'}h_1\right\}(s,\bar{V}(s))-\left\{ w^{\ell_\infty}\pa_\xi^{\vth'}(\chi_M \CK)\pa_\xi^{\vth-\vth'}\bar{h}_1\right\}(s,\bar{V}(s)).\label{CK-d}
\end{align}
Next, by mean value theorem, Lemma \ref{g-ck-lem} and \eqref{VV} as well as \eqref{VV-eq}, one has for $|\vth|\leq N-1$
\begin{align}\label{K-diff}
&\nu^{-1}(\bar{V})\left|\left\{ w^{\ell_\infty}\pa_\xi^{\vth'}(\chi_M \CK)\pa_\xi^{\vth-\vth'}h_1\right\}(s,V(s))-\left\{ w^{\ell_\infty}\pa_\xi^{\vth'}(\chi_M \CK)\pa_\xi^{\vth-\vth'}h_1\right\}(s,\bar{V}(s))\right|
\notag
\\&\quad=\nu^{-1}(\bar{V})\left|\na_\xi \left\{w^{\ell_\infty}\pa_\xi^{\vth'}(\chi_M \CK)\pa_\xi^{\vth-\vth'}h_1\right\}(s,\Theta_2)\cdot(V(s)-\bar{V}(s))\right|
\notag
\\&\quad\leq C\vps_0\al\|w^{\ell_\infty}[h_1-\bar{h}_1,h_2-\bar{h}_2](t)\|_{L^\infty},
\end{align}
where $\Theta_2=\bar{V}(s)+\sigma_2(V(s)-\bar{V}(s))$ with $\si_2\in[0,1]$ and $[h_1,h_2]\in\FY_{\al,T_0}$ has been used. And Lemma \ref{g-ck-lem} gives
\begin{align}
\nu^{-1}(\bar{V})&\left|\left\{ w^{\ell_\infty}\pa_\xi^{\vth'}(\chi_M \CK)\pa_\xi^{\vth-\vth'}h_1\right\}(s,\bar{V}(s))-\left\{ w^{\ell_\infty}\pa_\xi^{\vth'}(\chi_M \CK)\pa_\xi^{\vth-\vth'}\bar{h}_1\right\}(s,\bar{V}(s))\right|\notag\\
\quad&\leq C((1+M)^{-\ga}+\varsigma)\sum\limits_{\vth'\leq\vth}
\|w^{\ell_\infty}\pa_\xi^{\vth'}[h_1-\bar{h}_1,h_2-\bar{h}_2](t)\|_{L^\infty}.\label{CK-d-1}
\end{align}
Moreover, it can be directly verified that
\begin{align}
\int_0^te^{-\int_s^t\bar{\CA}(\tau)d\tau}\bar{\beta}^{\ga}\nu(\bar{V}(t))dt<+\infty.\label{Frv}
\end{align}
Consequently, when $|\vth|\leq N-1$, we get from
\eqref{diff-bet}, \eqref{d-AA},  \eqref{CJ-55}, \eqref{CK-d}, \eqref{K-diff}, \eqref{CK-d-1}  and \eqref{Frv} that
\begin{align}
|\CJ_5-\bar{\CJ}_5|\leq C(\al+(1+M)^{-\ga}+\varsigma)\sum\limits_{\vth'\leq\vth}\|w^\ell\pa_\xi^{\vth'}[h_1-\bar{h}_1,h_2-\bar{h}_2](t)\|_{L^\infty}.\notag
\end{align}
Likewise, for $\CJ_{10}-\bar{\CJ}_{10}$, we first have
\begin{align}
\CJ_{10}-\bar{\CJ}_{10}=&\int_0^te^{-\int_s^t\tilde{\CA}(\tau)d\tau}\Big\{\al^m\tilde{\beta}^{-\ga} w^{\ell_\infty} \pa_\xi^{\vth}Q(\sqrt{\mu}h,\sqrt{\mu}h)
\notag\\&\qquad+\al \tilde{\beta}^{\ga}w^{\ell_\infty}\pa_\xi^{\vth}\{Q(\sqrt{\mu}H_1+\al \sqrt{\mu}H_2,\sqrt{\mu}h)
+Q(\sqrt{\mu}h,\sqrt{\mu}H_1+\al \sqrt{\mu}H_2)\}\Big\}(s,V(s))ds,\notag\\
&-\int_0^te^{-\int_s^t\bar{\CA}(\tau)d\tau}\Big\{\al^m\bar{\beta}^{-\ga} w^{\ell_\infty} \pa_\xi^{\vth}Q(\sqrt{\mu}\bar{h},\sqrt{\mu}\bar{h})
\notag\\&\qquad+\al \bar{\beta}^{\ga}w^{\ell_\infty}\pa_\xi^{\vth}\{Q(\sqrt{\mu}\bar{H}_1+\al \sqrt{\mu}\bar{H}_2,\sqrt{\mu}\bar{h})
+Q(\sqrt{\mu}\bar{h},\sqrt{\mu}\bar{H}_1+\al \sqrt{\mu}\bar{H}_2)\}\Big\}(s,V(s))ds.\notag
\end{align}
Then, we write
\begin{align}
\big\{w^{\ell_\infty}& \pa_\xi^{\vth}Q(\sqrt{\mu}h,\sqrt{\mu}h)\big\}(V(s))
-\big\{w^{\ell_\infty} \pa_\xi^{\vth}Q(\sqrt{\mu}\bar{h},\sqrt{\mu}\bar{h})\big\}(\bar{V}(s))\notag\\
=&\big\{w^{\ell_\infty} \pa_\xi^{\vth}Q(\sqrt{\mu}h,\sqrt{\mu}h)\big\}(V(s))
-\big\{w^{\ell_\infty} \pa_\xi^{\vth}Q(\sqrt{\mu}\bar{h},\sqrt{\mu}h)\big\}(V(s))\notag\\
&+\big\{w^{\ell_\infty} \pa_\xi^{\vth}Q(\sqrt{\mu}\bar{h},\sqrt{\mu}h)\big\}(V(s))-\big\{w^{\ell_\infty} \pa_\xi^{\vth}Q(\sqrt{\mu}\bar{h},\sqrt{\mu}h)\big\}(\bar{V}(s))\notag\\
&+\big\{w^{\ell_\infty} \pa_\xi^{\vth}Q(\sqrt{\mu}\bar{h},\sqrt{\mu}h)\big\}(\bar{V}(s))
-\big\{w^{\ell_\infty} \pa_\xi^{\vth}Q(\sqrt{\mu}\bar{h},\sqrt{\mu}\bar{h})\big\}(\bar{V}(s)).\notag
\end{align}
Next, Lemma \ref{op.es.lem} gives
\begin{align}
\nu^{-1}(\bar{V})&\left|\big\{w^{\ell_\infty} \pa_\xi^{\vth}Q(\sqrt{\mu}h,\sqrt{\mu}h)\big\}(V(s))
-\big\{w^{\ell_\infty} \pa_\xi^{\vth}Q(\sqrt{\mu}\bar{h},\sqrt{\mu}h)\big\}(V(s))\right|\notag\\
+&\nu^{-1}(\bar{V})\left|\big\{w^{\ell_\infty} \pa_\xi^{\vth}Q(\sqrt{\mu}\bar{h},\sqrt{\mu}h)\big\}(\bar{V}(s))
-\big\{w^{\ell_\infty} \pa_\xi^{\vth}Q(\sqrt{\mu}\bar{h},\sqrt{\mu}\bar{h})\big\}(\bar{V}(s))\right|\notag\\
\leq& CM_0\sum\limits_{\vth'\leq\vth}\|w^{\ell_\infty}\pa_\xi^{\vth'}[h_1-\bar{h}_1,h_2-\bar{h}_2](t)\|_{L^\infty}.
\notag
\end{align}
Then Lemma \ref{op.es.lem} together with mean value theorem and \eqref{VV} yields that for $|\vth|\leq N-1$,
\begin{align}
\nu^{-1}(\bar{V})&\left|\big\{w^{\ell_\infty} \pa_\xi^{\vth}Q(\sqrt{\mu}\bar{h},\sqrt{\mu}h)\big\}(V(s))-\big\{w^{\ell_\infty} \pa_\xi^{\vth}Q(\sqrt{\mu}\bar{h},\sqrt{\mu}h)\big\}(\bar{V}(s))\right|\notag\\
\leq& CM_0\sum\limits_{\vth'\leq\vth}\|w^{\ell_\infty}\pa_\xi^{\vth'}[h_1-\bar{h}_1,h_2-\bar{h}_2](t)\|_{L^\infty}.\notag
\end{align}
Hence it follows
\begin{align}
\sup\limits_{0\leq t\leq T_0}&\nu^{-1}(\bar{V})\left|\big\{w^{\ell_\infty} \pa_\xi^{\vth}Q(\sqrt{\mu}h,\sqrt{\mu}h)\big\}(V(t))
-\big\{w^{\ell_\infty} \pa_\xi^{\vth}Q(\sqrt{\mu}\bar{h},\sqrt{\mu}\bar{h})\big\}(\bar{V}(t))\right|
\notag\\
\leq& CM_0\sup\limits_{0\leq t\leq T_0}\sum\limits_{\vth'\leq\vth}\|w^{\ell_\infty}\pa_\xi^{\vth'}[h_1-\bar{h}_1,h_2-\bar{h}_2](t)\|_{L^\infty}.\notag
\end{align}
Moreover, one can directly show that
\begin{align}
\sup\limits_{0\leq t\leq T_0}&\nu^{-1}(\bar{V})\Bigg|w^{\ell_\infty}\pa_\xi^{\vth}\{Q(\sqrt{\mu}\bar{H}_1+\al \sqrt{\mu}\bar{H}_2,\sqrt{\mu}\bar{h})
+Q(\sqrt{\mu}\bar{h},\sqrt{\mu}\bar{H}_1+\al \sqrt{\mu}\bar{H}_2)\}\Big\}(t,V(t))\notag\\
&-w^{\ell_\infty}\pa_\xi^{\vth}\{Q(\sqrt{\mu}\bar{H}_1+\al \sqrt{\mu}\bar{H}_2,\sqrt{\mu}\bar{h})
+Q(\sqrt{\mu}\bar{h},\sqrt{\mu}\bar{H}_1+\al \sqrt{\mu}\bar{H}_2)\}\Big\}(t,V(t))\bigg|\notag\\
\leq& C\sup\limits_{0\leq t\leq T_0}\sum\limits_{\vth'\leq\vth}
\|w^{\ell_\infty}\pa_\xi^{\vth'}[h_1-\bar{h}_1,h_2-\bar{h}_2](t)\|_{L^\infty}.\notag
\end{align}
As a consequence, similarly for obtaining \eqref{CJ-55}, we get
\begin{align}
|\CJ_{10}-\bar{\CJ}_{10}|\leq C(\al+(1+M)^{-\ga}+\varsigma)\sum\limits_{\vth'\leq\vth}\|w^{\ell_\infty}\pa_\xi^{\vth'}[h_1-\bar{h}_1,h_2-\bar{h}_2](t)\|_{L^\infty}.\notag
\end{align}
Therefore \eqref{CQ-c} holds true.
Then, by taking $T_0$, $\al$ and $\varsigma$ suitably small and $M$ sufficiently large, one has
\begin{align}
\|\CN([h_1,h_2])-\CN([\bar{h}_1,\bar{h}_2])\|_{\FY_{\al,T_0}}\leq \frac{1}{2}\|[h_1,h_2]-[\bar{h}_1,\bar{h}_2]\|_{\FY_{\al,T_0}}.
\notag
\end{align}
Hence, $\CN$ is a contraction mapping on $\FY_{\al,T_0}$. So, there exists a unique $[h_1,h_2]\in\FY_{\al,T_0}$ such that
$$
[h_1,h_2]=\CN([h_1,h_2]).
$$
This completes the proof of Theorem \ref{loc.ex}.
\end{proof}

\section{Global existence and large time behavior}\label{sec.rm}

In this section, we study the global existence and large time behavior of the remainder $G_R$ determined by \eqref{GR-eq} and \eqref{GR-id} in order to complete the proof of  Theorem \ref{mth}.

 In fact, the global existence of \eqref{GR-eq} and \eqref{GR-id} follows from the local existence and {\it a priori} estimates as well as the continuum argument. Here, we only show the {\it a priori} estimates  \eqref{ap-es}, because the local existence has been established in Theorem \ref{loc.ex} in Section \ref{sec.le} and the nonnegativity can be justified in a similar way as that of \cite{DL-arma-2021}.
The approach used in the following is based on Caflisch's decomposition.

Now, our goal is to prove
\begin{align}\label{ap-es}
\sup\limits_{0\leq s\leq t}\sum\limits_{|\vth|\leq N}\beta^{2\ga}(s)\|w^{\ell_\infty}\pa_\xi^{\vth}G_{R,1}(s)\|_{L^\infty}
&+\sup\limits_{0\leq s\leq t}\sum\limits_{|\vth|\leq N}\beta^{2\ga}(s)\|w^{\ell_\infty}\pa_\xi^{\vth}G_{R,2}(s)\|_{L^\infty}\notag\\
\leq& C\sum\limits_{|\vth|\leq N}\|w^{\ell_\infty}\pa_\xi^{\vth}G_{R,0}\|_{L^\infty}+C\al^{3-m},
\end{align}
for any $t\geq0$ and for some constant $C>0$, under the {\it a priori} assumption that
\begin{align}\label{ap-as}
\sup\limits_{0\leq s\leq t}&\sum\limits_{|\vth|\leq N}\beta^{2\ga}(s)\|w^{\ell_\infty}\pa_\xi^{\vth}G_{R,1}(s)\|_{L^\infty}
+\sup\limits_{0\leq s\leq t}\sum\limits_{|\vth|\leq N}\beta^{2\ga}(s)\|w^{\ell_\infty}\pa_\xi^{\vth}G_{R,2}(s)\|_{L^\infty}\leq 2M_0.
\end{align}
The proof of \eqref{ap-es} is proceeded in the following three subsections.

\subsection{$L^\infty$ estimates}
In this subsection, we deduce the $L^\infty$ estimates on $G_{R,1}$ and $G_{R,2}$. For this, we have the following result.

\begin{lemma}\label{g12-lf-lem}
Under the conditions \eqref{ap-as}, it holds that
\begin{multline}\label{g1k-es2}
\sup\limits_{0\leq s\leq t}\sum\limits_{|\vth|\leq N}\|w^{\ell_\infty}\pa_\xi^{\vth}g_{1}(s)\|_{L^\infty}\\
\leq \sum\limits_{|\vth|\leq N}\|w^{\ell_\infty}\pa_\xi^{\vth}G_{R,0}\|_{L^\infty}+
C\al\sup\limits_{0\leq s\leq t}\sum\limits_{|\vth|\leq N} \|w^{\ell_\infty}\pa_\xi^{\vth}g_{2}(s)\|_{L^\infty}+C\al^{3-m},
\end{multline}
and
\begin{align}\label{g2k-es2}
\sup\limits_{0\leq s\leq t}\sum\limits_{|\vth|\leq N}\|w^{\ell_\infty}\pa_\xi^{\vth}g_{2}(s)\|_{L^\infty}\leq&
C\sup\limits_{0\leq s\leq t}\sum\limits_{|\vth|\leq N}\|w^{\ell_\infty}\pa_\xi^{\vth}g_{1}(s)\|_{L^\infty}
+C\sup\limits_{0\leq s\leq t}\sum\limits_{|\vth|\leq N}\|\pa_\xi^{\vth}g_2(s)\|.
\end{align}

\end{lemma}
\begin{proof}
For brevity, we denote
$$
[g_{1},g_{2}](s,\xi)=\beta^{2\ga}(s)[G_{R,1}(s,\xi),G_{R,2}(s,\xi)].
$$
In view of \eqref{GR1-eq} and \eqref{GR2-eq}, one has
\begin{align}\label{g1-eq}
\pa_t &(w^{\ell_\infty}\pa_\xi^\vth g_{1})
-\frac{\beta'}{\beta}\xi \cdot\na_\xi(w^{\ell_\infty} \pa_\xi^\vth  g_{1})
-\frac{(2\ga+3)\beta'}{\beta}w^{\ell_\infty} \pa_\xi^\vth  g_{1}
+\frac{2{\ell}\beta'}{\beta}\frac{|\xi|^2}{1+|\xi|^2}w^{\ell_\infty}\pa_\xi^{\vth}g_{1}
\notag\\&-\al A\xi\cdot\na_\xi( w^{\ell_\infty} \pa_\xi^\vth g_{1})
+2{\ell}\al\frac{\xi\cdot A\xi}{1+|\xi|^2}w^{\ell_\infty}\pa_\xi^{\vth}g_{1}+\nu\beta^{\ga} w^{\ell_\infty} \pa_\xi^\vth g_{1}\notag\\
=&\beta^{\ga}\chi_M w^{\ell_\infty}\CK \pa_\xi^\vth g_{1}
+{\bf 1}_{\vth>0}C_\vth^{\vth'} w^{\ell_\infty}\pa_\xi^{\vth'}(\chi_M \CK)\pa_\xi^{\vth-\vth'}g_1
+\frac{\beta'}{\beta} {\bf 1}_{|\vth|>0}\sum\limits_{|\vth'|=1}
C_\vth^{\vth'}w^\ell\pa_\xi^{\vth'}\xi\cdot\na_\xi\pa_\xi^{\vth-\vth'}g_{1}
\notag\\&+\al{\bf 1}_{|\vth|>0}\sum\limits_{|\vth'|=1}C_\vth^{\vth'}
w^{\ell_\infty}\pa_\xi^{\vth'}(A\xi)\cdot\na_\xi\pa_\xi^{\vth-\vth'}g_{1}
-{\bf 1}_{|\vth|>0}C_\vth^{\vth'}\beta^{\ga} \pa_\xi^{\vth'}\nu w^{\ell_\infty} \pa_\xi^{\vth-\vth'} g_{1}
-\frac{\beta'}{2\beta} w^{\ell_\infty}\pa_\xi^{\vth}\left\{|\xi|^2\sqrt{\mu}g_2\right\}
\notag\\&-\frac{\al}{2} w^{\ell_\infty}\pa_\xi^{\vth}\left\{\xi\cdot(A\xi)\sqrt{\mu}g_{2}\right\}
-\al^{1-m}\beta^{2\ga}w^{\ell_\infty}\pa_\xi^{\vth}\left\{\sqrt{\mu}\pa_t G_1\right\}
-\al^{2-m}\beta^{2\ga}w^{\ell_\infty}\pa_\xi^{\vth}\left\{\sqrt{\mu}\pa_t G_2\right\}
\notag\\&+\al^{3-m}\beta^{2\ga}\beta_1w^{\ell_\infty}\na_\xi\cdot(\xi\mu)
+\al^{1-m}\beta'\beta^{2\ga-1}w^{\ell_\infty}\pa_\xi^{\vth}\left\{\na_\xi\cdot(\xi\sqrt{\mu}G_1)\right\}
\notag\\&+\al^{2-m}\beta'\beta^{2\ga-1}w^{\ell_\infty}\pa_\xi^{\vth}\left\{\na_\xi\cdot(\xi\sqrt{\mu}G_2)\right\}
+\al^{3-m}\beta^{2\ga}w^{\ell_\infty}\na_\xi\cdot(A\xi\sqrt{\mu}G_2)
\notag\\&+\al^{3-m}\beta^{3\ga}w^{\ell_\infty}\pa_\xi^{\vth}\{Q(\sqrt{\mu}G_1,\sqrt{\mu}G_2)+Q(\sqrt{\mu}G_2,\sqrt{\mu}G_1)\}
+\al^{4-m}\beta^{3\ga}w^{\ell_\infty} \pa_\xi^{\vth}Q(\sqrt{\mu}G_2,\sqrt{\mu}G_2)
\notag\\&+\al^m w^{\ell_\infty} \beta^{-\ga}\pa_\xi^{\vth}Q(\sqrt{\mu}g,\sqrt{\mu}g)
+\al\beta^{\ga}w^{\ell_\infty}\pa_\xi^{\vth}\{Q(\sqrt{\mu}G_1+\al \sqrt{\mu}G_2,\sqrt{\mu}g)
+Q(\sqrt{\mu}g,\sqrt{\mu}G_1+\al \sqrt{\mu}G_2)\},
\end{align}
\begin{align}
g_{1}(0,\xi)=G_{R,0},\notag
\end{align}
\begin{align}\label{g2-eq}
\pa_t (w^{\ell_\infty}\pa_\xi^\vth g_{2})&
-\frac{\beta'}{\beta}\xi \cdot\na_\xi(w^{\ell_\infty} \pa_\xi^\vth  g_{2})-\frac{(2\ga+3)\beta'}{\beta}w^{\ell_\infty} \pa_\xi^\vth  g_{2}
+\frac{2{\ell}\beta'}{\beta}\frac{|\xi|^2}{1+|\xi|^2}w^{\ell_\infty}\pa_\xi^{\vth}g_{2}
\notag\\&-\al A\xi\cdot\na_\xi( w^{\ell_\infty}\pa_\xi^\vth g_{2})+2{\ell}\al\frac{\xi\cdot A\xi}{1+|\xi|^2}w^{\ell_\infty}\pa_\xi^{\vth}g_{2}
+\beta^\ga \nu w^{\ell_\infty} \pa_\xi^\vth g_{2}\notag\\
=&\beta^\ga\chi_M w^{\ell_\infty} K \pa_\xi^\vth g_{2}
+{\bf 1}_{\vth>0}C_\vth^{\vth'}\beta^\ga w^{\ell_\infty}\pa_\xi^{\vth'}(\chi_M K)\pa_\xi^{\vth-\vth'}g_2
\notag\\&+\frac{\beta'}{\beta} {\bf 1}_{|\vth|>0}\sum\limits_{|\vth'|=1}
C_\vth^{\vth'}w^{\ell_\infty}\pa_\xi^{\vth'}\xi\cdot\na_\xi\pa_\xi^{\vth-\vth'}g_{2}
+\al{\bf 1}_{|\vth|>0}\sum\limits_{|\vth'|=1}C_\vth^{\vth'}w^{\ell_\infty}\pa_\xi^{\vth'}(A\xi)\cdot\na_\xi\pa_\xi^{\vth-\vth'}g_{1}
\notag\\&-{\bf 1}_{|\vth|>0}C_\vth^{\vth'}\beta^\ga w^{\ell_\infty}\pa_\xi^{\vth'}\nu  \pa_\xi^{\vth-\vth'} g_{2}
+\beta^\ga w^{\ell_\infty}\pa_\xi^\vth\left\{(1-\chi_M)\mu^{-\frac{1}{2}} \CK g_{2}\right\},
\end{align}
and
\begin{align}
g_{2}(0,\xi)=0,\notag
\end{align}
respectively, where $\sqrt{\mu}g=g_1+\sqrt{\mu}g_2.$

Recalling Lemma \ref{H12-pro}, since $[g_1,g_2]\in\FY_{\al,t}$ for any $t\geq0$, one can gets for any $\ell\geq0$
\begin{align}
\sum\limits_{|\vth|\leq N+1}\|w^{\ell}\pa_t^k\pa_\xi^\vartheta G_1(t,\xi)\|_{L^\infty}\leq C\al^{2k}\beta^{-\ga-k\ga},\label{G1-es}
\end{align}
and
\begin{align}
\sum\limits_{|\vth|\leq N}\|w^{\ell}\pa_t^k\pa_\xi^\vartheta G_2(t,\xi)\|_{L^\infty}\leq C\al^{2k}\beta^{-2\ga-k\ga},\label{G2-es}
\end{align}
where $C>0$ depends on $\ell, k$ and $N.$ Moreover, it holds that
\begin{align}\label{bet-eq}
\beta^\ga(t)\sim 1+\ga\vho_0\al^2t,
\end{align}
for any $t\geq0.$

We now turn to estimate $g_1$ and $g_2$.
Integrating along the backward trajectory \eqref{CL} with respect to $s\in[0,t]$, one can write the solutions of \eqref{g1-eq} and \eqref{g2-eq} as the following mild form
\begin{align}
w^{\ell_\infty}\pa_\xi^{\vth}g_{1}(t,x,v)=\sum\limits_{i=1}^{11}\CH_{i},\label{g1k}
\end{align}
with
\begin{align}
\CH_1=e^{-\int_0^t\CA(s)ds}w^{\ell_\infty}\pa_\xi^\vth G_{R,0}(V(0)),\notag
\end{align}
\begin{align}
\CH_2= {\bf 1}_{|\vth|>0}\sum\limits_{|\vth'|=1}C_\vth^{\vth'}\int_0^te^{-\int_s^t\CA(\tau)d\tau}
\left\{\frac{\beta'}{\beta}w^{\ell_\infty}\pa^{\vth'}_\xi \xi\cdot\na_\xi\pa_\xi^{\vth-\vth'}g_{1}\right\}(s,V(s))ds,\notag
\end{align}
\begin{align}
\CH_3=\al{\bf 1}_{|\vth|>0}\sum\limits_{|\vth'|=1}C_\vth^{\vth'}\int_0^te^{-\int_s^t\CA(\tau)d\tau}
\{w^{\ell_\infty}\pa_\xi^{\vth'}(A\xi)\cdot\na_\xi\pa^{\vth-\vth'}_\xi g_{1}\}(s,V(s))ds,\notag
\end{align}
\begin{align}
\CH_4=-{\bf 1}_{|\vth|>0}\sum\limits_{0<\vth'\leq \vth}C_\vth^{\vth'}
\int_0^te^{-\int_s^t\CA(\tau)d\tau}\{\beta^{\ga} w^{\ell_\infty}\pa^{\vth'}_\xi\nu\pa^{\vth-\vth'}_{\xi}g_{1}\}(s,V(s))ds
,\notag
\end{align}
\begin{align}
\CH_5=\int_0^te^{-\int_s^t\CA(\tau)d\tau}\{\beta^{\ga} w^{\ell_\infty}\chi_M\CK \pa_{\xi}^\vth g_{1}\}(s,V(s))ds
,\notag
\end{align}
\begin{align}
\CH_6={\bf 1}_{\vth\geq\vth'>0}C_\vth^{\vth'}\int_0^te^{-\int_s^t\CA(\tau)d\tau}\{\beta^{\ga} w^{\ell_\infty}\pa_\xi^{\vth'}(\chi_M \CK)\pa_\xi^{\vth-\vth'}g_1\}(s,V(s))ds
,\notag
\end{align}
\begin{align}
\CH_7=-\int_0^t&e^{\int_s^t\CA(\tau)d\tau}\left\{\frac{\beta'}{2\beta}w^{\ell_\infty}\pa_{\xi}^\vth(|\xi|^2\sqrt{\mu}g_{2})
+\frac{\al}{2} w^{\ell_\infty}\pa_{\xi}^\vth(\xi\cdot(A\xi)\sqrt{\mu}g_{2})\right\}(s,V(s))ds,\notag
\end{align}
\begin{align}
\CH_8=-\int_0^te^{\int_s^t\CA(\tau)d\tau}&\Bigg\{\al^{1-m}\beta^{2\ga} w^{\ell_\infty}\pa_\xi^{\vth}\left\{\sqrt{\mu}\pa_t G_1\right\}
+\al^{2-m}\beta^{2\ga} w^{\ell_\infty}\pa_\xi^{\vth}\left\{\sqrt{\mu}\pa_t G_2\right\}
\notag\\&-\al^{1-m}\beta'\beta^{2\ga-1} w^{\ell_\infty}\pa_\xi^{\vth}\left\{\na_\xi\cdot(\xi\sqrt{\mu}G_1)\right\}\bigg\}(s,V(s))ds,\notag
\end{align}
\begin{align}
\CH_9=\int_0^te^{\int_s^t\CA(\tau)d\tau}&\bigg\{
\al^{3-m}\beta^{2\ga}\beta_1w^{\ell_\infty}\na_\xi\cdot(\xi\mu)+\al^{2-m}\beta'\beta^{2\ga-1} w^{\ell_\infty}\pa_\xi^{\vth}
\left\{\na_\xi\cdot(\xi\sqrt{\mu}G_2)\right\}
\notag\\&+\al^{3-m}\beta^{2\ga} w^{\ell_\infty}\pa_\xi^{\vth}\left\{\na_\xi\cdot(A\xi\sqrt{\mu}G_2)\right\}\bigg\}(s,V(s))ds,\notag
\end{align}
\begin{align}
\CH_{10}=\int_0^te^{-\int_s^t\CA(\tau)d\tau}w^{\ell_\infty} \Big\{\al^{3-m}&\beta^{3\ga} \pa_\xi^{\vth}\{Q(\sqrt{\mu}G_1,\sqrt{\mu}G_2)+Q(\sqrt{\mu}G_2,\sqrt{\mu}G_1)\}
\notag\\&+\al^{4-m}\beta^{3\ga} \pa_\xi^{\vth}Q(\sqrt{\mu}G_2,\sqrt{\mu}G_2)\Big\}(s,V(s))ds,\notag
\end{align}
\begin{align}
\CH_{11}=\int_0^t&e^{-\int_s^t\CA(\tau)d\tau}\Big\{\al^m\beta^{-\ga} w^{\ell_\infty} \pa_\xi^{\vth}\{Q(\sqrt{\mu}g,\sqrt{\mu}g)\}\notag\\
&+\al w^{\ell_\infty}\beta^{\ga}\pa_\xi^{\vth}\{Q(\sqrt{\mu}G_1+\al \sqrt{\mu}G_2,\sqrt{\mu}g)
+Q(\sqrt{\mu}g,\sqrt{\mu}G_1+\al \sqrt{\mu}G_2)\}\Big\}(s,V(s))ds,\notag
\end{align}
and
\begin{align}
w^{\ell_\infty}\pa_\ze^\vth g_{2}(t,x,v)
=\sum\limits_{i=12}^{17}\CH_{i},\label{g2k}
\end{align}
with
\begin{align}
\CH_{12}= {\bf 1}_{|\vth|>0}\sum\limits_{|\vth'|=1}C_\vth^{\vth'}\int_0^te^{-\int_s^t\CA(\tau)d\tau}
\left\{\frac{\beta'}{\beta}w^{\ell_\infty}\pa^{\vth'}_\xi \xi\cdot\na_\xi\pa_\xi^{\vth-\vth'}g_{2}\right\}(s,V(s))ds,\notag
\end{align}
\begin{align}
\CH_{13}=\al{\bf 1}_{|\vth|>0}\sum\limits_{|\vth'|=1}C_\vth^{\vth'}\int_0^te^{-\int_s^t\CA(\tau)d\tau}
\{w^{\ell_\infty}\pa_\xi^{\vth'}(A\xi)\cdot\na_\xi\pa^{\vth-\vth'}_\xi g_{2}\}(s,V(s))ds,\notag
\end{align}
\begin{align}
\CH_{14}=-{\bf 1}_{|\vth|>0}\sum\limits_{0<\vth'\leq \vth}C_\vth^{\vth'}
\int_0^te^{-\int_s^t\CA(\tau)d\tau}\{\beta^\ga w^{\ell_\infty}\pa^{\vth'}_\xi\nu\pa^{\vth-\vth'}_{\xi}g_{2}\}(s,V(s))ds
,\notag
\end{align}
\begin{align}
\CH_{15}=\int_0^te^{-\int_s^t\CA(\tau)d\tau}\{\beta^\ga w^{\ell_\infty}\chi_M K \pa_{\xi}^\vth g_{2}\}(s,V(s))ds
,\notag
\end{align}
\begin{align}
\CH_{16}={\bf 1}_{\vth\geq\vth'>0}C_\vth^{\vth'}\int_0^te^{-\int_s^t\CA(\tau)d\tau}
\{\beta^\ga w^{\ell_\infty}\pa_\xi^{\vth'}(\chi_M K)\pa_\xi^{\vth-\vth'}g_2\}(s,V(s))ds,\notag
\end{align}
and
\begin{align}
\CH_{17}=\int_0^te^{-\int_s^t\CA(\tau)d\tau}\left\{\beta^\ga w^{\ell_\infty}\pa_\xi^\vth\left\{(1-\chi_M)\mu^{-\frac{1}{2}} \CK g_{2}\right\}\right\}(s,V(s))ds.\notag
\end{align}
Here, we have denoted
\begin{align}
\CA(\tau,V(\tau))=&\beta^{\ga}\nu(V(\tau))-\frac{(3+2\ga)\beta'}{\beta}+2\ell_{\infty} \frac{\beta'}{\beta} \frac{|V(\tau)|^2}{1+|V(\tau)|^2} +2{\ell_\infty} \al \frac{V(\tau)\cdot (AV(\tau))}{{1+|V(\tau)|^2}}.\notag
\end{align}
One sees that, as long as ${\ell_\infty}\al>0$ and $\al$ are suitably small,
$$\CA(\tau,V(\tau))\geq \frac{1}{2}\beta^{\ga}\nu(V(\tau))>\tilde{C}_0\beta^{\ga},$$
for some $\tilde{C}_0>0$. Moreover, it holds that
\begin{align}
\int_0^te^{-\int_s^t\CA(\tau)d\tau}\beta^{\ga}(s)\nu(V(s))ds<\infty.\label{Ala-ubd}
\end{align}
We now turn to estimate $\CH_i$ $(1\leq i\leq 17)$ separately.
We start with the nonlocal terms $\CH_5$, $\CH_6$, $\CH_{10}$, $\CH_{11}$, $\CH_{15}$, $\CH_{16}$ and $\CH_{16}$.

For $\CH_5$, applying \eqref{Ala-ubd} and using \eqref{CK2} in Lemma \ref{g-ck-lem}, we obtain
\begin{align}
|\CH_5|\leq& C\int_0^te^{-\int_s^t\frac{\beta^{\ga}\nu(V(\tau))}{2}d\tau}\beta^{\ga}\nu(V(s))ds
\sup\limits_{0\leq s\leq t} \left\|\left\{\nu^{-1}w^{\ell_\infty}(\chi_M\CK \pa_\xi^\vth g_{1})\right\}(s,V(s))\right\|_{L^\infty}
\notag\\
 \leq& C\left((1+M)^{-\ga}+\varsigma\right)\sup\limits_{0\leq s\leq t}\|w^{\ell_\infty}\pa_{\xi}^{\vth}g_{1}(s)\|_{L^\infty}.\notag
\end{align}
Recalling \eqref{sp.cL}, one gets from Lemma \ref{op.es.lem} that
\begin{align}
|\CH_6|\leq&{\bf 1}_{|\vth|>0} C\int_0^te^{-\int_s^t\frac{\beta^{\ga}\nu(V(\tau))}{2}d\tau}\beta^{\ga}\nu(V(s))ds\sum\limits_{\vth'<\vth}
\sup\limits_{0\leq s\leq t} \left\|w^{\ell_\infty}(\pa_\xi^{\vth'} g_{1})(s,V(s))\right\|_{L^\infty}
\notag\\
 \leq&C{\bf 1}_{|\vth|>0}\sum\limits_{\vth'<\vth}
\sup\limits_{0\leq s\leq t}\left\|w^{\ell_\infty}\pa_\xi^{\vth'} g_{1}(s)\right\|_{L^\infty}.\notag
\end{align}
For $\CH_{10}$, \eqref{G1-es}, \eqref{G2-es} and Lemma \ref{op.es.lem} give
\begin{align}
|\CH_{10}|\leq&C\al^{3-m}\int_0^te^{-\int_s^t\frac{\beta^{\ga}\nu(V(\tau))}{2}d\tau}\beta^{\ga}\nu(V(s))ds
\notag\\&\qquad\qquad\times\sum\limits_{\vth'+\vth''\leq\vth}\sup\limits_{0\leq s\leq t} \beta^{2\ga}(s)
\left\|w^{\ell_\infty}\pa_\xi^{\vth'} G_{1}(s)\right\|_{L^\infty}\left\|w^{\ell_\infty}\pa_\xi^{\vth''} G_{2}(s)\right\|_{L^\infty}
 \leq C\al^{3-m}.\notag
\end{align}
Likewise, for $\CH_{11}$, applying \eqref{G1-es}, \eqref{G2-es} and Lemma \ref{op.es.lem} as well as the {\it a priori} assumption \eqref{ap-as},
one has
\begin{align}
|\CH_{11}|\leq&C\al^m\int_0^te^{-\int_s^t\frac{\beta^{\ga}\nu(V(\tau))}{2}d\tau}\beta^{\ga}\nu(V(s))ds\sum\limits_{\vth'+\vth''\leq\vth}
\sup\limits_{0\leq s\leq t}
\left\|w^{\ell_\infty}\pa_\xi^{\vth'} [g_{1},g_2](s)\right\|_{L^\infty}\left\|w^{\ell_\infty}\pa_\xi^{\vth''} [g_{1},g_2](s)\right\|_{L^\infty}
\notag\\
&+C\al\int_0^te^{-\int_s^t\frac{\beta^{\ga}\nu(V(\tau))}{2}d\tau}\beta^{\ga}\nu(V(s))ds\sum\limits_{\vth'+\vth''\leq\vth}
\sup\limits_{0\leq s\leq t} \left\|w^{\ell_\infty}\pa_\xi^{\vth'} [G_{1},G_2](s)\right\|_{L^\infty}\left\|w^\ell\pa_\xi^{\vth''} [g_{1},g_2](s)\right\|_{L^\infty}\notag\\
 \leq&C\al\sum\limits_{\vth'\leq\vth}
\sup\limits_{0\leq s\leq t} \left\|w^{\ell_\infty}\pa_\xi^{\vth'} [g_{1},g_2](s)\right\|_{L^\infty}.\notag
\end{align}
For the delicate term $\CH_{15}$, we first rewrite
\begin{align}
\CH_{15}=\int_0^te^{-\int_s^t\CA(\tau)d\tau}\beta^\ga(s)\int_{\R^3}{\bf k}_w(V(s),\xi_\ast)(w^{\ell_\infty}\pa_\xi^{\vth}g_{2})(s,\xi_\ast)d\xi_\ast ds.\notag
\end{align}
As in the proof of Lemma \ref{H12-pro}, the computation for $\CH_{15}$ is then divided in the following three cases.

\medskip
\noindent\underline{{\it Case 1. $|V(s)|>M$.}} In this case, we get from Lemma \ref{Kop} that
\begin{align}
|\CH_{15}|\leq\frac{C}{M}\sup\limits_{0\leq s\leq t}\|w^{\ell_\infty}\pa_\xi^{\vth}g_{2}(s)\|_{L^\infty}.\notag
\end{align}
\underline{{\it Case 2. $|V(s)|\leq M$ and $|\xi_\ast|>2M$.}} In this situation, one has $|V(s)-\xi_\ast|>M$, thus it follows
\begin{equation*}
\mathbf{k}_w(V,\xi_\ast)
\leq Ce^{-\frac{\vps M^2}{8}}\mathbf{k}_w(V,\xi_\ast)e^{\frac{\vps |V-\xi_\ast|^2}{8}},
\end{equation*}
this together with Lemma \ref{Kop} leads to
\begin{align}
|\CH_{15}|\leq Ce^{-\frac{\vps M^2}{8}}\sup\limits_{0\leq s\leq t}\|w^{\ell_\infty}\pa_\xi^{\vth}g_{2}(s)\|_{L^\infty}.\notag
\end{align}

\noindent\underline{{\it Case 3. $|V(s)|\leq M$ and $|\xi_\ast|\leq2M$.}} At this stage, recalling ${\bf k}_{w,p}(V,\xi_\ast)$ defined by \eqref{km}, we write
\begin{align}
\CH_{15}=\int_0^te^{-\int_s^t\CA(\tau)d\tau}\int_{\R^3}\beta^\ga(s)
[{\bf k}_w-{\bf k}_{w,p}+{\bf k}_{w,p}](V(s),\xi_\ast)(w^{\ell_\infty}\pa_\xi^{\vth}g_{2})(s,\xi_\ast)d\xi_\ast ds,\notag
\end{align}
which further gives the bound
\begin{align*}
|\CH_{15}|&\leq C\int_{|\xi_\ast|\leq 2M}\mathbf{k}_{w,p}(\xi,\xi_\ast)|\pa_\xi^\vartheta g_2|d\xi_\ast
+\frac{1}{M}\|w^{\ell_\infty}\pa_\xi^\vartheta g_2\|_{L^\infty}\\
&\leq C_{p,M}\|\pa_\xi^\vartheta g_2\|+\frac{C}{M}\|w^{\ell_\infty}\pa_\xi^\vartheta g_2\|_{L^\infty}.
\end{align*}
To summarize, we arrive at
\begin{align*}
|\CH_{15}|
&\leq C\|\pa_\xi^\vartheta g_2\|+C\left\{\frac{1}{M}+e^{-\frac{\vps M^2}{8}}\right\}\|w^{\ell_\infty}\pa_\xi^\vartheta g_2\|_{L^\infty}.
\end{align*}
For $\CH_{16}$, from Lemma \ref{Ga}, it follows
\begin{align}
|\CH_{16}|\leq&C{\bf 1}_{\vth>0}\sum\limits_{\vth'<\vth}
\sup\limits_{0\leq s\leq t} \left\|w^{\ell_\infty}\pa_\xi^{\vth'} g_2(s)\right\|_{L^\infty}.\notag
\end{align}
For $\CH_{17}$, one has by Lemma \ref{op.es.lem}
\begin{align}
|\CH_{17}|\leq&C\sum\limits_{\vth'\leq\vth}
\sup\limits_{0\leq s\leq t} \left\|w^{\ell_\infty}\pa_\xi^{\vth'} g_1(s)\right\|_{L^\infty}.\notag
\end{align}
The remaining terms in \eqref{g1k} and \eqref{g2k} will be computed as follows
\begin{align}
|\CH_1|\leq\|w^{\ell_\infty}\pa_\xi^\vth G_{R,0}\|_{L^\infty}.\notag
\end{align}
\begin{align}
|\CH_3|\leq {\bf 1}_{\vth>0}C\al\sup\limits_{0\leq s\leq t}\sum\limits_{|\vth'|=|\vth|}
\|w^{\ell_\infty}\pa_{\xi}^{\vth'}g_{1}(s)\|_{L^\infty},\ \ |\CH_{13}|\leq {\bf 1}_{\vth>0}C\al\sup\limits_{0\leq s\leq t}\sum\limits_{|\vth'|=|\vth|}
\|w^{\ell_\infty}\pa_{\xi}^{\vth'}g_{2}(s)\|_{L^\infty},\notag
\end{align}
\begin{align}
|\CH_4|\leq {\bf 1}_{\vth>0}C\sup\limits_{0\leq s\leq t}\sum\limits_{\vth'<\vth}
\|w^{\ell_\infty}\pa_{\xi}^{\vth'}g_{1}(s)\|_{L^\infty},\ \ |\CH_{14}|\leq {\bf 1}_{\vth>0}C\sup\limits_{0\leq s\leq t}\sum\limits_{\vth'<\vth}
\|w^{\ell_\infty}\pa_{\xi}^{\vth'}g_{2}(s)\|_{L^\infty}.\notag
\end{align}
Since
\begin{align}
\frac{\beta'}{\beta}\sim \vho_0\beta^{-\ga}\al^2,\label{bbp}
\end{align}
according to \eqref{bet-eq}, one has
\begin{align}
|\CH_2|\leq {\bf 1}_{\vth>0}C\al^2\sup\limits_{0\leq s\leq t}
\|w^{\ell_\infty}\pa_{\xi}^{\vth}g_{1}(s)\|_{L^\infty},\ \ |\CH_{12}|\leq {\bf 1}_{\vth>0}C\al^2\sup\limits_{0\leq s\leq t}
\|w^{\ell_\infty}\pa_{\xi}^{\vth}g_{2}(s)\|_{L^\infty},\notag
\end{align}
and
\begin{align}
|\CH_{7}|\leq C\al\sup\limits_{0\leq s\leq t}\sum\limits_{\vth'\leq\vth}
\|w^{\ell_\infty}\pa_{\xi}^{\vth'}g_{2}(s)\|_{L^\infty}.\notag
\end{align}
Finally, using \eqref{G1-es}, \eqref{G2-es} and \eqref{bbp}, we obtain
\begin{align}
|\CH_{8}|,\ |\CH_{9}|\leq C\al^{3-m}.\notag
\end{align}
As a consequence, by plugging all the above estimates for $\CH_i$ $(1\leq i\leq 18)$ into \eqref{g1k} and \eqref{g2k}, respectively,
one gets
\begin{align}\label{g1k-es1}
|w^{\ell_\infty}\pa_\xi^{\vth}g_{1}(t,\xi)|\leq&\|w^{\ell_\infty}\pa_{\xi}^{\vth}G_{R,0}\|_{L^\infty}
+{\bf 1}_{\vth>0}C\sup\limits_{0\leq s\leq t}\sum\limits_{\vth'<\vth} \|w^{\ell_\infty}\pa_{\xi}^{\vth'}g_{1}(s)\|_{L^\infty}
\notag\\
&+C\left(\al+(1+M)^{-\ga}+\varsigma\right)
\sup\limits_{0\leq s\leq t} \|w^{\ell_\infty}\pa^{\vth}g_{1}(s)\|_{L^\infty}
\notag\\&+C\al
\sup\limits_{0\leq s\leq t}\sum\limits_{\vth'\leq\vth} \|w^{\ell_\infty}\pa_{\xi}^{\vth'}g_2(s)\|_{L^\infty}+C\al^{3-m},
\end{align}
and
\begin{align}\label{g2k-es1}
|w^{\ell_\infty}\pa_\xi^{\vth}g_{2}(t,\xi)|\leq&
C\sup\limits_{0\leq s\leq t}\sum\limits_{\vth'\leq \vth}\|w^{\ell_\infty}\pa_{\xi}^{\vth'}g_{1}(s)\|_{L^\infty}
+{\bf 1}_{\ze>0}C\sup\limits_{0\leq s\leq t}\sum\limits_{\vth'<\vth} \|w^{\ell_\infty}\pa_{\xi}^{\vth'}g_{2}(s)\|_{L^\infty}\notag\\
&+C\left(\al+e^{-\frac{\vps M^2}{8}}+\frac{1}{M}\right)\sup\limits_{0\leq s\leq t} \|w^{\ell_\infty}\pa_{\xi}^{\vth}g_{2}(s)\|_{L^\infty}
+\sup\limits_{0\leq s\leq t}\|\pa_\xi^\vartheta g_2(s)\|.
\end{align}
Furthermore, taking a linear combination of \eqref{g1k-es1} and \eqref{g2k-es1} with $|\vth|=0,1,\cdots, N$ and adjusting constants, we see that both \eqref{g1k-es2} and \eqref{g2k-es2} are true. This completes the proof of Lemma \ref{g12-lf-lem}.
\end{proof}


\subsection{$L^2$ estimates} In this subsection, we deduce the $L^2$ estimates on $G_{R,1}$ and $G_{R,2}$. For this, we have the following result.

\begin{lemma}\label{g12-l2-lem}
Under the conditions \eqref{ap-as}, it holds that
\begin{align}
\sum\limits_{|\vth|\leq N}&\|\pa_\xi^{\vth}g_2(t)\|^2+\sum\limits_{\vth \leq N}\|w^{\ell_2}\pa_\xi^{\vth}g_{1}(t)\|^2
\leq C\sum\limits_{\vth \leq N}\|w^{\ell_2}\pa_\xi^{\vth}g_{1}(0,\xi)\|^2+C\eta\|w^{\ell_\infty}\pa_\xi^{\vth}g_{1}\|_{L^\infty}^2+
C\al^{6-2m},\label{l2-final}
\end{align}
for any $t\geq0.$
\end{lemma}

\begin{proof}
The proof is divided in the following four steps.

\medskip
\noindent\underline{{\it Step 1. The estimates for $\FP_0g_2$.}}
In this step, we consider the basic $L^2$ estimate for $\FP_0g_2$, in which the conservation laws plays an important role.
Recall $[g_1,g_2](t,\xi)=\beta^{2\ga}(t)[G_{R,1},G_{R,2}](t,\xi)$ and $\sqrt{\mu}g=g_1+\sqrt{\mu}g_2$. It is straightforward to see that
$g$ satisfies
\begin{align}
\pa_t g&-\frac{\beta'}{\beta}\na_\xi\cdot(\xi\sqrt{\mu}g)\mu^{-\frac{1}{2}}
-\frac{2\ga\beta'}{\beta}g
-\al\na_\xi\cdot(A\xi\sqrt{\mu}g)\mu^{-\frac{1}{2}}+\beta^\ga L g\notag\\
=&-\al^{1-m} \beta^{2\ga}\pa_t G_1-\al^{2-m} \beta^{2\ga}\pa_t G_2+\al^{3-m}\beta^{2\ga}\beta_1w^{\ell_\infty}\na_\xi\cdot(\xi\mu)+ \al^{1-m}\beta'\beta^{2\ga-1}\na_\xi\cdot(\xi\sqrt{\mu}G_1)\mu^{-\frac{1}{2}}
\notag\\&+\al^{2-m}\beta'\beta^{2\ga-1}\na_\xi\cdot(\xi\sqrt{\mu}G_2)\mu^{-\frac{1}{2}}
+\al^{3-m}\beta^{2\ga}\na_\xi\cdot(A\xi\sqrt{\mu}G_2)\mu^{-\frac{1}{2}}+\al^{3-m}\beta^{3\ga} \{\Ga(G_1,G_2)+\Ga(G_2,G_1)\}\notag\\&+\al^{4-m}\beta^{3\ga}\Ga(G_2,G_2)
+\al\beta^{\ga}\{\Ga(G_1+\al G_2,g)+\Ga(g,G_1+\al G_2)\}+\al^{m}\beta^{-\ga}\Ga(g,g),\notag
\end{align}
with
\begin{align}
\sqrt{\mu}g(0,\xi)=G_{R,0}(\xi).\notag
\end{align}
Next let us define
\begin{align}
\FP_0 g&=\{a(t)+\Fb(t)\cdot v+c(t)(|\xi|^2-3)\}\sqrt{\mu},\notag \\
\bar{\FP}_0g_1&=\{a_1(t)+\Fb_1(t)\cdot \xi+c_1(t)(|\xi|^2-3)\}\mu,\notag
\end{align}
and
$$
\FP_0g_2=\{a_2(t)+\Fb_2(t)\cdot v+c_2(t)(|v|^2-3)\}\sqrt{\mu}.\notag
$$
Then it follows
\begin{align}
a(t)=a_1(t)+a_2(t),\ \Fb(t)=\Fb_1(t)+\Fb_2(t),\ c(t)=c_1(t)+c_2(t),\notag
\end{align}
for any $t\geq0$.

From \eqref{GR-con} and \eqref{GR-con-id}, it is straightforward to check
$$
\FP_0 g=0,
$$
hence
\begin{align}
a_1=-a_2,\ \Fb_1=-\Fb_2,\ c_1=-c_2.\notag
\end{align}
Thus, for $\geq\ell_2\geq2$, it follows
\begin{align}
|[a_2,\Fb_2,c_2](t)|\leq C\|w^{\ell_2} g_1(t)\|.\label{abc-es}
\end{align}

\noindent\underline{{\it Step 2. $L^2$ estimates for $\FP_1g_2$.}}
We now derive the $L^2$ estimate on $\FP_1g_2$.
Recall that $g_2$ satisfies
\begin{align}\label{f2-eq-2}
\pa_t g_2&-\frac{\beta'}{\beta}\na_\xi\cdot(\xi g_2)-\frac{2\ga\beta'}{\beta} g_2-\al\na_\xi\cdot(A\xi g_2)
+\beta^\ga Lg_2=\beta^\ga(1-\chi_M)\mu^{-\frac{1}{2}}\CK g_1.
\end{align}
Taking the inner product of \eqref{f2-eq-2} and $\FP_1g_2$ over $\R^3$, applying Cauchy-Schwarz's inequality and Lemmas \ref{Ga} and \ref{es-tri} as well as \eqref{abc-es},  we obtain
\begin{align}\label{g2-l2}
\frac{1}{2}\frac{d}{dt}\|\FP_1g_2\|^2
+\la\beta^\ga\|\FP_1g_2\|^2_\nu\leq& C\al^2|[a_2,\Fb_2,c_2]|^2
+C\beta^\ga\|w^{\ell_2}g_1\|^2\leq C\beta^\ga\|w^{\ell_2}g_1\|^2.
\end{align}

\noindent\underline{{\it Step 3. Higher order estimates for $\FP_1g_2$.}}
Since $\|\pa_\xi^\vth g_2\|\leq C|[a_2,\Fb_2,c_2]|+C\|\pa_\xi^\vth \FP_1g_2\|$, to obtain the higher order $L^2$ estimates on $g_2$,
it suffices to deduce the corresponding estimates on $\FP_1g_2$. For this, we first take $\FP_1$ projection of \eqref{f2-eq-2} to obtain
\begin{multline}\label{f2-eq-p1}
\pa_t\FP_1 g_2-\frac{(3+2\ga)\beta'}{\beta}\FP_1g_2
-\frac{\beta'}{\beta}[\xi\cdot\na_\xi(\FP_1 g_2+\FP_0 g_2)-\FP_0(\xi\cdot\na_\xi g_2)]
\\-\al[A\xi\cdot\na_\xi(\FP_1 g_2+\FP_0 g_2)-\FP_0(A\xi\cdot\na_\xi g_2)]+\beta^\ga L\FP_1g_2
=\beta^\ga\FP_1\left\{(1-\chi_M)\mu^{-\frac{1}{2}}\CK g_1\right\}.
\end{multline}
Then letting $1\leq|\vth|\leq N$, taking inner product of $\pa_\xi^\vth\eqref{f2-eq-p1}$ with $\pa_\xi^\vth\FP_1g_2$,
and applying Lemmas \ref{Ga}, \ref{es-L} and \ref{es-tri} as well as Cauchy-Schwarz's inequality, one has
\begin{align}\label{g2-l2-p2}
\sum\limits_{1\leq|\vth|\leq N}&\frac{d}{dt}\|\pa_\xi^{\vth}\FP_1g_2\|^2
+\la\sum\limits_{1\leq|\vth|\leq N}\beta^\ga\|\pa_\xi^{\vth}\FP_1g_2\|_\nu^2\notag\\
\leq& C\al^2\beta^{-\ga}|[a_2,\Fb_2,c_2]|^2
+C\beta^\ga\|\FP_1g_2\|_\nu^2
+C\sum\limits_{|\vth|\leq N}\beta^\ga\|w^{\ell_2}\pa_\xi^{\vth}g_1\|^2\notag\\
\leq& C\beta^\ga\|\FP_1g_2\|_\nu^2
+C\sum\limits_{|\vth|\leq N}\beta^\ga\|w^{\ell_2}\pa_\xi^{\vth}g_1\|^2,
\end{align}
where \eqref{abc-es} has been used again.

Next, combing \eqref{g2-l2} and \eqref{g2-l2-p2} gives
\begin{align}
\sum\limits_{|\vth|\leq N}\frac{d}{dt}\|\pa_\xi^{\vth}\FP_1g_2\|^2+\la\beta^\ga\sum\limits_{|\vth|\leq N}\|\pa_\xi^{\vth}\FP_1g_2\|^2_\nu
\leq C\beta^\ga\sum\limits_{|\vth|\leq N}\|w^{\ell_2}\pa_\xi^{\vth}g_1\|^2.\label{g2-l2-ly}
\end{align}

\noindent\underline{{\it Step 4. Weighted $H^N_\xi$ estimates for $g_1$.}} In this step, we intend to obtain the estimates of $\|w^{\ell_2}\pa^\vth_\xi g_1\|^2$
with $\ell_{\infty}\geq 2\ell_2\gg 5$ and $|\vth|\leq N$. Recall the following equations for $g_1$
\begin{align}\label{zg1-eq}
\pa_t g_{1}&-\frac{\beta'}{\beta}\na_\xi\cdot(\xi g_{1})
-\al\na_\xi\cdot(A\xi g_{1})-\frac{2\ga\beta'}{\beta}g_1+\beta^\ga \nu g_{1}\notag\\
=&\beta^\ga\chi_M \CK g_{1}-\frac{\beta'}{2\beta} |\xi|^2\sqrt{\mu}g_{2}-\frac{\al}{2}\xi\cdot(A\xi)\sqrt{\mu}g_{2}
-\al^{1-m}\sqrt{\mu}\beta^{2\ga}\pa_t G_1-\al^{2-m}\sqrt{\mu}\beta^{2\ga}\pa_t G_2\notag\\&+\al^{1-m}\beta'\beta^{2\ga-1}\na_\xi\cdot(\xi\sqrt{\mu}G_1)
+\al^{2-m}\beta'\beta^{2\ga-1}\na_\xi\cdot(\xi\sqrt{\mu}G_2)
+\al^{3-m}\beta^{2\ga}\na_\xi\cdot(A\xi\sqrt{\mu}G_2)
\notag\\&+\al^{3-m}\beta^{3\ga} \{Q(\sqrt{\mu}G_1,\sqrt{\mu}G_2)+Q(\sqrt{\mu}G_2,\sqrt{\mu}G_1)\}
+\al^{4-m}\beta^{3\ga}Q(\sqrt{\mu}G_2,\sqrt{\mu}G_2)
\notag\\&+\al^m\beta^{-\ga}Q(\sqrt{\mu}g,\sqrt{\mu}g)
+\al\beta^{\ga}\{Q(\sqrt{\mu}G_1+\al \sqrt{\mu}G_2,\sqrt{\mu}g)+Q(\sqrt{\mu}g,\sqrt{\mu}G_1+\al \sqrt{\mu}G_2)\},
\end{align}
\begin{align}
g_{1}(0,\xi)=G_{R,0}.\notag
\end{align}
Next, taking the inner product of $\pa_\xi^\vth\eqref{zg1-eq}$ and $w^{2\ell_2}\pa_\xi^\vth g_1$ and using Cauchy-Schwarz's inequality, one has
\begin{align}\label{zg1-l2-p1}
\frac{d}{dt}&\|w^{\ell_2}\pa_\xi^\vth g_{1}\|^2+\la\beta^\ga \|w^{\ell_2}\pa_\xi^\vth g_{1}\|^2_\nu\notag\\
\leq& C\left(\al^2+|\frac{\beta'}{\beta}|^2\right)\beta^{-\ga}\sum\limits_{\vth'\leq \vth}\|\pa_\xi^{\vth'} g_2\|^2
+C\al^{2-2m}\beta^{3\ga}\sum\limits_{\vth'\leq \vth}\|\pa_\xi^{\vth'}[\pa_t G_1,\pa_t G_2]\|^2
\notag\\&+C\al^{2-2m}|\beta'|^2\beta^{3\ga-2}\sum\limits_{\vth'\leq \vth}\|\pa_\xi^{\vth'}[G_1,G_2,\na_\xi G_1,\na_\xi G_2]\|^2
+C\al^{6-2m}\beta^{3\ga}\sum\limits_{\vth'\leq \vth}\|\pa_\xi^{\vth'}[G_2,\na_\xi G_2]\|^2
\notag\\&
+\beta^\ga|(\pa_\xi^\vth(\chi_M \CK g_{1}),w^{2\ell_2}\pa_\xi^\vth g_{1})|
+\al^{3-m}\beta^{3\ga}|(\pa_\xi^\vth\{Q(\sqrt{\mu}G_1,\sqrt{\mu}G_2)+Q(\sqrt{\mu}G_2,\sqrt{\mu}G_1)\},w^{2\ell_2}\pa_\xi^{\vth}g_{1})|
\notag\\&+\al^{4-m}\beta^{3\ga}|(\pa_\xi^{\vth}Q(\sqrt{\mu}G_2,\sqrt{\mu}G_2),w^{2\ell_2}\pa_\xi^{\vth}g_{1})|
+\al^m\beta^{-\ga} |(\pa_\xi^{\vth}Q(g_1+\sqrt{\mu}g_2,g_1+\sqrt{\mu}g_2),w^{2\ell_2}\pa_\xi^{\vth}g_{1})|\notag\\
&+\al\beta^\ga|(\pa_\xi^{\vth}\{Q(\sqrt{\mu}G_1+\al \sqrt{\mu}G_2,g_1+\sqrt{\mu}g_2)+Q(g_1+\sqrt{\mu}g_2,\sqrt{\mu}G_1+\al \sqrt{\mu}G_2)\},w^{2\ell_2}\pa_\xi^{\vth}g_{1})|.
\end{align}
We now turn to compute the right hand side of \eqref{zg1-l2-p1} term by term. First of all,
in light of \eqref{bbp}, one has
\begin{align}
\left(\al^2+|\frac{\beta'}{\beta}|^2\right)\beta^{-\ga}\sum\limits_{\vth'\leq \vth}\|\pa_\xi^{\vth'} g_2\|^2\leq
C\al^2\sum\limits_{\vth'\leq \vth}\|\pa_\xi^{\vth'} g_2\|^2
\leq
C\al^2\sum\limits_{\vth'\leq \vth}\|\pa_\xi^{\vth'} \FP_1g_2\|^2+C\al^2\|w^{\ell_2}g_{1}\|^2,\notag
\end{align}
where \eqref{abc-es} has been used in the last inequality,
and \eqref{bbp} together with \eqref{G1-es} and \eqref{G2-es} gives
\begin{align}
\al^{2-2m}|\beta'|^2\beta^{3\ga-2}\sum\limits_{\vth'\leq \vth}\|\pa_\xi^{\vth'}[G_1,G_2,\na_\xi G_1,\na_\xi G_2]\|^2
\leq C\al^{6-2m}.\notag
\end{align}
Moreover, \eqref{G1-es} and \eqref{G2-es} with $k=1$ also imply
\begin{align}
\al^{2-2m}\beta^{3\ga}\sum\limits_{\vth'\leq \vth}\|\pa_\xi^{\vth'}[\pa_t G_1,\pa_t G_2]\|^2\leq C\al^{6-2m},\notag
\end{align}
and
\begin{align}
\al^{6-2m}\beta^{3\ga}\sum\limits_{\vth'\leq \vth}\|\pa_\xi^{\vth'}[G_2,\na_\xi G_2]\|^2\leq C\al^{6-2m}.\notag
\end{align}
Furthermore, employing Proposition \ref{CK-l2-pro} and Cauchy-Schwarz's inequality, we have
\begin{align}
|(\pa_\xi^{\vth}(\chi_M \CK g_{1}),w^{2\ell_2}\pa_\xi^{\vth}g_{1})|\leq C_\eta\{(1+M)^{-\ga}+\varsigma\}\sum\limits_{\vth'\leq \vth}\|w^{\ell_2}\pa_\xi^{\vth'}g_{1}\|^2_\nu+\eta\|w^{\ell_2}\pa_\xi^{\vth}g_{1}\|^2_\nu.\notag
\end{align}
Next, Lemma \ref{es-tri}, \eqref{G1-es}, \eqref{G2-es}  and \eqref{abc-es} as well as the {\it a priori} assumption \eqref{ap-as} give
\begin{align}
\al^{3-m}\beta^{3\ga} &
|(\pa_\xi^{\vth}\{Q(\sqrt{\mu}G_1,\sqrt{\mu}G_2)+Q(\sqrt{\mu}G_2,\sqrt{\mu}G_1)\},w^{2\ell_2}\pa_\xi^{\vth}g_{1})|\notag\\
\leq& C\al^{3-m}\beta^{3\ga}\|w^{\ell_\infty}\pa_\xi^{\vth}g_{1}\|_{L^\infty}
\sum\limits_{\vth'\leq\vth}\|w^{\ell_2}\pa_\xi^{\vth'}G_{1}\|_\nu\|w^{\ell_2}\pa_\xi^{\vth'}G_{2}\|_\nu\notag\\
\leq& C_\eta\al^{6-2m}+\eta\|w^{\ell_\infty}\pa_\xi^{\vth}g_{1}\|_{L^\infty}^2,\notag
\end{align}
\begin{align}
\al^{4-m}\beta^{3\ga}&
|(\pa_\xi^{\vth}Q(\sqrt{\mu}G_2,\sqrt{\mu}G_2),w^{2\ell_2}\pa_\xi^{\vth}g_{1})|\notag\\
\leq& C\al^{4-m}\beta^{3\ga}\|w^{\ell_\infty}\pa_\xi^{\vth}g_{1}\|_{L^\infty}\sum\limits_{\vth'\leq\vth}\|w^{\ell_2}\pa_\xi^{\vth'}G_{2}\|_\nu^2
\leq C\al^{6-2m}+C\al^2\|w^{\ell_\infty}\pa_\xi^{\vth}g_{1}\|_{L^\infty}^2,\notag
\end{align}
\begin{align}
\al^{m}\beta^{-\ga}|(\pa_\xi^{\vth}&Q(g_1+\sqrt{\mu}g_2,g_1+\sqrt{\mu}g_2),w^{2\ell_2}\pa_\xi^{\vth}g_{1})|\notag\\
\leq& C\al^m\|w^{\ell_\infty}\pa_\xi^{\vth}g_{1}\|_{L^\infty}
\left\{\sum\limits_{\vth'\leq\vth}\|w^{\ell_2}\pa_\xi^{\vth'}g_{1}\|_\nu
+\sum\limits_{\vth'\leq\vth}\|\pa_\xi^{\vth'}g_{2}\|\right\}^2
\notag\\
\leq& C\al^m\sum\limits_{\vth'\leq\vth}\|w^{\ell_2}\pa_\xi^{\vth'}g_{1}\|^2_\nu
+C\al^m\sum\limits_{\vth'\leq\vth}\|\pa_\xi^{\vth'}\FP_1g_{2}\|_\nu^2
+C\al^m\|w^{\ell_2}g_{1}\|^2,\notag
\end{align}
and
\begin{align}
\al\beta^{\ga}&|(\pa_\xi^{\vth}\{Q(\sqrt{\mu}G_1+\al \sqrt{\mu}G_2,g_1+\sqrt{\mu}g_2)+Q(g_1+\sqrt{\mu}g_2,\sqrt{\mu}G_1+\al \sqrt{\mu}G_2)\},w^{2\ell_2}\pa_\xi^{\vth}g_{1})|\notag\\
\leq& C\al\beta^{\ga}\|w^{\ell_\infty}\pa_\xi^{\vth}g_{1}\|_{L^\infty} \left\{\sum\limits_{\vth'\leq\vth}\|w^{\ell_2}\pa_\xi^{\vth'}g_{1}\|_\nu
+\sum\limits_{\vth'\leq\vth}\|\pa_\xi^{\vth'}g_{2}\|\right\} \sum\limits_{\vth'\leq\vth}\|w^{\ell_2}\pa_\xi^{\vth'}\{\sqrt{\mu}[G_1,G_{2}]\}\|_\nu\notag\\
\leq& \eta\left\{\sum\limits_{\vth'\leq\vth}\|w^{\ell_2}\pa_\xi^{\vth'}g_{1}\|_\nu^2+\sum\limits_{\vth'\leq\vth}\|\pa_\xi^{\vth'}\FP_1g_{2}\|_\nu^2
+\|w^{\ell_2}g_{1}\|^2\right\}+C_\eta\al^2.\notag
\end{align}
Plugging the above estimates into \eqref{zg1-l2-p1} and adjusting constants, we obtain for $m\in(2,3)$
\begin{multline}\label{zg1-l2-p2}
\frac{d}{dt}\sum\limits_{\vth \leq N}\|w^{\ell_2}\pa_\xi^{\vth}g_{1}\|^2+\la\beta^\ga \sum\limits_{\vth \leq N}\|w^{\ell_2}\pa_\xi^{\vth}g_{1}\|^2_\nu\\
\leq C(\al^2+\eta)\sum\limits_{\vth \leq N}\|\pa_\xi^{\vth}\FP_1g_2\|^2+\eta\|w^{\ell_\infty}\pa_\xi^{\vth}g_{1}\|_{L^\infty}^2
+C\al^{6-2m}.
\end{multline}
Consequently, one gets from \eqref{g2-l2-ly} and \eqref{zg1-l2-p2} that
\begin{align}
\frac{d}{dt}&\sum\limits_{\vth \leq N}\|\pa_\xi^{\vth}\FP_1g_2\|^2+
\frac{d}{dt}\sum\limits_{\vth \leq N}\|w^{\ell_2}\pa_\xi^{\vth}g_{1}\|^2
+\la\beta^\ga \sum\limits_{\vth \leq N}\|w^{\ell_2}\pa_\xi^{\vth}g_{1}\|^2_\nu
+\la\beta^\ga\sum\limits_{\vth \leq N}\|\pa_\xi^{\vth}\FP_1g_2\|^2_\nu\notag\\
\leq &\eta\|w^{\ell_\infty}\pa_\xi^{\vth}g_{1}\|_{L^\infty}^2+C\al^{6-2m},\notag
\end{align}
which together with \eqref{abc-es} further implies \eqref{l2-final}. This then ends the proof of Lemma \ref{g12-l2-lem}.
\end{proof}

\subsection{Proof of Theorem \ref{mth}}
We are ready to complete the proof of Theorem \ref{mth}. In fact, as explained at the beginning of this section, it suffices to prove \eqref{ap-es} under the assumption \eqref{ap-as}. Indeed, combing \eqref{g1k-es2}, \eqref{g2k-es2} and \eqref{l2-final} together, we conclude that
\begin{multline}
\sup\limits_{0\leq s\leq t}\sum\limits_{|\vth|\leq N}\|w^{\ell_\infty}\pa_\xi^{\vth}g_{1}(s)\|_{L^\infty}
+\sup\limits_{0\leq s\leq t}\sum\limits_{|\vth|\leq N}\|w^{\ell_\infty}\pa_\xi^{\vth}g_{2}(s)\|_{L^\infty}\\
\leq C\sum\limits_{|\vth|\leq N}\|w^{\ell_\infty}\pa_\xi^{\vth}G_{R,0}\|_{L^\infty}+C\al^{3-m},\notag
\end{multline}
which in turn makes the assumption \eqref{ap-as} close.
The above estimate further gives
\begin{align}\label{g1-g2-sum-p2}
&\sum\limits_{|\vth|\leq N}\|w^{\ell_\infty}\pa_\xi^{\vth}(\sqrt{\mu}G_{R})(t)\|_{L^\infty}\notag\\
&\leq C\beta^{-2\ga}(t)\left\{\sum\limits_{|\vth|\leq N}\|w^{\ell_\infty}\pa_\xi^{\vth}G_{R,0}\|_{L^\infty}+C\al^{3-m}\right\}
\notag\\ 
&\leq C\beta^{-2\ga}(t)\al^{-m}\sum\limits_{|\vth|\leq N}\|w^{\ell_\infty}\pa_\xi^{\vth}[F_0(v)-(\mu+\sqrt{\mu}\{\al G_1(0,v)+\al^2 G_2(0,v)\})]\|_{L^\infty}\notag\\
&\quad+C\beta^{-2\ga}(t)\al^{3-m},
\end{align}
for any $t\geq0$,
according to $\sqrt{\mu}G_{R}=\beta^{-2\ga}(g_1+\sqrt{\mu}g_2)$
and
$$
G_{R,0}(\xi)=G_{R,0}(v)=\al^{-m}\left\{F_0(v)-[\mu+\sqrt{\mu}\{\al G_1(0,v)+\al^2 G_2(0,v)\}]\right\}.
$$
In addition, from \eqref{F-G}, it follows
\begin{align}\label{F-G-2}
G(t,\xi)=\beta^3F(t,\beta \xi).
\end{align}
Finally, \eqref{sol-decay} follows from \eqref{bet-eq} and \eqref{F-G-2} together with \eqref{g1-g2-sum-p2} by renaming the velocity variable. Moreover,  since $\beta(t)\to \infty$ as $t\to \infty$, it follows  from \eqref{eq.betaR} and \eqref{g1-g2-sum-p2} that 
\begin{equation}\notag
\lim\limits_{t\to\infty}\frac{\beta^\gamma}{1+\ga \varrho_0\al^2 t}=\lim\limits_{t\to\infty}\frac{\beta^{\gamma-1}\beta'}{\varrho_0\al^2}
=1+\lim\limits_{t\to\infty}\left[\frac{\vho_1}{\vho_0}\al\beta^{-\ga}-\frac{\al^{m-1}}{3\vho_0}\int_{\R^3}\xi\cdot A\xi \beta^{\ga} (\sqrt{\mu}G_R)\,d\xi\right]=1,
\end{equation} 
which proves \eqref{sol.beta}. This ends the proof of Theorem \ref{mth}.

\section{Appendix}\label{pre-sec}

In this section, we provide those estimates that have been used in the previous sections. In particular, we give the basic estimates on the linearized operator $L$ as well as the nonlinear operators $\Ga$ and $Q$, and also present a key estimate for the operator $\CK$ in the case of hard potentials.

The following lemma is concerned with the integral operator $K$ given by \eqref{sp.L}, and its proof in case of the hard sphere model $(\ga=1)$ has been given by \cite[Lemma 3, pp.727]{Guo-2010}.
\begin{lemma}\label{Kop}
Let $K$ be defined as \eqref{sp.L}, then it holds that
\begin{align}\label{K-sp}
Kf(\xi)=K_2f(\xi)-K_1f(\xi)=\int_{\R^3}(\Fk_2(\xi,\xi_\ast)-\Fk_1(\xi,\xi_\ast))f(\xi_\ast)\,d\xi_\ast
\end{align}
with
\begin{equation*}
\Fk_1(\xi,\xi_\ast)=\tilde{C}_1|\xi-\xi_\ast|^\ga e^{-\frac{|\xi|^2+|\xi_\ast|^{2}}{4}}, 
\end{equation*}
and
\begin{equation*}
\Fk_2(\xi,\xi_\ast)= \tilde{C}_2|\xi-\xi_\ast|^{-2+\ga}
e^{-\frac{1}{8}|\xi-\xi_\ast|^{2}-\frac{1}{8}\frac{\left||\xi|^{2}-|\xi_\ast|^{2}\right|^{2}}{|\xi-\xi_\ast|^{2}}}. 
\end{equation*}
Here both $\tilde{C}_1$ and $\tilde{C}_2$ are positive constants.

In addition, let
\begin{align}
\Fk(\xi,\xi_\ast)=\Fk_2(\xi,\xi_\ast)-\Fk_1(\xi,\xi_\ast),\
\Fk_w(\xi,\xi_\ast)=w^{\ell}(\xi)\Fk(\xi,\xi_\ast)w^{-\ell}(\xi_\ast)\notag
\end{align}
with  $\ell\geq0$,
then it also holds that
\begin{equation}
\int_{\R^3} \Fk_w(\xi,\xi_\ast)e^{\frac{\varepsilon|\xi-\xi_\ast|^2}{8}}dv_\ast\leq \frac{C}{1+|\xi|},\label{K-decay}
\end{equation}
for $\varepsilon=0$ or any $\varepsilon> 0$ small enough.

Moreover, for any multi-indices $\vth$ and any $\ell\geq0$ it holds that
\begin{align}\label{K-wif}
|w^\ell\pa_\xi^\vth (Kf)|\leq C\sum\limits_{\vth'\leq\vth}\|w^\ell \pa_\xi^{\vth'}f\|_{L^\infty}.
\end{align}

\end{lemma}
\begin{proof}
We prove \eqref{K-wif} only, since the other statements in the Lemma is well known.
By \eqref{K-sp} and a change of variables $\xi_\ast-\xi\rightarrow u$, we have
\begin{align}
\pa_\xi^\vth (K_1f)=&\tilde{C}_1\sum\limits_{\vth'\leq\vth}C_\vth^{\vth'}\int_{\R^3}|u|^\ga\pa_\xi^{\vth'}
\left\{e^{-\frac{|\xi|^2+|u+\xi|^{2}}{4}}\right\}\pa_\xi^{\vth-\vth'}f(u+\xi)du\notag\\
=&\tilde{C}_1\sum\limits_{\vth'\leq\vth}C_\vth^{\vth'}\int_{\R^3}
\tilde{\Fk}_1(\xi,\xi_\ast)(\pa_\xi^{\vth-\vth'}f)(\xi_\ast)du,\notag
\end{align}
and
\begin{align}
\pa_\xi^\vth (K_2f)=&\tilde{C}_1\sum\limits_{\vth'\leq\vth}C_\vth^{\vth'}\int_{\R^3}|u|^{-2+\ga}\pa_\xi^{\vth'}
\left\{e^{-\frac{1}{8}|u|^{2}-\frac{1}{8}\frac{\left||\xi|^{2}-|u+\xi|^{2}\right|^{2}}{|u|^{2}}}\right\}
\pa_\xi^{\vth-\vth'}f(u+\xi)du\notag\\
=&\tilde{C}_1\sum\limits_{\vth'\leq\vth}C_\vth^{\vth'}\int_{\R^3}
\tilde{\Fk}_2(\xi,\xi_\ast)(\pa_\xi^{\vth-\vth'}f)(\xi_\ast)du,\notag
\end{align}
where
\begin{align}
\tilde{\Fk}_1(\xi,\xi_\ast)=|\xi-\xi_\ast|^\ga\left(\pa_\xi^{\vth'}
\left\{e^{-\frac{|\xi|^2+|u+\xi|^{2}}{4}}\right\}\right)\bigg|_{u=\xi-\xi_\ast},\notag
\end{align}
and
\begin{align}
\tilde{\Fk}_2(\xi,\xi_\ast)=|\xi-\xi_\ast|^{-2+\ga}\left(\pa_\xi^{\vth'}
\left\{e^{-\frac{1}{8}|u|^{2}-\frac{1}{8}\frac{\left||\xi|^{2}-|u+\xi|^{2}\right|^{2}}{|u|^{2}}}\right\}\right)\bigg|_{u=\xi-\xi_\ast}.\notag
\end{align}
Furthermore, it is direct to see
\begin{align}
|\tilde{\Fk}_1(\xi,\xi_\ast)|\leq C(\vth)|\xi-\xi_\ast|^\ga e^{-\frac{|\xi|^2+|\xi_\ast|^{2}}{8}},\notag
\end{align}
and
\begin{align}
|\tilde{\Fk}_2(\xi,\xi_\ast)|\leq C(\vth)|\xi-\xi_\ast|^{-2+\ga}
e^{-\frac{1}{16}|\xi-\xi_\ast|^{2}-\frac{1}{16}\frac{\left||\xi|^{2}-|\xi_\ast|^{2}\right|^{2}}{|\xi-\xi_\ast|^{2}}}.\notag
\end{align}
Then performing the similar calculation as for obtaining \eqref{K-decay}, one sees that \eqref{K-wif} is true. This completes the proof of Lemma \ref{Kop}.
\end{proof}

For the weighted velocity derivative estimates on the nonlinear operator $\Ga$, one has the following result.

\begin{lemma}\label{Ga}
Let $0\leq \ga\leq 1$ and $\ta\in[0,1]$. For any $p\in[1,+\infty]$ and any $\ell\geq0$, it holds that
\begin{align}
\|w^\ell\nu^{-\ta}\pa_\xi^\vth\Ga(f,g)\|_{L_\xi^p}\leq C\sum\limits_{\vth'+\vth''\leq \vth}
\left\{\|w^\ell\nu^{1-\ta}\pa_\xi^{\vth'}f\|_{L_\xi^p}\|\pa_\xi^{\vth''}g\|_{L_\xi^p}+\|\pa_\xi^{\vth'}f\|_{L_\xi^p}
\|w^\ell\nu^{1-\ta}\pa^{\vth''}_\xi g\|_{L_\xi^p}\right\}.\notag
\end{align}

\end{lemma}

The following lemma is concerned with  coercivity estimates for the linear collision operator $L$.

\begin{lemma}\label{es-L}
Let $0\leq \ga\leq1$, then there is a constant $\de_0>0$ such that
\begin{align}
\lag Lf,f\rag=\lag L\FP_1f,\FP_1f\rag\geq\de_0\|\FP_1f\|_\nu^2,\notag
\end{align}
where $\|\cdot\|_\nu=\|\nu^{\frac{1}{2}}\cdot\|.$
Moreover, there are constants $\de_1>0$ and $C>0$ such that for $|\vth|>0$
\begin{align}
\lag \pa_{\xi}^\vth L f,\pa_\xi^\vth  f\rag\geq\de_1\|\pa_\xi^\vth f\|_\nu^2-C\|f\|^2.\notag
\end{align}
\end{lemma}


Next, the following lemma which was proved in \cite[Proposition 3.1, pp.13]{DL-arma-2021} gives the $L^\infty$ estimates of the solutions in the case of Maxwell molecule model.
\begin{lemma}\label{CK}
Let $\ga=0$ and $\CK$ be given by \eqref{sp.cL}, then for any nonnegative integer $|\vth|\geq 0$, there is $C>0$ such that for any arbitrarily large $\ell>0$, there is $M=M(\ell)>0$ such that it holds that
\begin{align}
\chi_M w^{\ell}|\pa_{\xi}^\vth(\CK f)|\leq \frac{C}{\ell} \sum\limits_{0\leq \vth'\leq \vth}\|w^{\ell}\pa_{\xi}^{\vth'}f\|_{L^\infty}.\notag
\end{align}
In particular, one can choose $M=\ell^2$.
\end{lemma}

In the case of $0<\ga\leq1$, the following lemma with $\vth=0$ which can be found in
\cite[Proposition 3.1, pp.397]{AEP-87} enables us to gain the smallness property of $\CK$ at large velocity.
\begin{lemma}\label{g-ck}
Let $0\leq\ga\leq1$, $\ell>4$ and for any multi-indices $\vth\geq0$, then there exists a function $\varsigma(\ell)$ which satisfies $\varsigma(\ell)\rightarrow0$ as $\ell\rightarrow+\infty$
such that
\begin{align}
w^\ell\{|\pa_\xi^\vth &Q_{\rm{loss}}(f,g)|+|\pa_\xi^\vth Q_{\rm{gain}}(f,g)|+|\pa_\xi^\vth Q_{\rm{gain}}(g,f)|\}\notag\\
\leq& \sum\limits_{\vth'\leq\vth}
\|w^\ell \pa_\xi^{\vth'}f\|_{L^\infty}\{C(\ell)\|w^{\ell+\ga/2}\pa_\xi^{\vth'}g\|_{L^\infty}
+\varsigma(\ell)\|w^{3}\pa_\xi^{\vth'}g\|_{L^\infty}(1+|\xi|)^\ga\},\label{Q-lf2}
\end{align}
where $Q_{\textrm{loss}}$ denotes the negative part of $Q$ in \eqref{Q-op}.
\end{lemma}
\begin{proof}
Since the case that $\vth=0$ of \eqref{Q-lf2} has been given in \cite[Proposition 3.1, pp.397]{AEP-87}, here we focus on the case of $\vth>0$.
As a matter of fact, for $\vth>0$, as \eqref{CK-de-exp-p1}, by a change of variables $\xi_\ast-\xi\rightarrow u$, we have
\begin{align}
\pa^\vth_\xi Q(f,g)=&\sum\limits_{\vth'\leq\vth}C_\vth^{\vth'}\int_{\R^3\times\S^2}B_0|u|^\ga (\pa^{\vth-\vth'}_\xi f)(\xi+u_{\perp})
(\pa^{\vth'}_\xi g)(\xi+u_{\parallel}) d\om du
\notag\\&-\sum\limits_{\vth'\leq\vth}
\int_{\R^3\times\S^2}B_0|u|^\ga (\pa^{\vth-\vth'}_\xi f)(u+\xi)(\pa^{\vth'}_\xi g)(\xi)d\om du,\notag
\end{align}
which further equals 
\begin{align}
\sum\limits_{\vth'\leq\vth}&C_\vth^{\vth'}\int_{\R^3\times\S^2}B_0|\xi-\xi_\ast|^\ga (\pa^{\vth-\vth'}_\xi f)(\xi_\ast')
(\pa^{\vth'}_\xi g)(\xi') d\om d\xi_\ast
\notag\\&-\sum\limits_{\vth'\leq\vth}
\int_{\R^3\times\S^2}B_0|\xi-\xi_\ast|^\ga (\pa^{\vth-\vth'}_\xi f)(\xi_\ast)(\pa^{\vth'}_\xi g)(\xi)d\om d\xi_\ast,\notag
\end{align}
by changing the variables back. Then performing the same calculations as for the case $\vth=0$ in \cite[Proposition 3.1, pp.397]{AEP-87}, one sees that \eqref{Q-lf2} holds true. This completes the proof of Lemma \ref{g-ck}.
\end{proof}

The following result is a direct consequence of Lemma \ref{g-ck}.

\begin{lemma}\label{g-ck-lem}
Let $0<\ga\leq1$, then there is a constant $C>0$ such that for any arbitrarily large $\ell>4$ and any multi-indices $\vth\geq0$, there are sufficiently large $M=M(\ell)>0$ and suitably small $\varsigma=\varsigma(\ell)>0$ such that it holds that
\begin{align}\label{CK2}
\chi_M\nu^{-1} w^{\ell}|\pa_\xi^\vth \CK f|\leq C \sum\limits_{\vth'\leq\vth}\{(1+M)^{-\ga}+\varsigma\}\|w^{\ell} \pa_\xi^{\vth'} f\|_{L^\infty}.
\end{align}
\end{lemma}

The following Lemma concerning the polynomial weighted estimates on the collision operator $Q$ can be verified by using a parallel argument as for obtaining \cite[Proposition 3.1, pp.397]{AEP-87}.

\begin{lemma}\label{op.es.lem}Let $\ell>4$ and $\ga\geq0$, then it holds that
\begin{equation}\notag
|w^{\ell} \nu^{-1}\pa_\xi^\vth Q_{gain}(F_1,F_2)|,\ |w^{\ell} \nu^{-1}\pa_\xi^\vth Q_{loss}(F_1,F_2)|\leq C\sum\limits_{\vth'\leq\vth}\|w^\ell \pa_\xi^{\vth'} F_1\|_{L^\infty}\|w^\ell \pa_\xi^{\vth-\vth'}F_2\|_{L^\infty}.
\end{equation}
\end{lemma}

Finally, we give the following crucial estimates on the inner product involving $Q(f,g)$.
\begin{lemma}\label{es-tri}
Let $\ell_{\infty}> 2\ell_2\gg 4$, then it holds that
\begin{equation}\label{es-tri-p1}
|(\pa_\xi^\vth Q(f,g),w^{2\ell_2}  h)|\leq
C\sum\limits_{\vth'+\vth''\leq \vth}\|w^{\ell_\infty}h\|_{L^\infty}\|\nu^{\frac{1}{2}}w^{\ell_2} \pa_\xi^{\vth'}f\|\|\nu^{\frac{1}{2}}w^{\ell_2}\pa_\xi^{\vth''}g\|.
\end{equation}
In particular, it holds that
\begin{equation}\label{es-tri-p2}
|( \pa_\xi^\vth[(1-\chi_M) Q(f,g)], h)|\leq C\sum\limits_{\vth'+\vth''\leq \vth}\|h\|\|w^{\ell_2}\pa_\xi^{\vth'}f\|\|w^{\ell_2}\pa_\xi^{\vth''}g\|.
\end{equation}
\end{lemma}
\begin{proof}
Following the similar proof of Proposition \ref{CK-l2-pro}, we first prove that both \eqref{es-tri-p1} and \eqref{es-tri-p2} are true in the case of $\vth=0$.
Notice that
\begin{align}
|( Q(f,g),w^{2\ell_2}  h)|\leq& \int_{\R^6}\int_{\S^2}B_0|\xi_\ast-\xi|^\ga|f(\xi_\ast')g(\xi')(w^{2\ell_2}h)(\xi)|dwd\xi_\ast d\xi
\notag\\&+\int_{\R^6}\int_{\S^2}B_0|\xi_\ast-\xi|^\ga|f(\xi_\ast)g(\xi)(w^{2\ell_2}h)(\xi)|dwd\xi_\ast d\xi=:\CQ_1+\CQ_2.\label{Q-tri-exp}
\end{align}
Since $\ell_{\infty}> 2\ell_2\gg4$, we have by H\"older's inequality that
\begin{align}
\CQ_2\leq C\|w^{\ell_\infty}h\|_{L^\infty}\int_{\R^3}\lag \xi\rag^{\ga}w^{-2\ell_2}d\xi
\|\nu^{\frac{1}{2}}w^{\ell_2}f\|\|\nu^{\frac{1}{2}}w^{\ell_2}g\|\leq C\|w^{\ell_\infty}h\|_{L^\infty}\|\nu^{\frac{1}{2}}w^{\ell_2}f\|\|\nu^{\frac{1}{2}}w^{\ell_2}g\|,\notag
\end{align}
where the fact that $|\xi_\ast-\xi|^\ga\leq\lag\xi\rag^\ga\lag\xi_\ast\rag^\ga$ has been used.

Likewise, for $\CQ_1$,  H\"older's inequality and a change of variables $(\xi,\xi_\ast)\rightarrow(\xi',\xi'_\ast)$
give
\begin{align}
\CQ_1\leq& C\|w^{\ell_\infty}h\|_{L^\infty}\int_{\R^3}
\left(\int_{\R^3\times\S^2}|\xi'_\ast-\xi'|^{2\ga}w^{-2\ell_2}(\xi')w^{-2\ell_2}(\xi'_\ast)d\xi_\ast d\om\right)^{\frac{1}{2}}
\notag\\&\qquad\qquad\qquad\times\left(\int_{\R^3\times\S^2}|(w^{\ell_2}f)(\xi_\ast')(w^{\ell_2}g)(\xi')|^2d\xi_\ast d\om \right)^{\frac{1}{2}}d\xi\notag\\
\leq& C\|w^{\ell_\infty}h\|_{L^\infty}\left(\int_{\R^6\times\S^2}|\xi_\ast-\xi|^{2\ga}w^{-2\ell_2}(\xi)w^{-2\ell_2}(\xi_\ast)d\xi_\ast d\xi d\om\right)^{\frac{1}{2}}
\notag\\&\qquad\qquad\qquad\times\left(\int_{\R^6\times\S^2}|(w^{\ell_2}f)(\xi_\ast')(w^{\ell_2}g)(\xi')|^2d\xi_\ast d\xi d\om\right)^{\frac{1}{2}}
\notag\\
\leq& C\|w^{\ell_\infty}h\|_{L^\infty}\|w^{\ell_2}f\|\|w^{\ell_2}g\|.\notag
\end{align}
We now turn to prove \eqref{es-tri-p2}. As \eqref{Q-tri-exp}, it follows that
\begin{multline}
|( (1-\chi_M)Q(f,g), h)|\leq \int_{\R^3}\int_{\R^3}\int_{\S^2}B_0|\xi_\ast-\xi|^\ga|f(\xi_\ast')g(\xi')((1-\chi_M)h)(\xi)|dwd\xi_\ast d\xi
\notag\\+\int_{\R^3}\int_{\R^3}\int_{\S^2}B_0|\xi_\ast-\xi|^\ga|f(\xi_\ast)g(\xi)((1-\chi_M)h)(\xi)|dwd\xi_\ast d\xi=:\tilde{\CQ}_1+\tilde{\CQ}_2.\notag
\end{multline}
Notice $\ell_2\geq2$. It then follows
\begin{align}
\int_{\R^3}|\xi_\ast-\xi|^{2\ga}w^{-2\ell_2}(\xi_\ast)d\xi_\ast\leq C\lag\xi\rag^{2\ga},\label{ga-de}
\end{align}
which together with H\"older's inequality gives
\begin{align}
\tilde{\CQ_2}\leq& C\|w^{\ell_2}f\|\int_{\R^3}\left(\int_{\R^3}|\xi_\ast-\xi|^{2\ga}w^{-2\ell_2}(\xi_\ast)d\xi_\ast\right)^{\frac{1}{2}}
|[g(1-\chi_M)h](\xi)|d\xi\notag\\
\leq& C\|w^{\ell_2}f\|\|g\|\|h\|.\notag
\end{align}
For $\tilde{\CQ}_1$, by using H\"older's inequality and a change of variables $(\xi,\xi_\ast)\rightarrow(\xi',\xi'_\ast)$, one has
\begin{align}
\tilde{\CQ}_1\leq& C\int_{\R^3}
\left(\int_{\R^3\times\S^2}|\xi'_\ast-\xi'|^{2\ga}w^{-2\ell_2}(\xi')w^{-2\ell_2}(\xi'_\ast)d\xi_\ast d\om\right)^{\frac{1}{2}}
\notag\\&\qquad\qquad\qquad\times\left(\int_{\R^3\times\S^2}|(w^{\ell_2}f)(\xi_\ast')(w^{\ell_2}g)(\xi')|^2d\xi_\ast d\om \right)^{\frac{1}{2}}((1-\chi_M)h)(\xi)d\xi\notag\\
\leq& C\|h\|\left(\int_{\R^6\times\S^2}|(w^{\ell_2}f)(\xi_\ast')(w^{\ell_2}g)(\xi')|^2d\xi_\ast d\xi d\om\right)^{\frac{1}{2}}
\notag\\
\leq& C\|h\|\|w^{\ell_2}f\|\|w^{\ell_2}g\|,\notag
\end{align}
where \eqref{ga-de} has been used again.

Next, if $\vth>0$, one has as for obtaining \eqref{CK-de-exp-p2} that
\begin{align}
\pa^\vth_\xi Q(f,g)=&\sum\limits_{\vth'+\vth''\leq\vth}C_\vth^{\vth',\vth''}
\int_{\R^3\times\S^2}B_0|\xi_\ast-\xi|^\ga (\pa^{\vth'}_\xi f)(\xi_\ast')
(\pa^{\vth''}_\xi g)(\xi') d\om d\xi_\ast
\notag\\&\quad+\sum\limits_{\vth'+\vth''\leq\vth}C_\vth^{\vth',\vth''}
\int_{\R^3\times\S^2}B_0|\xi_\ast-\xi|^\ga
(\pa^{\vth'}_\xi f)(\xi')(\pa^{\vth''}_\xi g)(\xi_\ast')d\om d\xi_\ast
\notag\\&\quad-\sum\limits_{\vth'+\vth''\leq\vth}
\int_{\R^3\times\S^2}B_0|\xi_\ast-\xi|^\ga (\pa^{\vth'}_\xi f)(\xi_\ast)(\pa^{\vth''}_\xi g)(\xi)d\om d\xi_\ast
\notag\\&\quad-\sum\limits_{\vth'+\vth''\leq\vth}
\int_{\R^3\times\S^2}B_0|\xi_\ast-\xi|^\ga (\pa^{\vth'}_\xi f)(\xi)(\pa^{\vth''}_\xi g)(\xi_\ast)d\om d\xi_\ast,\notag
\end{align}
and
\begin{align}
&\pa^\vth_\xi [(1-\chi_M) Q(f,g)]\notag\\
&=\sum\limits_{\vth'+\vth''\leq\vth}C_\vth^{\vth',\vth''}
\pa^{\vth-\vth'-\vth''}_\xi(1-\chi_M)\int_{\R^3\times\S^2}B_0|\xi_\ast-\xi|^\ga (\pa^{\vth'}_\xi f)(\xi_\ast')
(\pa^{\vth''}_\xi g)(\xi') d\om d\xi_\ast
\notag\\&\quad+\sum\limits_{\vth'+\vth''\leq\vth}C_\vth^{\vth',\vth''}
\pa^{\vth-\vth'-\vth''}_\xi(1-\chi_M)\int_{\R^3\times\S^2}B_0|\xi_\ast-\xi|^\ga
(\pa^{\vth'}_\xi f)(\xi')(\pa^{\vth''}_\xi g)(\xi_\ast')d\om d\xi_\ast
\notag\\&\quad-\sum\limits_{\vth'+\vth''\leq\vth}C_\vth^{\vth',\vth''}\pa^{\vth-\vth'-\vth''}_\xi(1-\chi_M)
\int_{\R^3\times\S^2}B_0|\xi_\ast-\xi|^\ga (\pa^{\vth'}_\xi f)(\xi_\ast)(\pa^{\vth''}_\xi g)(\xi)d\om d\xi_\ast
\notag\\&\quad-\sum\limits_{\vth'+\vth''\leq\vth}C_\vth^{\vth',\vth''}\pa^{\vth-\vth'-\vth''}_\xi(1-\chi_M)
\int_{\R^3\times\S^2}B_0|\xi_\ast-\xi|^\ga (\pa^{\vth'}_\xi f)(\xi)(\pa^{\vth''}_\xi g)(\xi_\ast)d\om d\xi_\ast.\notag
\end{align}
Then, performing the similar calculations as in the case of $\vth=0$, we see that both \eqref{es-tri-p1} and \eqref{es-tri-p2} are also valid for $\vth>0$. 
The proof of Lemma \ref{es-tri} is finished.
\end{proof}

\noindent {\bf Acknowledgements:}
RJD was partially supported by the General Research Fund (Project No.~14303321) from RGC of Hong Kong and a Direct Grant from CUHK. SQL was supported by grants from the National Natural Science Foundation of China (contracts: 11971201 and 11731008). This work was also partially supported by the Fundamental Research Funds for the Central Universities.



\begin{thebibliography}{99}


\bibitem{AEP-87}
L. Arkeryd, R. Esposito and M. Pulvirenti, The Boltzmann equation for weakly inhomogeneous data. {\it Comm. Math. Phys.}
{\bf 111}  (1987),  no. 3, 393--407.

\bibitem{BM-15}
J. Bedrossian and N. Masmoudi, Inviscid damping and the asymptotic stability of planar shear flows in the 2D Euler equations. {\it Publ. Math. Inst. Hautes \'Etudes Sci.} {\bf 122} (2015), 195--300.

\bibitem{BMV-16}
J. Bedrossian, N. Masmoudi and V. Vicol, Enhanced dissipation and inviscid damping in the inviscid limit of the Navier-Stokes equations near the two dimensional Couette flow. {\it Arch. Ration. Mech. Anal.} {\bf 219} (2016), no. 3, 1087--1159.


\bibitem{BGM-17}
J. Bedrossian, P. Germain and N. Masmoudi, On the stability threshold for the 3D Couette flow in Sobolev regularity. {\it Ann. Math.} {\bf 185} (2017), no. 2, 541--608.


\bibitem{BGM-BAMS}
J. Bedrossian, P. Germain and N. Masmoudi, Stability of the Couette flow at high Reynolds numbers in two dimensions and three dimensions. {\it Bull. Amer. Math. Soc. (N.S.)} {\bf 56} (2019), no. 3, 373--414.


\bibitem{Bo21}
A.~V. Bobylev, On a class of self-similar solutions of the Boltzmann equation. 2021, arXiv:2111.00872.

%
\bibitem {Bo75}
A.~V. Bobylev, The method of the Fourier transform in the theory of the Boltzmann equation for Maxwell molecules. (Russian)
\textit{Dokl. Akad. Nauk. SSSR} \textbf{225} (1975), no. 6, 1041--1044.



\bibitem{Bo88}
A.~V. Bobylev, The theory of the nonlinear spatially uniform Boltzmann equation for Maxwellian molecules.
\textit{Sov. Scient. Rev. C} \textbf{7} (1988), 111--233.


\bibitem {BC02b}
A.~V. Bobylev and C. Cercignani, Exact eternal solutions of the Boltzmann equation.
\textit{J. Stat. Phys.} \textbf{106} (2002), no. 5-6, 1019--1039.




\bibitem {BC02a}
A.~V. Bobylev and C. Cercignani, Self-similar solutions of the Boltzmann equation and their applications.
\textit{J. Stat. Phys.} \textbf{106} (2002),  no. 5-6, 1039--1071.



\bibitem {BC03}
A.~V. Bobylev and C. Cercignani, Self-similar asymptotics for the Boltzmann equation with inelastic and elastic interactions. \textit{J. Stat. Phys.} \textbf{110} (2003), no. 1-2, 335--375.




%
%

%
\bibitem {BCS}
A.~V. Bobylev, G.~L. Caraffini and  G. Spiga, On group
invariant solutions of the Boltzmann equation.
\textit{J. Math. Phys.} \textbf{37} (1996), 2787--2795.
%



\bibitem{BNV-2019}
A. Bobylev, A. Nota and J. J. L. Vel\'azquez, Self-similar asymptotics for a modified Maxwell-Boltzmann equation in systems subject to deformations. arXiv:1912.12498.






\bibitem{CK10}
M. Cannone and G. Karch, Infinite energy solutions to the homogeneous Boltzmann equation. {\it Comm. Pure Appl. Math.} {\bf 63} (2010), no. 6, 747--778.

\bibitem{CK13}
M. Cannone and G. Karch, On self-similar solutions to the homogeneous Boltzmann equation. {\it Kinet. Relat. Models} {\bf 6} (2013), no. 4, 801--808.


\bibitem{CerBook}
C. Cercignani, {\it The Boltzmann Equation and Its Applications}. Applied Mathematical Sciences, 67. Springer-Verlag, New York, 1988.

\bibitem{Cer89}
C. Cercignani, Existence of homoenergetic affine flows
for the Boltzmann equation. \textit{Arch. Rat. Mech. Anal.}
\textbf{105} (1989), no. 4, 377--387.

\bibitem{Cer00}
C. Cercignani, Shear Flow of a Granular Material.
\textit{J. Stat. Phys.} \textbf{102} (2001), no. 5, 1407--1415.

\bibitem{Cer02}
C. Cercignani, The Boltzmann equation approach to the shear
flow of a granular material. \textit{Philosophical Trans. Royal Society.} \textbf{360} (2002), 437--451.



\bibitem{DL-arma-2021}
R.-J. Duan and S.-Q. Liu, The Boltzmann equation for uniform shear flow. {\it Arch. Ration. Mech. Anal.} {\bf 242} (2021), no. 3, 1947--2002.

\bibitem{DL-2022}
R.-J. Duan and S.-Q. Liu, On smooth solutions to the thermostated Boltzmann equation with deformation. {\it Commun. Math. Anal. Appl.}
{\bf 1} (2022), no.~1, 152--212.


\bibitem{DLY-2021}
R.-J. Duan, S.-Q. Liu and T. Yang, The Boltzmann equation for plane Couette flow. To appear in {\it Journal of the European Mathematical Society} (2022), arXiv:2107.02458.



\bibitem{ELM-94}
R. Esposito, J.L. Lebowitz and R. Marra, Hydrodynamic limit of the stationary Boltzmann equation in a slab. {\it Comm. Math. Phys.} {\bf 160} (1994), no. 1, 49--80.

\bibitem{ELM-95}
R. Esposito, J.L. Lebowitz and R. Marra, The Navier-Stokes limit of stationary solutions of the nonlinear Boltzmann equation. {\it J. Statist. Phys.} {\bf 78} (1995), 389--412.


\bibitem{G1}
V.~S. Galkin, On a class of solutions of Grad's moment
equation. \textit{PMM} \textbf{22}(3) (1958), 386--389. (Russian version
\textit{PMM} \textbf{20} (1956), 445--446.
%
%






\bibitem{GaSa}
V. Garz\'o and A. Santos, {\it Kinetic Theory of Gases in Shear Flows. Nonlinear transport.}  Fundamental Theories of Physics, {\bf 131} (2003). Kluwer Academic Publishers, Dordrecht.
%

\bibitem{GMM}
M.~P. Gualdani, S. Mischler and C. Mouhot, Factorization of non-symmetric operators and exponential H-theorem. {\it M\'em. Soc. Math. Fr. (N.S.)} {\bf 153} (2017), 137 pp.





\bibitem{Guo-2010}
Y. Guo,
Decay and continuity of the Boltzmann equation in bounded domains.
{\it Arch. Ration. Mech. Anal.} {\bf 197} (2010), no. 3, 713--809.






%


\bibitem{LG-fourier}
L. Grafakos, {\it Classical Fourier analysis}. Second edition. Graduate Texts in Mathematics, 249. Springer, New York, 2008. xvi+489 pp. ISBN: 978-0-387-09431-1



\bibitem{IJ}
A. D. Ionescu and H. Jia, Inviscid damping near the Couette flow in a channel. {\it Comm. Math. Phys.} {\bf 374} (2020), no.~3, 2015--2096.


\bibitem{JNV-ARMA}
R. D. James, A. Nota and J. J. L. Vel\'azquez, Self-similar profiles for homoenergetic solutions of the Boltzmann equation: particle velocity distribution and entropy. {\it Arch. Rat. Mech. Anal.} {\bf 231} (2019), no. 2, 787--843.

\bibitem{JNV-JNS}
R. D. James, A. Nota and J. J. L. Vel\'azquez, Long-time asymptotics for
homoenergetic solutions of the Boltzmann equation: collision-dominated case. {\it J. Nonlinear Sci.} {\bf 29} (2019), no. 5, 1943--1973.

\bibitem{JNV19}
R. D. James, A. Nota and J. J. L. Vel\'azquez, Long-time asymptotics for homoenergetic solutions of the Boltzmann equation: hyperbolic-dominated case.  {\it Nonlinearity} {\bf 33} (2020), no. 8, 3781--3815.

\bibitem{Ke1}
B. Kepka, Self-similar profiles for homoenergetic solutions of the Boltzmann equation for noncutoff Maxwell molecules. 2021, arXiv:2103.10744.


\bibitem{Ke2}
B. Kepka, Long time behavior for homoenergetic solutions in the collision dominated regime for hard potentials. 2022, arXiv:2202.09074.



\bibitem{Ko}
M.~N. Kogan, {\it Rarefied Gas Dynamics}. Plenum Press, New York (1969).



%


\bibitem {MZ}
N. Masmoudi and W. Zhao, Nonlinear inviscid damping for a class of monotone shear flows in finite channel. 2020, arXiv:2001.08564


\bibitem {MT}
K. Matthies and F. Theil, Rescaled objective solutions of Fokker-Planck and Boltzmann equations. {\it  SIAM J. Math. Anal.} {\bf 51} (2019), no. 2, 1321--1348.


\bibitem{MYZ}
Y. Morimoto, T. Yang and H.-J. Zhao, Convergence to self-similar solutions for the homogeneous Boltzmann equation. {\it J. Eur. Math. Soc.} {\bf 19} (2017), no. 8, 2241--2267.

\bibitem{NV}
A. Nota and J. J. L. Vel\'azquez, Homoenergetic solutions of the Boltzmann equation: the case of
simple-shear deformations. {\it Mathematics in Engineering} {\bf 5} (2022), no.~1, 1--25.


\bibitem{ScHe}
P.J. Schmid and D.S. Henningson, {\it Stability and Transition in Shear Flows}. Applied Mathematical Sciences, vol. {\bf 142} (2001). Springer, New York.

\bibitem{Sone07}
Y. Sone, {\it Molecular Gas Dynamics: Theory, Techniques, and Applications}. Birkhauser, Boston, 2007.

\bibitem{T}
C. Truesdell, On the pressures and flux of energy in a gas according to Maxwell's kinetic theory II. \textit{J. Rat. Mech. Anal.} \textbf{5} (1956), 55--128.
%



\bibitem{TM}
C. Truesdell and R.G. Muncaster, {\it Fundamentals of Maxwell's Kinetic Theory of a Simple Monatomic Gas}. Academic Press, New York, 1980.



\end{thebibliography}
\end{document}